\documentclass[10pt]{article}
\usepackage{amsmath}
\usepackage{amsfonts}
\usepackage{amssymb}
\usepackage{authblk}
\usepackage{style_one_counter,ulem}
\usepackage{soul}
\title{A large deviation principle for block models}

\author[1]{Christian Borgs \thanks{Berkeley AI Research Group, Department of Electrical Engineering and Computer Science. Email: borgs@berkeley.edu.}}
\author[1]{Jennifer Chayes \thanks{Department of Electrical Engineering and Computer Science, Department of Mathematics, Department of Statistics, School of Information. Email: jchayes@berkeley.edu.}}
\author[2]{Julia Gaudio \thanks{Department of Industrial Engineering and Management Sciences. Email: julia.gaudio@northwestern.edu.}}
\author[3]{Samantha Petti \thanks{Department of Mathematics, Tufts University. Email: spetti@tufts.edu.}}
\author[4]{Subhabrata Sen \thanks{Department of Statistics. Email: subhabratasen@fas.harvard.edu.}} 
\affil[1]{University of California, Berkeley}
\affil[2]{Northwestern University}
\affil[3]{Tufts University}
\affil[4]{Harvard University}

\begin{document}

\maketitle
\abstract{We initiate a study of large deviations for  block model random graphs in the dense regime. Following \cite{Chatterjee2011}, we establish an LDP for dense block models, viewed as random graphons.  As an application of our result, we study upper tail large deviations for homomorphism densities of regular graphs. We identify the existence of a ``symmetric'' phase, where the graph, conditioned on the rare event, looks like a block model with the same block sizes as the generating graphon. In specific examples, we also identify the existence of a ``symmetry breaking'' regime, where the conditional structure is not a block model with compatible dimensions. This identifies a ``reentrant phase transition'' phenomenon for this problem---analogous to one established for \ER random graphs \cite{chatterjee2010applications, Chatterjee2011}. Finally, extending the analysis of \cite{Lubetzky2015}, we identify the precise boundary between the symmetry and symmetry breaking regimes for homomorphism densities of regular graphs and the operator norm on \ER bipartite graphs.}

\vspace{5pt}
\noindent 
{\bf Keywords:}  large deviation, stochastic block model, symmetry/symmetry breaking, bipartite \ER graph.

\tableofcontents

\section{Introduction}
The study of large deviation problems on random graphs has a long and rich history in Probability and Combinatorics. Research in this area is motivated by the following fundamental question:
\emph{What is the structure of a random graph, conditioned on a rare event?}

In a seminal paper, Chatterjee and Varadhan  \cite{Chatterjee2011} formalized this question by combining the theory of graph limits \cite{Lovasz2012} with classical Large Deviations theory \cite{amirdembooferzeitouni2010}, and established a Large Deviation Principle (LDP) for the Erd\H{o}s-R\'{e}nyi binomial random graph $G(n,p)$. This is the simplest random graph model, constructed by adding edges independently among $n$ vertices with probability $p$.
As an application of this LDP, Chatterjee and Varadhan \cite{Chatterjee2011} examined upper tail large deviations for regular subgraph counts. The homomorphism density $t(H,G)$ of a graph $H$ on $v$ vertices measures the probability that $H$ appears on $v$ randomly chosen vertices of a graph $G$ (see \Cref{hom-density}). Let $H$ be a $d$-regular graph, and for notational convenience, define the event $\mathcal{E}_{\delta}=\{t(H,G) > (1+ \delta) \mathbb{E}[t(H, G)]\}$. Chatterjee and Varadhan \cite{Chatterjee2011} established the existence of $0<\delta_{\min}(H) < \delta_{\max}(H)$ such that if $\delta<\delta_{\min}(H)$ or $\delta> \delta_{\max}(H)$, conditioned on
$\mathcal{E}_\delta$,
$G(n,p)$ ``looks like'' an \ER random graph, albeit with a higher edge density. They call this the ``replica symmetric'' phase. On the contrary, \cite{Chatterjee2011} also establishes that for $p$ sufficiently small, there exists $\delta \in [\delta_{\min}(H), \delta_{\max}(H)]$ such that conditioned on
$\mathcal{E}_\delta$,
 the graph is not distributed as an Erd\H{o}s-R\'{e}nyi random graph---this regime was termed as the ``replica symmetry breaking'' regime. Using the framework of \cite{Chatterjee2011}, Lubetzky and Zhao \cite{Lubetzky2015} characterized the precise boundary between the symmetry and the symmetry breaking regimes, in terms of $\delta$ and $p$. We defer an in-depth survey of large deviations on random graphs to Section \ref{sec:history}.

Random graphs are simple stochastic models for large networks observed in a myriad of scientific applications, and in this context, it is often natural to study graph models with \emph{inhomogeneities} or \emph{constraints}. Large deviation phenomena are of natural interest in this general setting, although progress in this direction requires several new ideas.  The study of large deviations for constrained random graphs has been initiated in the recent literature---\cite{dembo2018large} studies large deviations for the uniform random graph with a given number of edges, while in \cite{Dhara2019}, in joint work with Souvik Dhara, S.S. studied large deviations for random graphs with given degrees. Finally, \cite{bhattacharya2019upper} focuses on large deviations for random regular graphs in the sparse regime. In contrast, large deviations for inhomogeneous random graphs is relatively unexplored (see \cite{bhattacharya2019upper} for some preliminary results on sparse graphs). This paper seeks to fill this gap by initiating the study of large deviations for block model random graphs.

We study large deviations for block model graphs drawn from a ``base graphon'' with $k$ blocks that is specified by a set of values $\{p_{ij}\}_{1 \leq i \leq j \leq k}$, $p_{ij} \in [0,1]$.
A graph on $kn$ vertices is sampled from this model as follows:
(i) collect the vertices into $k$ groups of size $n$, indexed by $1, \dots, k$, and (ii) connect two vertices from groups $i$ and $j$ with probability $p_{ij}$.
(See Section \ref{subsec:basics} for a formal definition.)
Our contributions in this article can be summarized as follows:

\begin{enumerate}
    \item We adopt the framework of \cite{Chatterjee2011}, and establish an LDP for block model random graphs, viewed as random graphons. The induced law of the random graph satisfies an LDP with speed $n^2$---the rate function in this case is the lower semicontinuous envelope of an appropriate relative entropy functional (see formal statement as Theorem \ref{ldp}).
    Perhaps surprisingly, although the block model is quite similar to the \ER random graph, our derivation of the LDP requires going substantially beyond the ideas introduced in \cite{Chatterjee2011}.

    In particular, the derivation of the LDP in \cite{Chatterjee2011} relies heavily on the fact that an \ER random graph remains invariant in law under permutations of the vertices, a fact that is no longer true for general block models. To overcome this barrier, we rely on a two-step approach.
    (a) Using Szemer\'{e}di's Regularity Lemma, we construct a Szemer\'edi net of block graphons and \emph{cover} an event by a finite union of open balls centered on the elements of this net. Thus it suffices to characterize the limiting probability of each open ball.
    (b) To this end, we employ a ``method of types''-style argument, similar to the classical proof of Sanov's Theorem. A similar two-step strategy was employed earlier in \cite{Dhara2019} while deriving an LDP for random graphs with given degrees. A crucial technical difference between the two settings is that in \cite{Dhara2019}, the graphon being sampled from was bounded away from zero and one, whereas our results include block models that take value zero or one. As an immediate application of this general result, we obtain an LDP for the \ER bipartite graph.
 \item Our general LDP, in turn, directly implies an LDP for graph parameters continuous with respect to the cut topology (see Theorem \ref{theorem:variational-problem-nice}), e.g. homomorphism density, largest eigenvalue, etc. For such graph parameters, the rate function is expressed as a variational problem on the space of graphons.
    \item Next, we turn our attention to the variational problem for upper tail large deviations of regular subgraphs. In Theorem \ref{sym-regime}, we establish that close to the expected value, this problem exhibits a symmetric phase---where the variational problem admits solutions which exhibit the same block structure as the base graphon.  We also demonstrate that for large target values of the homomorphism density, the variational principle admits symmetric solutions.
 \item In some specific block graphons, we exhibit the existence of a non-symmetric phase---where there does not exist a symmetric minimizer (see Section \ref{exists non-sym} for the specific examples). This establishes an analogue of the reentrant phase transition\footnote{In statistical physics and chemistry, a reentrant phase transition describes a phenomenon where while walking on a straight line in parameter space, one leaves one phase, enters a new one, and then reenters into the first phase; we prefer to use this standard notion to the term ``double'' phase transition used in  \cite{chatterjee2010applications,Chatterjee2011}.}
     phenomenon, noted earlier for the upper tail problem on Erd\H{o}s--R\'{e}nyi random graphs \cite{chatterjee2010applications,Chatterjee2011}.
\item Finally, we turn to the bipartite \ER random graph in Section \ref{b-er-results} and study the variational problems corresponding to the upper tails of regular  subgraphs and largest eigenvalue. We extend the analysis of Lubetzky and Zhao \cite{Lubetzky2015} and determine the precise transition boundary between the symmetric and the symmetry breaking regimes.
\end{enumerate}

We present a brief review of the relevant facts from graph limit theory \cite{Borgs2008,Borgs2012,Lovasz2007} and detail our main results in the rest of this section.

\subsection{Graph limit theory: a brief review}
\label{subsec:basics}

In this section, we collect some facts from the theory of dense graph limits  \cite{Borgs2008,Borgs2012,Lovasz2007} which will be relevant for the subsequent discussion. We refer the interested reader to \cite{Lovasz2012} for an in-depth survey of this area.
Define the function
$\vartheta_n : [0,1] \to [n]
$ as
\begin{align}
\vartheta_n(x) &= \begin{cases} 1 & 0 \leq x \leq \frac{1}{n},\\
i & \frac{i-1}{n} < x \leq \frac{i}{n}, \,\, 1<  i \leq n.
\end{cases}\label{eq:k-n}
\end{align}
\begin{definition}[Graphon] Let $\mathcal{W}$ be the space of all measurable functions $f : [0,1]^2 \mapsto [0,1]$ such that $f(x,y) = f(y,x)$ for all $(x,y) \in [0,1]^2$. We call $f \in \mathcal{W}$ a \emph{graphon}.
\end{definition}

\begin{definition}[Empirical Graphon] Let $G$ be a simple graph on
$[n]= \{ 1, \dots, n\}$.
The empirical graphon $f^G:[0,1]^2 \to [0,1]$ is defined as follows
$$f^G(x,y)=
\begin{cases}
1 & \text{if}\quad (\vartheta_n(x), \vartheta_n(y)) \text{ is an edge in $G$},\\
0 &\text{otherwise}.
\end{cases}
$$
\end{definition}

\noindent
Next, we recall the notions of the cut distance and cut metric.
\begin{definition}[Cut Distance]
The \emph{cut distance} between two graphons $f, g\in \mathcal{W}$ is defined as
\[d_{\square}(f, g) = \sup_{S,T \subseteq [0,1]} \left| \int_{S \times T} \left(f(x,y) - g(x,y) \right) dx dy \right|,\]
where $S,T$ are measurable subsets of $[0,1]$.
\end{definition}
\begin{definition}[Cut Metric]
For $f, g \in \mathcal{W}$, the cut metric is defined as
\[\delta_{\square}(f, g) = \inf_{\phi \in \mathcal{M}} d_{\square}(f, g^{\phi})= \inf_{\phi, \psi \in \mathcal{M}} d_{\square}(f^{\psi}, g^{\phi})=\inf_{\psi \in \mathcal{M}} d_{\square}(f^{\psi}, g),\]
where $\mathcal{M}$ denotes the set of bijective, Lebesgue measure-preserving maps $\phi: [0,1] \mapsto [0,1]$.
\end{definition}

\noindent
We will establish our large deviation principle in the natural quotient space associated with $\delta_\square$.
For $f,g \in \mathcal{W}$, write $f \sim g$ if $\delta_\square(f,g) = 0$ and $\tilde f$ for the equivalence class of all $f'\sim f$. For convenience of  notation, we write ${\ti f}^G$ rather than $\wt{f^G}$ for the equivalence class containing the step function $f^G$.

Set $\tW= \W /_{\sim}$.
For $f,g \in \mathcal{W}$, we set \[\delta_{\square}(\tilde{f}, g) = \delta_{\square}(f, g)
\qquad\text{and}\qquad
\delta_{\square}(\tilde{f}, \tilde{g}) = \delta_{\square}(f, g).\]
The above are well-defined, as for all $f_1, f_2$ such that $f_1 \sim f_2$, it holds that $\delta_{\square}(f_1, g) = \delta_{\square}(f_2, g)$.
In \cite{Lovasz2007}, Lov\'{a}sz and Szegedy prove one of the central results in graph limit theory---the metric space $(\tW, \delta_\square)$ is compact.
 In particular, it is separable, which implies  that the Borel $\sigma$-algebra  over $\tW$ is generated by the open balls.  We use $\tilde{\mathcal A}$ to denote this $\sigma$-algebra.

Throughout the paper we use $V(G)$ and $E(G)$ to denote the sets of vertices and edges of the graph $G$ respectively. 
\begin{definition}[Homomorphism Density]\label{hom-density}
 Let $H = (V(H), E(H))$ be a simple graph, where the vertices are labeled as
 $[v]=\{1,\dots,v\}$, where $v=|V(H)|$. Define the homomorphism density of $H$ in $f \in \W$ as
\begin{align*}
t(H,f) &= \int_{[0,1]^v} \prod_{(i,j) \in E(H)} f(x_i, x_j) dx_1 \dots dx_v.
\end{align*}
\end{definition}
\noindent Since $t(H,f)=t(H,g)$ whenever $f \sim g$, $t(H, \cdot)$ is well-defined on $\tW$. With a slight abuse of notation, we use
the same symbol for the function $t(H, \cdot):\tW\to [0,1]:\tilde f\mapsto t(H,\tilde f)$.
As shown in \cite{Borgs2008,Lovasz2007}, this function is continuous for any finite graph $H$.

In this article, we study large deviations for block model random graphs. To this end, we denote as $\mathcal{B}^\gamma$ the set of block graphons where the width of the blocks are given by the values in the vector $\gamma$, which we assume to be rational. Let $\Delta_m= \{\gamma \in [0,1]^m \cap \mathbb{Q}^m: \sum_{i=1}^m \gamma_i=1 \}$ denote the $(m-1)$-dimensional simplex restricted to rational points.

\begin{definition}\label{block-graphon}
Given $\gamma \in \Delta_m$,
we define
$I_1=[0,\gamma_1]$ and
$$I_j=\left(\sum_{k=0}^{j-1} \gamma_k ,\sum_{k=0}^{j} \gamma_{k} \right]\qquad 1<j\leq m.$$
From these intervals, define the interval membership function
\begin{align}
\label{eq:k-gamma}
\vartheta_{\gamma}(x) = \sum_{j=1}^m j \mathbbm{1}\{x \in I_j\}.
\end{align}
When $\gamma$ is clear from context, we write $\vartheta(x)$.
Let $\mathcal{B}^\gamma$ be the set of graphons $f \in \mathcal{W}$ of the form
\[f(x,y) = p_{\vartheta(x), \vartheta(y)},\]
where $p_{ij}=p_{ji}\in [0,1]$. We call such a graphon an $m$-block graphon.
When $\gamma$ is clear from context, we write $f \in \mathcal{B}^\gamma$ as $f=(p_{ij})_{i,j \in [m]}$. When $\gamma =(1/m, \dots 1/m)$, we say $f \in \mathcal{B}^\gamma$ is a uniform size (or simply uniform) $m$-block graphon. When $m = 2$, we write $\mathcal{B}^{(\gamma, 1-\gamma)}$ for $\gamma \in (0,1) \cap \mathbb{Q}$ to denote the set of graphons with blocks of size $\gamma$ and $1-\gamma$.

Let $$\mathcal{B}^{\gamma,*}= \left\{ f \in \mathcal{B}^\gamma: f \not \in \mathcal{B}^\eta \text{ for all } \eta \in \Delta_{m-1} \right\}.$$
In other words, $\mathcal{B}^{\gamma, *}$ is the subset of graphons in $\mathcal{B}^{\gamma}$ that cannot be described by a smaller number of blocks.
Let
$$\ti{\mathcal{B}}^\gamma= \{ \ti{f} \in \tW: \delta_\square(\ti f, g) =0 \text{ for some }  g \in \mathcal{B}^\gamma\}.$$
\end{definition}
Finally we define the sampling distribution for dense block model random graphs.
We recall that $\tilde{\mathcal A}$ denotes the Borel $\sigma$-algebra over the metric space $(\tW,\delta_{\square})$.
\begin{definition}[Sampling from a block model]
Let $W_0 =(p_{ij})_{i,j\in [k]}$ be a uniform $k$-block graphon. Let $\pr_{kn,W_0}$ denote the probability distribution over $\mathcal{W}$ obtained by sampling from $W_0$ as follows. Construct a simple graph $G$ on $kn$ vertices with unique labels in $[kn]$. Independently, add an edge between vertex $i$ and vertex $j \neq i$ with probability
$W_0(i/kn,j/kn)=p_{\lceil i/{n}  \rceil, \lceil  j/{n} \rceil}$.
Return the empirical graphon $f^G$. Let $\ti{\pr}_{kn,W_0}$ denote the probability distribution induced on $\tW$ by the measure $\pr_{kn,W_0}$, i.e.,
$
\ti{\pr}_{kn,W_0}(\tilde A)={\pr}_{kn,W_0}({\ti f^G}\in \tilde A)
$
for all $A\in \tilde{\mathcal A}$.
\end{definition}

\begin{remark}
Note that any graphon with rational-length blocks is a uniform $k$-block graphon for some $k$, and thus the above scheme can be used to sample from such graphons.
\end{remark}

\subsection{A large deviation principle for block models}
First we define the relative entropy function, both pointwise and for entire graphons. These definitions will be used to define the rate function for the LDP. Throughout we use the conventions
$0\log 0=0$ and $0\log(0/0)=0$.  

\begin{definition}[Relative entropy]\label{rel-entropy-defn}
Define $I_{W_0}: \W \to \mathbb{R} \cup \{\infty\}$ as
$$
I_{W_0}(f) = \frac{1}{2}\int_{[0,1]^2} h_{W_0(x,y)} \left(f(x,y)\right) dx dy,
$$
where $h_p(u)$ is the usual relative entropy,
\[h_p(u) =  u \log \frac{u}{p} + (1-u) \log \frac{1-u}{1-p}. \]
\end{definition}

\noindent While the previous definition applies to all graphons, we now specialize to block graphons.
Given $W_0$, let $\Omega= \{(x,y) : W_0(x,y)\in (0,1)\}$. Define
\begin{equation}\label{defn:W-Omega}
    \mathcal{W}_{\Omega} = \{f \in \mathcal{W} : \lambda\left(\{(x,y) \in \Omega ^c : f(x,y) \neq W_0(x,y) \}\right) = 0 \}
\end{equation}
and
 \begin{equation}
    \tW_\Omega=\{ \ti{f} \in \tW :\delta_\square(f, g)=0 \text{ for some $g \in \W_\Omega$}\},
 \end{equation}
where $\lambda(\cdot)$ is the Lebesgue measure on $[0,1]^2$.
In other words, $\mathcal{W}_{\Omega}$ is the set of graphons that agree with $W_0$ wherever $W_0$ takes value $0$ or $1$, except possibly on a measure-zero set. Note that $\mathbb{P}_{kn, W_0}$ and $\tilde{\mathbb{P}}_{kn, W_0}$ are supported on $\W_\Omega$ and $\tW_\Omega$ respectively.
Lemma~\ref{closed tilde} states that $\tW_\Omega$ is closed (and hence compact), and
\Cref{finite-or-not} states that
$I_{W_0}$  is bounded on $\W_\Omega$, and infinite on $\W\setminus \W_\Omega$.

Note that \ER random graphs correspond to the constant base graphon $W_0 = p$---this model satisfies an LDP with speed $n^2$, and rate function $I_{p}$ \cite{Chatterjee2011}. However, in the general case, the function $I_{W_0}(\cdot)$ is not well-defined on the quotient space $\tW$, and thus cannot be the rate function for our LDP. We introduce our candidate rate function $J_{W_0}$  on $\tW$ as follows. To this end, we will use the symbols $B$ and $S$ to denote the
 closed balls in $\W$ and $\tW$:
\begin{align*}
B(\tilde{f}, \epsilon) &= \{ g \in \W : \delta_\square( \ti{f}, g) \leq \ve\}\\
S(\tilde{f}, \epsilon) &= \{\tilde{g} \in\tW: \delta_{\square}(\tilde{f}, \tilde{g}) \leq \epsilon\}.
\end{align*}

\begin{definition}[Rate function]\label{def:rate-function}
The rate function is defined as
$$J_{W_0}(\tilde{f}) =
\begin{cases}
\sup_{\eta > 0} \inf_{ h \in B( \tilde{f}, \eta)} I_{W_0}(h) & \ti{f}\in \tW_\Omega\\
\infty & \ti{f}\not\in \tW_\Omega.
\end{cases} $$
\end{definition}
\noindent
In \Cref{prelim} we prove that $J_{W_0}$ is lower semi-continuous on $\tW$ (\Cref{lemma:semi-continuity-J}),
and that it is bounded by some constant
$C(W_0)<\infty$ on $\tW_\Omega$  (\Cref{finite-or-not}). Follow-up work by Grebik and Pikhurko \cite{Grebik2021} simplified the expression for the rate function, and established that one can instead work with 
$$J_{W_0}(\tilde{f}) =
 \inf_{ h : \delta_{\square}(h, \tilde{f}) = 0} I_{W_0}(h).
$$
It would be interesting to see if this alternate expression can simplify our subsequent analysis. Additional follow-up work by Markering \cite{Markering2020} showed that the same rate function applies when $\log(W_0), \log(\mathbf{1} - W_0) \in L^1([0,1]^2)$.
Note that a block graphon with blocks containing zeros or ones does not satisfy such an integrability condition.

\begin{theorem}\label{ldp}
Let $W_0$ be a uniform $k$-block graphon. The sequence $\tilde{\mathbb{P}}_{kn, W_0}$ obeys a large deviation principle in the space $\tW$ (equipped with the cut metric $\delta_\square$) with speed $(kn)^2$ and rate function $J_{W_0}$. Explicitly,
\begin{enumerate}
    \item  For any open set $\tilde{U} \subseteq\tW$,
$\liminf_{n \to \infty} \frac{1}{(kn)^2} \log \tilde{\mathbb{P}}_{kn, W_0}(\tilde{U}) \geq - \inf_{\tilde{h} \in \tilde{U}} J_{W_0}(\tilde{h})$,
\item For any closed set $\tilde{F} \subseteq\tW$,
$\limsup_{n \to \infty} \frac{1}{(kn)^2} \log \tilde{\mathbb{P}}_{kn, W_0}(\tilde{F}) \leq - \inf_{\tilde{h} \in \tilde{F}} J_{W_0}(\tilde{h})$,
\end{enumerate}
(where we define the $\inf$ over the empty set to be $\infty$.)
\end{theorem}

\begin{remark}
Note that any graphon with rational-length blocks is a uniform $k$-block graphon for some $k$. Therefore, our result also describes large deviations events for any base graphon with rational-length blocks.
\end{remark}

The proof of the LDP requires several new ideas, beyond those introduced in \cite{Chatterjee2011}. To explain the main additional difficulties, note that for the \ER random graph, $W_0$ is the constant graphon taking a value $p$, and thus the cut-distance $\delta_\square(W_0,f)$ to an arbitrary graphon $f\in \W$ is equal to the distance $d_\square(W_0,f)$. Somewhat related, the relative entropy $I_{W_0}$ is a well-defined rate function on equivalence classes $\ti f=\{g:\delta_{\square}(f,g)=0\}$.  Neither of these holds if $W_0$ is a block model with more than one block. To some extent, similar issues were faced in \cite{Dhara2019} in the context of large deviations for dense random graphs with given degrees. Our proof follows their general proof outline.  However, the graphons $W_0$ considered in \cite{Dhara2019} are bounded away from zero and one, thus making the distinction between $\ti{\W}$ and $\ti{\W}_{\Omega}$ unnecessary. In contrast, the base graphon $W_0$ in our setting can have zero or one blocks---this creates many new analytic and probabilistic hurdles, and makes our analysis substantially more challenging.

\subsection{LDP for graph parameters and the associated variational problem}
In this section, we turn our attention to upper tail large deviations for continuous graph parameters.

\begin{definition}\label{defn:cont-parameter}
A \emph{graph parameter} is a function $\tau : \tW \to \mathbb{R}$. We extend such a function $\tau$ to $\mathcal{W}$ by setting $\tau(f)=\tau(\ti f)$, where as before, $\tilde f$ is the equivalence class containing $f$.
We further write $\tau(G) = \tau(f^{G})$ for any graph $G$. 
We set $t_{\max}^{\tau}(\tW)=\max_{\ti f\in \tW}\tau(\ti f)$ and $t_{\max}^{\tau}(\tW_\Omega)=\max_{\ti  f\in \tW_\Omega}\tau(\ti f)$. When $\tau$ and $\tW$ or $\tW_{\Omega}$ are clear from context, we simply write $t_{\max}$. We similarly define $t^{\tau}_{\min}(\tW)$, $t^{\tau}_{\min}(\tW_{\Omega})$, and $t_{\min}$ as the corresponding minimal graph parameter values.

A graph parameter $\tau$ is \emph{continuous} if it is continuous with respect to $\delta_{\square}$. 
\end{definition}

\begin{remark}
Note that the metric space $(\tW,\delta_{\square})$ is compact, and thus every continuous graph parameter is, in fact, uniformly continuous with respect to $\delta_{\square}$. 
\end{remark}

\noindent
Note that by the compactness  of $\tW$ and $\tW_\Omega$, the maxima in the above expressions are actually maxima and not suprema. Also, note that $t_{\max}^{\tau}(\tW) = \max_{f \in \mathcal{W}} \tau(f)$.

\begin{definition}\label{definition:nice-graph-parameter}
Let $\tau$ be a continuous graph parameter. For $W_0 \in \mathcal{W}$ and
$t \leq t^{\tau}_{\max}(\tW)$
we set
\begin{align}
\phi_{\tau}(W_0,t) = \min \{J_{W_0}(\tilde{f}) : \tilde{f} \in\tW, \tau(\tilde f) \geq t\}.
\label{phi-defn}
\end{align}
For $t>t^{\tau}_{\max}(\tW)$, we set $\phi_\tau(W_0,t)=\infty$.
\end{definition}

\noindent Note that continuity of $\tau$ and compactness of $(\tW, \delta_{\square})$ imply that $\{ \ti{f} \in \tW: \tau(\ti{f}) \geq t\}$ is compact. Since the lower semi-continuous function $J_{W_0}$ (\Cref{lemma:semi-continuity-J}) attains its minimum on any compact set, it follows that $\phi_{\tau}(W_0,t)$ is well defined.
 Note also that $\{ \ti{f} \in \tW: \tau(\ti{f}) \geq t\}$ has non-empty intersection with $\tW_\Omega$
 if $t\leq t^{\tau}_{\max}(\tW_\Omega)$, in which case
 $\phi_{\tau}(W_0,t)$ is bounded above by a constant depending on $W_0$, and that
$\{ \ti{f} \in \tW: \tau(\ti{f}) \geq t\}\cap \tW_\Omega=\emptyset$ and $\phi_{\tau}(W_0,t)=\infty$ if $t>t_{\max}$.  So in particular, $\phi_{\tau}(W_0,t)$ is discontinuous at $t=t_{\max}$. In addition,  $\phi_{\tau}(W_0,t)=0$ if $t\leq \tau(W_0)$, and $\phi_{\tau}(W_0,t)>0$ on $( \tau(W_0), t_{\max} ]$. To see this, observe that if $t\leq \tau(W_0)$, then $\{ \ti{f} \in \tW: \tau(\ti{f}) \geq t\}$ contains the equivalence class $\widetilde{W}_0$, and thus $\phi_{\tau}(W_0,t)=0$. On the other hand, $J_{W_0}(\ti{f})=0$ if and only if $\delta_{\square}(\tilde{f},W_0)=0$ (\Cref{prop:j-same}), and thus $\phi_{\tau}(W_0, t) >0$ for $t\in (\tau(W_0), t_{\max}]$.

Our next result establishes $\phi_{\tau}$ as the rate function for the upper tail large deviation of the graph parameter $\tau$.
Moreover, this result proves that conditioned on the rare event, the random graph \emph{concentrates} on the minimizers of \eqref{phi-defn}. This result is a direct adaptation of \cite[Theorem 2.7]{Lubetzky2015} to general $k$-block graphons $W_0$.

\begin{theorem}\label{theorem:variational-problem-nice}
Let $W_0$ be a uniform $k$-block graphon. Let $\tau$ be a continuous graph parameter, $t \leq t^{\tau}_{\max}(\tW)$, and let $G_{kn}$ be the graph on $kn$ vertices sampled from $W_0$ according to the probability distribution $\pr_{kn,W_0}$.
Recall $\phi_{\tau}(W_0, t)$ from \eqref{phi-defn}, and assume that  $\phi_{\tau}(W_0, \cdot)$ is continuous at $t$.  Then
\[\lim_{n \to \infty} \frac{1}{(kn)^2} \log \pr_{kn,W_0}\left(\tau(G_{kn}) \geq t \right) = - \phi_{\tau}(W_0, t). \]
Set $\ti{F}^{\star}$ to  be the set of minimizers of \eqref{phi-defn}. Then $\tilde{F}^{\star}$ is a non-empty compact subset of $\tW$.   If $t>t_{\max}$, then
$\pr_{kn,W_0} \left( \tau(G_{kn}) \geq  t\right)=0$, and if  $t<t_{\max}$, then for $n$ sufficiently large and
each $\epsilon > 0$, there exists $C = C(\tau, \epsilon, W_0, t) > 0$ such that
\[
\pr_{kn,W_0} \left(\delta_{\square}(G_{kn}, \tilde{F}^{\star}) < \epsilon \Big| \tau(G_{kn}) \geq  t\right) \geq 1 - \exp{-Cn^2}.
\]
In particular, if $\tilde{F}^{\star} = \{\tilde{f}^{\star}\}$ for some $\tilde{f}^{\star} \in \tW$, then as $n \to \infty$, the conditional distribution of  ${\ti f}^{G_{kn}}$ given the event $\tau(G_{kn}) \geq t$ converges to the point mass at $\tilde{f}^{\star}$.
\end{theorem}

\begin{remark}
Note that, in general,\Cref{theorem:variational-problem-nice} holds only at the continuity points of $\phi_\tau(W_0, t)$. \Cref{countable} explains that $\phi_\tau$ has at most countably many points of discontinuity when $\tau$ is a continuous graph parameter. Moreover, we establish  (see \Cref{p-star}) that  $\phi_\tau$ is continuous
on $\R\setminus\{t_{\max}\}$ if $\tau$ satisfies the ``sufficient increase property'' (\Cref{defn:p-star}).
In turn, the proof of
\Cref{homomorphism-cont} establishes that homomorphism densities
$t(H,\cdot)$
have the sufficient increase property for all finite graphs $H$ and all step functions $W_0$, and the proof of \Cref{operator-cont} establishes that the operator norm has the sufficient increase property for a specific family of graphons $W_0$, namely those which generate bipartite \ER graphs.

So in particular, we know that
for these graph parameters, the conclusions of  \Cref{theorem:variational-problem-nice} hold for all
$t$, with the possible exception of $t=t_{\max}$.
As we will see in \Cref{thm:tmax_behavior}, for the graph paramater
$\tau = t(H, \cdot)$, where $H$ is a finite $d$-regular graph, they also hold
 at $t = t_{\max}$, in spite of the fact that  $\phi_\tau(W_0, t)$ is not continuous at this point.
\end{remark}

 \Cref{theorem:variational-problem-nice}
 establishes that typical behavior under the upper tail large deviation event is governed by the solutions of the variational problem
  \eqref{phi-defn}. This directly motivates our subsequent investigations into the properties of this problem.

\begin{definition}[Symmetric Regime] \label{j-sym} Let $m \in \mathbb{Z}^+$, $\gamma \in \Delta_m$, $W_0 \in \mathcal{B}^{\gamma, \ast}$, and let $\tau$ be a continuous graph parameter. We say that $t \leq t_{\max}^{\tau}(\tW_{\Omega})$ is in the symmetric regime for $W_0$ and $\tau$ if all
minimizers
$\ti{g}$
of
\begin{align}\label{min_eq}
\min_{\ti f \in \tW}\{ J_{W_0}(\ti{f}) \colon  \tau(\ti f) \geq t\},\end{align}
 satisfy  $\ti{g} \in \ti{\mathcal{B}}^\gamma$. We call the symmetric solution unique if a unique element of $\tW_\Omega$ minimizes \eqref{min_eq}.
\end{definition}

\Cref{theorem:variational-problem-nice} implies that in the symmetric regime, the conditional distribution of the random graph concentrates on a set of graphons with block structure agreeing with $W_0$. In addition, if  there is a unique symmetric solution, the graph concentrates on the point mass corresponding to this solution. Our subsequent results explore the existence of a  symmetric regime for specific graph parameters and establish uniqueness for a class of bipartite graphons $W_0$.

Next we specialize to the graph parameter defined by $d$-regular subgraph densities, i.e., to the graph parameter
$\tau:\tilde f\mapsto t(H,\tilde f)$ for a $d$ regular graph $H$.  In Section~\ref{sec:symm}, we first show that for $\delta$ sufficiently small, $t=(1+\delta)t(H, W_0)$ is in the symmetric regime of $W_0$ and this graph parameter.  Then we show that when $t$ is sufficiently close to the maximum homomorphism density, $t$ is also in the symmetric regime.  In \Cref{exists non-sym}, we study examples of two-block graphons $W_0$ that have a non-symmetric regime---this exhibits that in these examples, these two symmetric regimes are separated by a non-symmetric regime, establishing a ``reentrant'' phase transition phenomenon for large deviations in stochastic block models, analogous to the one established
in \cite{chatterjee2010applications, Chatterjee2011}
for large deviations in \ER random graphs.

\subsection{The existence of a symmetric regime for $d$-regular graphs}
\label{sec:symm}

The next theorem establishes the existence of a symmetric regime for $\delta$ sufficiently small.
\begin{theorem}\label{sym-regime} Let $m \in \mathbb{Z}^+$, $\gamma \in \Delta_m$, $W_0 \in \mathcal{B}^{\gamma, \ast}$.
Let $H$ be a $d$-regular graph, 
and let $\tau = t(H, \cdot)$.
If $t(H, W_0)<t_{\max}^{\tau}(\tW_{\Omega})$, then  there
 exists $\delta >0$ sufficiently small such that for all $t \in [t(H, W_0), (1+\delta) t(H, W_0))$, $t$ is in the symmetric regime for $W_0$ and $t(H, \cdot)$. 
 If $h$ is a minimizer, then
$$J_{W_0}(\ti{h})=I_{W_0}(h)=\min\{I_{W_0}(g): g \in \mathcal{B}^\gamma, t(H, g) \geq (1+\delta)t(H, W_0)\}.$$

\end{theorem}
Next, we explore the variational problem near the maximum homomorphism density, and establish the existence of a symmetric regime in this setting.

\begin{theorem}\label{sym-regime-near-one} Let $m \in \mathbb{Z}^+$, $\gamma \in \Delta_m$, $W_0 \in \mathcal{B}^{\gamma, \ast}$.
Let $H$ be a $d$-regular graph, and let $\tau = t(H,\cdot)$.
If $t(H, W_0)<t^{\tau}_{\max}(\tW_{\Omega})$, then  exists $\eta>0$ such that for all $t \in  ((1-\eta)t_{\max},t_{\max} ]$, $t$ is in the symmetric regime for $W_0$ and $t(H, \cdot)$. In addition, if $h$ is a minimizer, then
$$J_{W_0}(\ti{h})=I_{W_0}(h)=\min\{I_{W_0}(g): g \in \mathcal{B}^\gamma, t(H, g) \geq (1+\delta)t(H, W_0)\}.$$
\end{theorem}

\noindent
Theorems \ref{sym-regime} and \ref{sym-regime-near-one} establish the existence of a symmetric regime for the homomorphism density of regular graphs.  This is challenging due to the form of the rate function $J_{W_0}(\cdot)$, and is one of the main technical contributions of this paper. To this end, our first contribution is to establish that 
\begin{align} \label{J-to-I}
\min \{J_{W_0}(\ti{f}) : \ti{f} \in \mathcal{W}, \tau(\ti f) \geq t\}= \inf \{I_{W_0}(f) : f \in \mathcal{W}, \tau(f) \geq t\}
\end{align}
under mild assumptions on the graph parameter $\tau$, which are satisfied for homomorphism densities and the operator norm (\Cref{phi is h}).  This insight facilitates our subsequent analysis, and allows us to work with the relatively entropy functional $I_{W_0}$, instead of the complicated rate function $J_{W_0}$.

Even with this simplification, our proof is quite involved. To exhibit the existence of a symmetric phase, we will establish that for certain ranges of $t$ (depending on $W_0$), any minimizer of \eqref{min_eq} is in $\tilde{\mathcal{B}}^\gamma$. To this end, we will establish that if $\tilde{f}$ is a minimizer of \eqref{min_eq}, there exists a sequence of block constant graphons $\{f_n: n \geq 1\} \subseteq \mathcal{B}^\gamma$  such that $\delta_\square(\ti{f},f_n) \to 0$. This will imply that $\ti{f} \in \tilde{\mathcal{B}}^\gamma$, as $\ti{\mathcal{B}}^\gamma$ is closed in $(\tW, \delta_\square)$. We refer the reader to Section \ref{sec:symmetry} for details on the construction of this sequence $\{f_n : n \geq 1\}$.


Combined with
Theorem \ref{theorem:variational-problem-nice}, these two theorems  characterize the ``typical'' structure of the graph, conditioned on an upper tail large deviation event for the graph parameter $\tau = t(H, \cdot)$  in the vicinity of the endpoints of $[t(H, W_0), t_{\max}]$.
  However, as stated,
Theorem \ref{theorem:variational-problem-nice} applies only for  $t \in [t(H, W_0), t_{\max})$.
 It is natural to wonder what happens when $t=t_{\max}$. This is the content of the next theorem.

To state it, we recall the notation $W_0=(p_{ij})_{i,j\in [m]}$ for a graphon with blocks $I_{i}\times I_{j}$, $i,j\in [m]$, and define a block $I_{a}\times I_{b}$ to be relevant
if $p_{ab}>0$ and $t(H,W_0)$ strictly decreases if $p_{ab}$ is lowered. Note that by definition, all blocks where $p_{ab}=0$ are not relevant, while the blocks where $p_{ab}=1$ may or may not be relevant. Note that the  maximum homomorphism density of a fixed subgraph $H$ in a random graph drawn from $\pr_{kn,W_0}$ is $\max_{f \in \W_\Omega} t(H, f)$.

\begin{theorem}
\label{thm:tmax_behavior}
Let $W_0$ be a uniform $k$-block graphon.
 Let $H$
 be a finite $d$-regular graph, let $\tau = t(H, \cdot)$, 
 and let $f_{\max}$ be the step function which is equal to $1$ on all relevant blocks, and equal to $W_0$ on all irrelevant blocks.
 Then $\ti f_{\max}$ is the unique minimizer
of \eqref{min_eq} at $t^{\tau}_{\max}(\tW_{\Omega})$.
Moreover, for any $\epsilon>0$, there exists a constant $C>0$ such that
\begin{align*}
\mathbb{P}_{kn, W_0} \left( \delta_{\square}( f^{G_{kn}} , \ti{f}_{\max} ) < \epsilon  \Big| t(H, G_{kn}) \geq t_{\max}   \right)
\geq 1- \exp{-Cn^2},
\end{align*}
implying that as $n \to \infty$, the conditional distribution of  ${\ti f}^{G_{kn}}$ given the event $\tau(G_{kn}) \geq t_{\max}$ converges to the point mass at $\tilde{f}_{\max}$.
\end{theorem}

\begin{remark}\label{remark:tmax}
If $\tau(\ti{f}) =  t^{\tau}_{\max}(\widetilde{\W}_{\Omega}) $ has a unique solution $\ti{f}_{\max}$, it is immediately clear that 
	\begin{align*}
		\mathbb{P}_{kn,W_0}\left(\delta_{\square}(f^{G_{kn} },\ti{f}_{\max}  )=0
		\Big| \tau(\ti{f}^{G_{kn}})  \geq  t_{\max} \right) =1,
	\end{align*}implying that the conditional distribution of $\ti{f}^{G_{kn}}$ given the event $\{ \tau(\ti{f}^{G_{kn}}) \geq t_{\max}  \}$
	is  the point mass at  $\ti{f}_{\max}$.
The equation $\tau(\ti{f}) =  t^{\tau}_{\max}(\tW_{\Omega})$ has a unique solution, for example, if $\tau=t(H, \cdot)$, where $H$ is a finite $d$-regular graph, and all blocks of $W_0$ that are subsets of $\Omega$ are relevant.
 In this case $f_{\max} = \mathbf{1}_{W_0>0}$. This also holds if $W_0$ is a bipartite graphon with two blocks and $\tau(\ti{f}) = \| f\|_{\text{op}}$ (see Theorem \ref{theorem:eigenvalue}), again with $f_{\max} = \mathbf{1}_{W_0 >0}$. See the Appendix for additional details. 
\end{remark}

The proof of \Cref{thm:tmax_behavior} is relatively straightforward given the proofs of  \Cref{theorem:variational-problem-nice,sym-regime,sym-regime-near-one}, and is deferred to the Appendix.

\subsection{A non-symmetric regime in special cases} \label{exists non-sym}

Next, we establish the existence of a non-symmetric regime in some specific families of two-block graphons. Let
\begin{align*}
f_{p,q,r}^{\gamma}(x,y) &= \begin{cases}
p & \text{if } (x,y) \in [0,\gamma ]^2\\
r &\text{if } (x,y) \in (\gamma, 1]^2\\
q & \text{otherwise.}
\end{cases}
\end{align*}

\tikzset{blockr/.style= {rectangle, draw=black!50, fill=black!20, thick}}
\tikzset{blockr/.style= {rectangle, draw=black!50, fill=black!20, thick}}
\tikzset{blockr/.style= {rectangle, draw=black!50, fill=black!20, thick}}
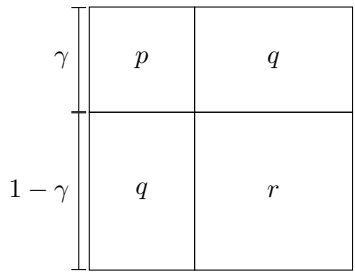
\begin{figure}[h]
\centering
\begin{tikzpicture}[scale=0.7]
\draw [] (0,0) rectangle node {$p$} (2,-2);
\draw [] (2,-2) rectangle node {$r$} (5,-5);
\draw [] (2,0) rectangle node {$q$} (5,-2);
\draw [] (0,-2) rectangle node {$q$} (2,-5);
\draw [|-|] (-0.2,0) -- (-0.2,-2) node[midway,left] {$\gamma$};
\draw [|-|] (-0.2,-2) -- (-0.2,-5) node[midway,left] {$1-\gamma$};
\end{tikzpicture}
\label{fig:special-graphon}
\caption{Illustration of the graphon $f_{p,q,r}^{\gamma}$.}
\end{figure}

\noindent
We show the existence of a non-symmetric regime for base graphons of the form $f_{0,p,p}^\gamma,f_{1,p,p}^\gamma, f_{1,p,0}^\gamma$ when $p$ is sufficiently small. The first model corresponds to an \ER random graph with a planted independent set, while the second example covers \ER graphs with a planted clique. Finally, the third graphon leads to a bipartite \ER random graph with a planted clique in one of the partitions.

\begin{theorem}\label{nonsym opt}
Let $\gamma \in (0,1) \cap \mathbb{Q}$ and $H$ be a $d$-regular graph.  Assume that
\begin{enumerate}
    \item  $0 < t< t(H, f^\gamma_{0,0,1})$ and denote $W_p=f_{0,p,p}^\gamma$,  or
\item $t(H, f^\gamma_{1,0,0}) < t< 1$ and denote $W_p=f_{1,p,p}^\gamma$,  or
\item  $t(H,f^\gamma_{1,0,0}) < t< t(H, f^\gamma_{1,1,0})$ and denote $W_p=f_{1,p,0}^\gamma$.
\end{enumerate}
Separately, under each of these assumptions, there exists $p_0 > 0$ such that if $p < p_0$,
\[\min \{ J_{W_p}(\ti{g}) : t(H,\ti{g}) \geq t\}<\min \{ J_{W_p}(\ti{g}) : \ti{g} \in \ti{\mathcal{B}}^{\left(\gamma, 1-\gamma\right)}, t(H,\ti{g}) \geq t\}.\]
These statements imply that for $p$ small enough, the minimizer of the variational problem
\eqref{min_eq} is non-symmetric.
\end{theorem}

\noindent
In \Cref{b-gamma-compact}, we show that $\{\ti{g} \in \ti{\mathcal{B}}^{\left(\gamma, 1-\gamma\right)}: t(H,\ti{g}) \geq t\}$ is compact, which justifies the minimum on the right hand side in \Cref{nonsym opt}. Note that in the second and third case, we cover all $t\in (t_{\min},t_{\max})$, 
where $t_{\min},t_{\max}$ are the minimal and maximal values of $t(H, W_p)$ ranging over $p \in (0,1)$. On the other hand, in the first case, we do not consider the full range, since we exclude $t \in [t(H, f^\gamma_{0,0,1}), t(H, f^\gamma_{0,1,1}))$.

To establish this result, we recall that
$I_{W_0}$ is significantly more tractable than the rate function $J_{W_0}$. Our first step (see \Cref{nice-i-min}) is to show that if $W_0$ is a graphon of the form $f_{0,p,p}^{\gamma}$ or $f_{1,p,p}^{\gamma}$, then
\begin{align}\label{goal-now-now}\min\{ J_{W_0}(\ti{f}): \ti{f} \in \ti{\mathcal{B}}^{(\gamma,1-\gamma)},\, \tau(f) \geq t\}= \min\{ I_{W_0}(f): f \in \mathcal{B}^{(\gamma,1-\gamma)} \cup \mathcal{B}^{(1-\gamma, \gamma)}, \, \tau(f) \geq t\}.
\end{align}
A similar simplification occurs for $W_0$ of the form $f_{z_1,p,z_2}^\gamma$, where $z_1, z_2 \in \{0,1\}$.

Next we show that for graphons $W_p$ of the form $W_p=f_{0,p,p}^\gamma$ and $W_p=f_{1,p,p}^\gamma$, there exists $p_0 > 0$ such that if $p < p_0$,
\begin{align}\label{last-piece}\inf\{I_{W_p}(f): t(H,f) \geq t\} < \min\{ I_{W_p}(f): f \in \mathcal{B}^{(\gamma,1-\gamma)} \cup \mathcal{B}^{(1-\gamma, \gamma)}, \, t(H,f) \geq t\}\end{align} for some range of $t$
(\Cref{construct nonsym opt}). We establish this by constructing explicit graphons with lower entropy than that of all graphons in $\mathcal{B}^{(\gamma,1-\gamma)} \cup \mathcal{B}^{(1-\gamma, \gamma)}$. A similar simplification occurs for $W_p=f_{1,p,0}^\gamma$.
Together with \eqref{J-to-I} and  \eqref{goal-now-now}, \eqref{last-piece} implies the desired conclusion.

\subsection{Bipartite \ER graphs: symmetry vs. symmetry  breaking}\label{b-er-results}
Lubetzky and Zhao \cite{Lubetzky2015} characterize the symmetric regimes for $d$-regular subgraph counts and the largest eigenvalue in the \ER model. We extend these results to bipartite \ER random graphs.
Let $f_p^{\gamma}$ denote the graphon $f_{0,p,0}^\gamma$, i.e.,
\begin{align*}
f_p^{\gamma}(x,y) &= \begin{cases}
0 & (x,y) \in [0,\gamma]^2 \cup (\gamma, 1]^2\\
p & \text{otherwise}.
\end{cases}
\end{align*}

\tikzset{blockr/.style= {rectangle, draw=black!50, fill=black!20, thick}}
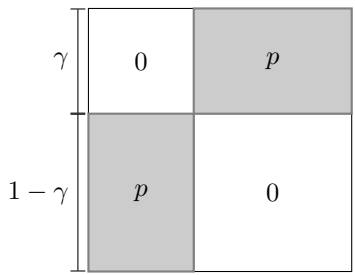
\begin{figure}[h]
\centering
\begin{tikzpicture}[scale=0.7]
\draw [] (0,0) rectangle node {$0$} (2,-2);
\draw [] (2,-2) rectangle node {$0$} (5,-5);
\draw [blockr] (2,0) rectangle node {$p$} (5,-2);
\draw [blockr] (0,-2) rectangle node {$p$} (2,-5);
\draw [|-|] (-0.2,0) -- (-0.2,-2) node[midway,left] {$\gamma$};
\draw [|-|] (-0.2,-2) -- (-0.2,-5) node[midway,left] {$1-\gamma$};
\end{tikzpicture}
\caption{Illustration of the graphon $f_p^{\gamma}$.}\label{fig:bipartite-graphon}
\end{figure}
\noindent For $W_0$ of the form $f_p^\gamma$, the following theorem completely characterizes the symmetric and non-symmetric regimes for $t(H, \cdot)$, where $H$ is a regular graph.

\begin{theorem}\label{theorem:variational-problem-bipartite}
Fix $0 < p < 1$, $\gamma \in (0,1) \cap \mathbb{Q}$, and $H$ a $d$-regular graph with $d\geq 1$. Let $W_0 = f_p^\gamma$. Let $r \in [p,1]$ and define $t_r^\gamma=t(H, f_r^\gamma)$.
\begin{enumerate}
    \item If $(r^d, h_p(r))$ lies on the convex minorant of $x \mapsto h_p(x^{1/d})$,  then $t_r^\gamma$
is in the symmetric regime for $W_0$ and $t(H, \cdot)$. Moreover, $\ti{f}_r^\gamma$ is the unique symmetric solution.
\item If $(r^d, h_p(r))$ does not lie on the convex minorant of $x \mapsto h_p(x^{1/d})$, then $t_r^\gamma$ is not in the symmetric regime of $W_0$ and $t(H, \cdot)$.
\end{enumerate}
\end{theorem}

\begin{remark}
The symmetric regime for subgraph counts in \ER graphs \cite{Lubetzky2015} takes a similar form, with $t_r^{\gamma}$ replaced by $t(H, r)$, where $r$ denotes the constant graphon with value $r$. 
\end{remark}

Finally, we characterize the symmetric regime for the largest eigenvalue. Similar to \ER graphs, the boundary for the symmetric regime for the largest eigenvalue coincides with that of the density of two-regular graphs.

\begin{definition}
For a graphon $f \in \mathcal{W}$, define the Hilbert--Schmidt kernel operator $T_f$ on $L^2([0,1])$ by
\[(T_f u)(x) = \int_0^1 f(x,y) u(y) dy \] for any $u \in L^2([0,1])$. The operator norm is given by
\[\Vert f \Vert_{\emph{op}} = \min \{ c \geq 0 : \Vert T_f u \Vert_2 \leq c \Vert u \Vert_2 \text{ for all } u \in L^2([0,1])\}. \]
\end{definition}

\begin{lemma}[\cite{Lubetzky2015}, Lemma 3.6]\label{lemma:continuous-extension}
The function $\Vert \cdot \Vert_{\emph{op}}$ is a continuous extension of the normalized graph spectral norm, i.e., $\lambda_1(G)/n$ for a graph $G$ on $n$ vertices, to $(\tW, \delta)$. Moreover, $\Vert \cdot \Vert_{\emph{op}}$ is a continuous graph parameter.
\end{lemma}

\begin{theorem}\label{theorem:eigenvalue}
Fix $0 < p < 1$, $\gamma \in (0,1) \cap \mathbb{Q}$, and let $W_0 = f_p^\gamma$. Let $r \in [p,1]$ and define $t_r^\gamma=\Vert f_r^{\gamma} \Vert_{\emph{op}}$.
\begin{enumerate}
    \item If $(r^2, h_p(r))$ lies on the convex minorant of $x \mapsto h_p(x^{1/2})$,  then $t_r^\gamma$
is in the symmetric regime for $W_0$ and $\Vert \cdot \Vert_{\emph{op}}$. Moreover, $\ti{f}_r^\gamma$ is the unique symmetric solution.
\item If $(r^2, h_p(r))$ does not lie on the convex minorant of $x \mapsto h_p(x^{1/2})$, then
 $t_r^\gamma$ is not in the symmetric regime for $W_0$ and $\Vert \cdot \Vert_{\emph{op}}$.
\end{enumerate}
\end{theorem}

\begin{remark}
It is not hard to see that  for  $\tau=t(H,\cdot)$ and $\tau=\|\cdot\|_{\text{op}}$, the function $r\mapsto \tau(f_r^\gamma)$ is a continuous and non-decreasing function on  $ [p,1]$ and that $ \tau(f_1^\gamma) =t^{\tau}_{\max}(\tW_{\Omega})$. Thus Theorems~\ref{theorem:variational-problem-bipartite} and \ref{theorem:eigenvalue} cover the full range $[\tau(W_0),t_{\max}]$.
\end{remark}

To establish Theorems \ref{theorem:variational-problem-bipartite} and \ref{theorem:eigenvalue}, we follow the general approach introduced in \cite{Lubetzky2015}. \Cref{nice-i-min} implies that \eqref{goal-now-now} holds for $\tau(g) = t(H, g)$ where $H$ is a $d$-regular graph or $\tau(g) = \Vert g \Vert_{\text{op}}$, meaning that we can again reason about symmetry through the function $I_{W_0}$ rather than $J_{W_0}$. For $r \in (0,1]$, let $f_r^{\gamma}$ be the bipartite graphon with value $r$, and  $t_r^{\gamma} = t(H, f_r^{\gamma})$ be the corresponding homomorphism density of a $d$-regular graph $H$. We apply a  generalized H\"older inequality to show that whenever $f \in \mathcal{W}_{\Omega}$ satisfies $t(H,f) \geq t(H, f_r^{\gamma})$, it holds that $\Vert f \Vert_d^d \geq 2 \gamma (1-\gamma) r^d$ (\Cref{lemma:norm}). Finally, we show that if $(r^d, h_p(r))$ lies on the convex minorant of $x \mapsto  h_p(x^{1/d})$ and $\Vert f \Vert_d^d \geq 2 \gamma (1-\gamma) r^d,$
then $I_{W_0}(f) \geq I_{W_0}(f_r^\gamma)$, with equality occurring if and only if $f = f_r^\gamma$ (\Cref{lemma:symmetric-general}). To establish the non-symmetric regime, we show that whenever $(r^d, h_p(r))$ is not on the convex minorant, we can construct a graphon $g$ with $t(H,g) > t(H,f_r^\gamma)$ and $I_{W_0}(g) < I_{W_0}(f_r^\gamma)$ (\Cref{lemma:non-symmetric}). This construction is more complicated than the one in \cite{Lubetzky2015}, due to the bipartite nature of the underlying graph (see \Cref{fig:bipartite-construction} for the construction). The proof for the spectral norm $\tau(g) = \Vert g \Vert_{\text{op}}$ follows using similar arguments.

\subsection{History and related work}
\label{sec:history}
The upper tail large deviation problem for subgraphs of $G(n,p)$ has attracted considerable attention in Probability and Combinatorics. By applying the theory of graph limits,
Chatterjee and Varadhan found the precise constant in the large deviation probability in the dense case \cite{Chatterjee2011}. This approach does not work in the sparse regime where $p\to 0$, as graphon theory only applies to dense graphs.

The challenge of deriving an LDP for sparse graphs has attracted considerable attention in recent years.  In the sparse regime, even determining the right order of this probability on the exponential scale proved to be considerably challenging. Following partial advances \cite{janson2004upper,janson2002infamous,janson2004deletion,kim2004divide,vu2001large}, this was finally resolved for $H=K_3$ in  \cite{chatterjee2012missing,demarco2012upper}. Subsequently, \cite{demarco2012tight} identified the right order of this probability for  $H= K_r$, $r\geq 4$, and formulated a conjecture regarding the correct order for general subgraphs. See \cite{vsileikis2019counterexample} for a recent counterexample to this general conjecture.
Recently, the development of general theory \cite{chatterjee2016nonlinear,eldan2018gaussian,austin2019structure} and problem-specific ideas \cite{cook2018large,augeri2018nonlinear,harel2019upper,basak2019upper} have contributed to rapid progress on large deviations in the sparse setting.
These results relate the large deviation probability to an entropic variational problem. In turn, some of these variational problems have also been solved \cite{bhattacharya2016upper,bhattacharya2017upper,bhattacharya2018upper}, leading to deep insights regarding the structure of the random graph, conditioned on the rare event.

We emphasize that these remarkable results are mostly applicable for sparse random graphs or hypergraphs \cite{liu2019upper}, and do not shed any direct insight on the problem considered in this paper. Instead, our work is the first step towards a full generalization of the work of \cite{Chatterjee2011} and \cite{Lubetzky2015} to block models. As  \cite{Chatterjee2011} did for \ER graphs, we establish an LDP for block models and demonstrate the existence of a reentrant phase transition for the upper tail of $d$-regular subgraph counts. While we exhibit a reentrant phase transition for a limited class of block models, we show the existence of a symmetric regime for arbitrary block models. Our methods are inspired by the work of \cite{Lubetzky2015}, which completely characterizes the symmetric and non-symmetric regimes for \ER graphs. Moreover, analogous to \cite{Lubetzky2015}, we fully characterize the symmetric and non-symmetric regimes for bipartite \ER graphs. As discussed in the introduction, our work fits into the broader theme of large deviations for dense random graphs with inhomogeneities or constraints, and provides the first rigorous analysis of the large deviations problem for dense block models.

Following the posting of this paper to arXiv, there has been some follow-up work. Grebik and Pikhurko \cite{Grebik2021} simplified our function given by \Cref{def:rate-function} showing that taking the lower-semicontinuous regularization is not necessary. Building on our work, Grebik and Pikhurko \cite{Grebik2021} derived an LDP for graphs sampled from step graphons whose blocks are not necessarily of rational length.  Markering \cite{Markering2020} derived a large deviation principle for inhomogeneous \ER random graphs, showing that the rate function takes a simple form under certain integrability assumptions. 

\noindent\textbf{Outline:} The rest of the paper is structured as follows. We establish our main LDP results, Theorem \ref{ldp} and Theorem \ref{theorem:variational-problem-nice}, in Section \ref{sec:ldp_proof}. In Section \ref{sec:properties}, we derive some analytic properties of $\phi_{\tau}$ which are crucial in the analysis of the variational problem. Section \ref{sec:symmetry} establishes the existence of a symmetric regime in the upper tail, while Section \ref{sec:non-sym} establishes the existence of a non-symmetric regime in specific examples. Finally, we characterize the symmetric regime in \ER bipartite models in Section \ref{sec:bipartite}. We finish with some open problems in Section \ref{sec:open}.\\

\noindent
\textbf{Acknowledgments:}
This work was initiated when CB, JC, JG, and SP were affiliated with Microsoft Research New England (MSR NE). The authors all
thank MSR NE for the vibrant research environment. SP was supported by the NSF Graduate Research Fellowship DGE-1650044. JG was supported by a Microsoft Research PhD fellowship. The work was continued while SP was at the School of Mathematics at the Georgia Institute of Technology and then at the NSF-Simons Center for the Mathematical and Statistical Analysis of Biology at Harvard, and JG was at the Operations Research Center and then the Mathematics Department, Massachusetts Institute of Technology. The authors thank Yufei Zhao and Souvik Dhara for helpful discussions during the early part of the project. They also thank Oleg Pikhurko and Jan Grebik for generously sharing their results on LDPs for $W$-random graphs. Finally, the authors thank the anonymous referee for pointing out an error regarding the uniqueness of the optimizer in the symmetric phase in an earlier version of this manuscript.

\section{Large deviation principle}
\label{sec:ldp_proof}
In this section we establish the LDP. Since the space $(\tW,\delta_\square)$ is compact, it will be enough to prove the bounds in Theorem \ref{ldp} for balls in the metric $\delta_\square$; the precise statement is given in the following lemma.
\begin{lemma}\label{rem:LDP-compact}
Since the space $(\tW,\delta_\square)$ is compact, the bounds in Theorem~\ref{ldp} are equivalent to
\begin{enumerate}
\item
For all $\epsilon > 0$ and $\ti{h} \in\tW$,
$
\liminf_{n \to \infty}\frac{1}{(kn)^2} \log \ti{\mathbb{P}}_{kn, W_0}\brac{S (\ti{h}, \epsilon)} \geq - J_{W_0}(\ti{h});
$
\item
For all $\ti{g} \in\tW$,
$
\lim_{\alpha \to 0} \limsup_{n \to \infty} \frac{1}{(kn)^2} \log \ti{\pr}_{kn,W_0}\brac{S(\ti{g}, \alpha)} \leq -J_{W_0}(\ti g).
$
\end{enumerate}
\end{lemma}
\noindent
The proof is standard (see e.g., \cite[Theorems 4.1.11, 4.1.18]{amirdembooferzeitouni2010}), and is thus omitted.

 In \Cref{prelim}, we begin by establishing several useful facts about the rate function and the space $\W_\Omega$. We establish the LDP lower and upper bounds in \Cref{sec:lower,sec:upper} respectively. Finally, in \Cref{sec:pf-vp-nice}, we prove \Cref{theorem:variational-problem-nice}, which establishes upper tail large deviations for continuous graph parameters.

\subsection{Preliminaries}\label{prelim}
In this section, we establish a variety of useful analytical properties. Of particular interest are \Cref{closed tilde} and \Cref{lemma:semi-continuity-J}, which state that $\tW_\Omega$ is closed and that $J_{W_0}$  is lower semi-continuous on $(\tW, \delta_\square)$.

\begin{lemma}\label{closed tilde} For any $m \in \mathbb{Z}^+$, $\gamma \in \Delta_m$, and $W_0\in  \mathcal{B}^\gamma$, the set
 $\tW_\Omega$ is closed in $\tW$ with respect to the cut metric topology $(\tW, \delta_\square)$.
 \end{lemma}

 \begin{lemma}\label{lemma:semi-continuity-J}
For any $m \in \mathbb{Z}^+$, $\gamma \in \Delta_m$, and $W_0\in  \mathcal{B}^\gamma$, the function $J_{W_0}(\cdot)$ is
 lower semi-continuous on $(\tW, \delta_{\square})$.
\end{lemma}

We start by stating some elementary properties of the relative entropy $h_p(\cdot)$.
\begin{lemma}\label{lem:h_p}
Let $\beta\in(0,1/2]$ and let $G_p(a,q)=aq-\log(pe^a+1-p)$. Then the following holds

(i) For all $p\in [\beta,1-\beta]$, $\|h_p\|_\infty\leq \log(2/\beta)$.

(ii) The family of functions $(h_p)_{p\in [\beta,1-\beta]}$ is equicontinuous on $[0,1]$.

(iii) For all $p\notin\{0,1\}$, $h_p(q)=\sup_{a\in \R}G_p(a,q)$

(iv) For $q\notin\{0,1\}$, the $\sup$ in (iii) is achieved by
$a=\log\left(\frac q{1-q}\frac {1-p}p\right)$.
\end{lemma}
\begin{proof}
 (i) Follows by observing that
 $|x\log x+(1-x)\log(1-x)|\leq \log 2$
 and
 $|x\log p|+|(1-x)\log p|\leq x| \log\beta|+(1-x)|\log\beta|=-\log\beta$.

\noindent (ii) Follows from uniform continuity of the  function $x\mapsto x\log x+(1-x)\log (1-x)$.

\noindent (iii) and (iv) are elementary exercises left to the reader.
\end{proof}

 The function $I_{W_0}$ also comes up naturally in \cite{Dhara2019}---however,
 in \cite{Dhara2019} it is assumed that the base graphon $W_0$ is bounded away from zero and one, and thus the function $I_{W_0}$ is necessarily finite. This is not the case in our context.  We use $Im(W_0)$ to denote the image of $W_0$ in $[0,1]$.

\begin{proposition}
 \label{finite-or-not} Let $W_0 \in \W$.  If $f \not \in \W_\Omega$, then $I_{W_0}(f)=\infty$.  If $f \in \W_\Omega$
 and $W_0$ obeys the assumption
 \begin{equation}\label{Im-gap}
 \beta= \inf\{w>0 : w \in Im(W_0) \text{ or } 1-w \in Im(W_0)\}>0,
 \end{equation}
then $I_{W_0}(f) \leq \frac 12 \log (2/\beta)$.
\end{proposition}

\begin{proof}
Since $h_0(x) =\infty$ for $x \not=0$ and $h_1(x)=\infty$ for $x \not =1$, it follows that $I_{W_0}(f)=\infty$ when
$f \not \in \W_\Omega$.
To bound $I_{W_0}(f)$ for $f\in \W_\Omega$,  observe that
\begin{equation} \label{int-over-omega} I_{W_0}(f) =
 \frac{1}{2}\int_{\Omega} h_{W_0(x,y)} \left(f(x,y)\right) dx dy,
\end{equation}
since $f$ and $W_0$ agree on $ \Omega^c$.
The proof is completed by invoking Lemma~\ref{lem:h_p} (i).
\end{proof}

\begin{proposition}
 \label{uniform-cont} Let $\ve>0$ and
 assume that $W_0\in \W$ obeys the condition \eqref{Im-gap}.  Then there
 exists $\eta > 0$ such that if $f, g \in \W_\Omega$ and $\Vert f-g \Vert_\infty \leq \eta$, then $|I_{W_0}(f)- I_{W_0}(g)| \leq \ve$.
\end{proposition}

\begin{proof}
Fix $\ve>0$.  By Lemma~\ref{lem:h_p} (ii) there exists an $\eta>0$ such that
$|h_p(u)- h_p(v)| \leq \ve$ whenever $|u-v| \leq \eta$ and $p\in [\beta,1-\beta]$.
Inserted into \eqref{int-over-omega}, this
completes the proof.
\end{proof}

\noindent
We derive a variational representation for $I_{W_0}$ using convex duality.
\begin{proposition}\label{proposition:rate-function-equivalent}
Let $W_0 \in \W$ and let $S$ be the set of all symmetric functions in $ L_2([0,1]^2)$.
For $a\in S$ and $f\in\W$, define
\begin{align}
K_{W_0}(f,a) &= \int_{[0,1]^2} \left[a(x,y) f(x,y)  -  \log \Big(W_0(x,y) e^{a(x,y)} + 1-W_0(x,y)\Big) \right]dx dy .\label{eq:supremum-I-function}
\end{align}
Then $ I_{W_0}(f)=\frac 12 \sup_{a\in S} K_{W_0}(f,a)$.
\end{proposition}
\begin{proof}
First, we consider the case $f \not \in \W_\Omega$ (in which case  $ I_{W_0}(f)=\infty$ by \Cref{finite-or-not}).  Then there exists $\Gamma \subseteq [0,1]^2$ with positive measure such that  $W_0(x,y)\in\{0,1\}$  and $W_0(x,y) \not = f(x,y)$ for $(x,y) \in \Gamma$. Choosing
 $$
 a_M(x,y)=
\begin{cases}
0 & (x,y) \not \in \Gamma\\
M & W_0(x,y)= 0\\
-M & W_0(x,y)= 1.
\end{cases}
$$
and taking $M \to \infty$, we see that $\sup_{a\in S} K_{W_0}(f,a)=\infty$ in this case.

Next we consider the case $f \in \W_\Omega$.  Recalling the definition of $G_p$ from Lemma~\ref{lem:h_p} and noting that the integrand in \eqref{eq:supremum-I-function} is zero if $W_0(x,y)=f(x,y) \in \{0,1\}$,
we then have $K_{W_0}(f,a)=\int_\Omega G_{W_0(x,y)}(a(x,y),f(x,y))dxdy$.  Combined with \eqref{int-over-omega} and Lemma~\ref{lem:h_p} (iii),
this shows that
$$
 I_{W_0}(f)=\frac 12 \int_\Omega h_{W_0(x,y)}(f(x,y))dxdy\geq
 \frac 12 \sup_{a\in S}\int_\Omega G_{W_0(x,y)}(a(x,y),f(x,y))dxdy
 =\frac 12 \sup_{a\in S} K_{W_0}(f,a).
$$
To prove equality, we may w.l.o.g. assume that the right hand side is finite.  We may further restrict the integrals on both sides to the subset $\tilde\Omega\subset\Omega$ where $f(x,y)\not\in\{0,1\}$, since the contributions of both sides to the complement can easily be seen to be equal.
Finally, on $\tilde\Omega$, we may use Lemma~\ref{lem:h_p} (iv) to conclude that
$$
\int_{\tilde\Omega} h_{W_0(x,y)}(f(x,y))dxdy=\int_{\tilde \Omega} G_{W_0(x,y)}(a_0(x,y),f(x,y))dxdy
$$
where $a_0(x,y)=\log\left(\frac {f(x,y)}{1-f(x,y)}\frac {1-W_0(x,y)}{W_0(x,y)}\right)$.  If $a_0\in L^2(\tilde\Omega)$, the right hand side is bounded
by the $\sup$ over all $a\in S$, giving the desired upper bound.  If it is not, we replace $\tilde\Omega$ on both sides by its intersection with the set of points for which $|a_0(x,y)|\leq M$ before bounding the right hand side by a $\sup$ over all square integrable $a$.  The proof is concluded by using the monotone convergence theorem.
\end{proof}

\noindent
The following  propositions will be used to establish the lower semi-continuity of $I_{W_0}$.

\begin{proposition}\label{closed} Let $m \in \mathbb{Z}^+$ and $\gamma \in \Delta_m$. For $W_0\in  \mathcal{B}^\gamma$, the set
 $\W_\Omega$ is a closed subset of $\W$ with respect to the topology $(\W,d_\square)$.
 \end{proposition}
\begin{proof}
Let $\{f_n: n \geq 1\} \subseteq \W_\Omega$ be a convergent sequence of graphons satisfying $d_\square(f_n, f) \to 0$  for some $f \in \mathcal{W}$. Then on each block $I \times J$ of $W_0$ that takes value $0$ or $1$, $f_n= W_0$ for all $n \geq 1$. Since $|\int_{I \times J} f_n - f | \leq d_\square(f_n, f)$ and $d_\square(f_n, f) \to 0$, it follows that $f= W_0$  on $I \times J$. Thus $f \in \W_\Omega$.
\end{proof}

\begin{proposition} \label{extension}
Let $(S,d)$ be either $(\W, d_\square)$ or $(\ti{\W}, \delta_\square)$, 
and let $F \subseteq S$ be a closed subset of $S$.
Let $f: F \to \R$ be a lower semi-continuous function on $(F, d)$. The extension $f^*: S \to \R \cup \{ \infty \}$  where $$f^*(x)= \begin{cases} f(x) & x \in F\\ \infty & x \in S \setminus F\end{cases}$$ is lower semi-continuous on $(S, d)$.
\end{proposition}

\begin{proof}
We show that $f^*$ is lower semi-continuous on $S$ by demonstrating that for all
$\alpha\in\R$,
the set $\{x \in S |f^{*}(x) > \alpha\}$ is open. Observe
$$ \{x \in S |f^{*}(x) > \alpha\}=\{x \in F | f(x) >\alpha \} \cup \brac{ S \setminus F}.$$
By lower semi-continuity of $f$ on $F$, $A=\{x \in F | f(x) >\alpha \} $ is open in $F$, and so $A^c$ is relatively closed with respect to $F$ and therefore closed in $S$ (since $F$ is closed). It follows that
$$\{x \in S | f^*(x) >\alpha \}^c= F \cap A^c$$ is closed, and so we conclude that $\{x \in S | f^*(x) >\alpha \}$ is open.
\end{proof}

\begin{lemma}\label{lemma:semi-continuity-I}
Let $m \in \mathbb{Z}^+$ and $\gamma \in \Delta_m$. For $W_0\in  \mathcal{B}^\gamma$,
the function $I_{W_0}(\cdot)$ is lower semi-continuous on $(\mathcal{W}, d_{\square})$.
\end{lemma}
\begin{proof}

 First, note that  by \Cref{closed,extension}, and the observation that $I_{W_0}(f) =\infty$ for all $f \in \W \setminus \W_\Omega$,
it is enough
to establish that $I_{W_0}$ is lower semi-continuous on $(\W_\Omega, d_\square)$.  Second, by
 \Cref{proposition:rate-function-equivalent}, $I_{W_0}$ can be written as supremum over the functions $K_{W_0}(\cdot,a)$, so it will be
enough to show that for all $a\in S$, the function $K_{W_0}(\cdot,a)$ is continuous on $(\W_\Omega, d_\square)$.

Consider two functions $f,g\in \W_\Omega$, and observe that every  $a\in S$ can be approximated by step functions in $L^2$.  Given $\ve$ and $a$ we can therefore find $k<\infty$ and a
$k$-step function $a_k$ such that $\|a-a_k\|_2\leq \frac \ve 2$.
As a consequence
$$
\Big|K_{W_0}(f,a)-K_{W_0}(g,a)\Big|=
\Big|\int a(f-g)\Big|\leq \|a(f-g)\|_1\leq \|a_k(f-g)\|_1+\frac \ve 2
\leq \|a_k\|_\infty k^2 d_{\square}(f,g)+ \frac\ve 2.
$$
This is smaller than $\ve$ if $d_{\square}(f,g)<\ve/( \|a_k\|_\infty k^2)$, which proves that $K_{W_0}( \cdot, a)$
is continuous on $(\W_\Omega, d_\square)$, as required.
\end{proof}

\noindent
The following proposition will be used to prove that for any continuous graph parameter $\tau$ and graphon $W_0$, the function $\phi_{\tau}(W_0, \cdot)$ is strictly positive on $(\tau(W_0), t_{\max}]$.

\begin{proposition}
 \label{prop:j-same} For $m \in \mathbb{Z}^+$, $\gamma \in \Delta_m$, $W_0\in  \mathcal{B}^{\gamma}$, and  $\ti{f} \in \tW$, it holds that
$$J_{W_0}(\ti{f})=0 \text{ if and only if } \delta_\square(\ti f, W_0)=0.$$
\end{proposition}
\begin{proof}
Noting that our
definition of $J_{W_0}$ agrees with the definition of $J_{W_0}$ given in \cite{Dhara2019} whenever $f \in \W_\Omega$,
and that $\delta_{\square}(f,W_0)>0$ if $f \notin \W_\Omega$, the proposition follows from the analogous statement in
 \cite[Lemma 2.2]{Dhara2019}.
\end{proof}

Our next result establishes that $\tW_\Omega$ is closed in the cut metric. For a partition $\mathcal{P}$ of $[0,1]$, we define $W_\mathcal{P}$ as the step function graphon that is obtained by averaging over all blocks induced by the partition classes. Setting $\Gamma(x) \subseteq [0,1]$ to be the partition class in $\mathcal{P}$ that contains $x$, we obtain
$$
W_\mathcal{P}(x,y)=\frac{1}{|\Gamma(x)|\cdot |\Gamma(y)|} \int_{\Gamma(x) \times \Gamma(y)} W(u,v) du dv.$$
We call  $\mathcal{P}$ an equipartition if all classes have the same measure, and use
$|\mathcal P|$ to denote the number of classes in $\mathcal P$.  Note that up to sets of measure zero, there is just one
equipartition of $[0,1]$ into $n$ intervals; for definiteness and consistency with our previous conventions, we use the partition $([0,1/n],(1/n,2/n],\dots,((n-1)/n,1])$.

\begin{lemma}[Corollary 3.4 of  \cite{Borgs2008}]\label{equipartition}
Let $f \in \W$ and $s$ be a positive integer. For every equipartition $\mathcal{Q}$ of $[0,1]$ , there is  an equipartition $\mathcal{P}$ with ${s|\mathcal Q|}$ classes such that  $\mathcal{P}$ refines $\mathcal{Q}$ and
$$d_\square( f, f_\mathcal{P}) \leq \sqrt{\frac{20}{\log_2 s}}.$$
\end{lemma}

\noindent
The next lemma follows from Lemma~\ref{equipartition}.
\begin{lemma}\label{intervalpartition}
Let $m \in \mathbb{Z}^+$, $\gamma \in \Delta_m$, and $W_0\in  \mathcal{B}^\gamma$. Then there exists a sequence of refining partitions $\mathcal P_k$
of $[0,1]$ into equal length intervals such that
for all
$f\in \W_\Omega$ there exists a sequence of step functions $f_k\in  \W_\Omega$ with steps in $\mathcal P_k$
such that (i) $(f_{k+1})_{\mathcal P_{k}}=f_{k}$ and
(ii) $\delta_\square(f,f_k)\leq 1/k$  for all $k\geq 1$.
\end{lemma}
\begin{proof}
Let $s_k$ be such that $\sqrt{{20}/{\log_2s_k}}\leq 1/k$, let
$q_1$ be such that the lengths of the intervals described by $\gamma$ are integer multiples of $1/q_1$,  define $q_k$ inductively by $q_k=s_kq_{k-1}$, and
let $\mathcal P_k$ be the partition of $[0,1]$ into intervals of lengths $1/q_k$.
We will define $f_k$ as $f_k=(g_k)_{\mathcal P_k}$ where $g_k\in \W_\Omega$ will be inductively defined in such a way
that (a) $(g_{k})_{\mathcal P_{k-1}}=(g_{k-1})_{\mathcal P_{k-1}}$ for all $k\geq 2$,  (b) $\delta_\square(f,g_k)=0$ for all $k\geq 1$ and
(c) $d_\square(g_{k},(g_{k})_{\mathcal P_k})\leq 1/k$ for all $k\geq 1$.  This clearly implies the statement
of the lemma, since  $g_k\in \W_\Omega$ implies $(g_k)_{\mathcal P_k}\in  \W_\Omega$ by the fact that $\mathcal P_k$ is a refinement of $\mathcal P_1$,
$$
	(f_{k+1})_{\mathcal P_k}=((g_{k+1})_{\mathcal P_{k+1}})_{\mathcal P_k}=(g_{k+1})_{\mathcal P_k}=(g_k)_{\mathcal P_k}=f_k
	$$
by (a) and the fact that $\mathcal P_{k+1}$ is a refinement of $\mathcal P_k$,  and the two statements (b) and (c) imply (ii).

We start our inductive construction by setting $g_1=f$.  Noting that $d_\square(h,h')\leq 1$ for all $h,h'\in \W$, this shows that $g_1$ satisfies the inductive assumptions.

Let $k\geq 2$ and assume that $g_{k-1}$ satisfies the inductive assumption.  By Lemma~\ref{equipartition}, we can find an equipartition $\mathcal Q_{k}$ of $[0,1]$ into $q_k=s_k|\mathcal P_{k-1}|$ classes
such that $\mathcal Q_{k}$ refines $\mathcal P_{k-1}$ and $d_\square(g_{k-1},(g_{k-1})_{\mathcal Q_k})\leq 1/k$.
 We now define a measure preserving bijection $\phi:[0,1]\to[0,1]$ as follows: Let $I$ be an interval in $\mathcal P_{k-1}$, and let
$Y_1,\dots Y_{s_k}$ those elements of ${\mathcal Q_k}$ subdividing $I$.  By Theorem A.7 in \cite{jansongraphonbook}, we can find a measure preserving bijection from $I$ to itself such that the image of $Y_1,\dots Y_{s_k}$ are the $s_k$ intervals in $\mathcal P_k$ that subdivide $I$. Doing this for all intervals in $\mathcal Q_k$
we obtain a measure preserving bijection $\phi$ such that the images of the partition classes of $\mathcal Q_k$ are the partition classes of $\mathcal P_k$, and such that $\phi$ maps each interval in $\mathcal P_{k-1}$ onto itself.  Applying this bijection to $g_{k-1}$ gives a graphon $g_k\in \W_\Omega$ such that $\delta_\square(g_k,g_{k-1})=0$ and
$d_\square(g_{k},(g_{k})_{\mathcal P_k})=d_\square(g_{k-1},(g_{k-1})_{\mathcal Q_k})\leq 1/k$.  By the inductive assumption (b)
we have that $\delta_\square(g_k,f)=0$, and by the fact that $\phi$ maps each interval in $\mathcal P_{k-1}$ onto itself we have that
$(g_k)_{\mathcal P_{k-1}}=(g_{k-1})_{\mathcal P_{k-1}}$.  This completes the inductive proof.
\end{proof}

\noindent
 We show that $\tW_\Omega$ is closed, using ideas from the proof that $(\tW , \delta_\square)$ is compact \cite[Theorem 5.1]{Lovasz2007}.

\begin{proof}[Proof of \Cref{closed tilde}]
We establish the lemma by showing that  $\tW_\Omega$ contains its limit points.
Let $(\wt{W}_n)_{n \geq 0}$ be a sequence of graphons in $\tW_\Omega$ that converges  to $\wt{W}\in\tW$. Since $\wt{W}_n \in \tW_\Omega$, we may  chose a sequence $W_n \in \mathcal{W}_\Omega$ such that $\delta_\square(W_n,\wt W)\to 0$.
 We claim that $\wt{W} \in \tW_\Omega$.

 By \Cref{intervalpartition}, we can find a sequence of refining partitions $\mathcal P_k$ of $[0,1]$ into intervals of length $1/|\mathcal P_k|$ and sequences $W_{n,k}\in \W_\Omega$ such that
\begin{enumerate}[(i)]
    \item  $\delta_\square( W_n, W_{n,k}) \leq 1/k$
    \item $(W_{n,k+1})_{\mathcal P_k}=W_{n,k}$.
\end{enumerate}
Next we claim that it is possible to replace $(W_n)$ with a subsequence such that for all $k$, $W_{n,k}$ converges almost everywhere to a step function $U_k$
with steps made out of the intervals in $\mathcal P_k$.
Indeed, select a subsequence of $(W_n)$ such that
the  value of $W_{n,1}$ converges on the product of   all intervals $I,I'\in \mathcal P_1$.
We obtain $W_{n,1} \to U_1$ almost everywhere for $U_1$ a step function on $s_1$ intervals of $[0,1]$. Taking further subsequences for $k=2,3 \dots$, we obtain a subsequence of $(W_n)$ such that $W_{n,k} \to U_k$ almost everywhere for all $k$.  By the Dominated Convergence Theorem, $\Vert W_{n,k}- U_k \Vert_1 \to 0$, and so $d_\square( W_{n,k}, U_k) \to 0$. Each $U_k$ is a step function on $s_k$ intervals of $[0,1]$. Note that since  $W_{n,k} \in \W_\Omega$ and $d_\square(W_{n,k}, U_k) \to 0$, \Cref{closed} implies that $U_k \in \mathcal{W}_\Omega$. For the remainder of this proof, we replace $(W_n)$ with this subsequence; doing so does not change the limit of the corresponding sequence in $\tW$.

Next we claim that the sequence $(U_k)_{k\geq 1}$ has a limit $U$ in $\mathcal{W}_\Omega$.
 It follows from (ii) that $U_k=(U_\ell)_{\mathcal{P}_k}$ for all $\ell > k$. Let $(x,y)$ be a uniform random point in $[0,1]$. Since $U_k=(U_\ell)_{\mathcal{P}_k}$, the sequence $(U_1(x,y), U_2(x,y), \dots)$ is a martingale with respect to the canonical filtration. The random variables $U_i(x,y)$ are bounded, and so the Martingale Convergence Theorem \cite[Theorem 4.2.11]{durrett2019} implies that the sequence $(U_1(x,y), U_2(x,y), \dots)$ converges with probability one. Thus there exists $U \in \W$ such that $U_k \to U$ almost everywhere. By the Dominated Convergence Theorem, $\Vert U_k -U \Vert_1 \to 0$ and therefore $d_\square(U, U_k) \to 0$. Since $U_k \in \W_\Omega$ for all $k$, and $\W_\Omega$ is closed, it follows that $U \in \W_\Omega$. Moreover $\ti{U} \in \tW_\Omega$.

It remains to show that $\delta_\square(W_n,U) \to 0$ (as this implies that $\delta_\square( {\wt{W}}_n, \wt{U}) \to 0$, which establishes that the limit of the sequence is in $\tW_\Omega$). Let $\ve>0$. Choose $k > 3/\ve$ sufficiently large such that $\Vert U - U_k \Vert_1 < \ve/3$. For this fixed $k$, there exists $n_0$ such that $\Vert U_k - W_{n,k} \Vert_1 < \ve/3$ for all $n \geq n_0$. Observe that
\begin{align*}
\delta_\square(U, W_n) &\leq  d_\square(U, U_k) +  d_\square(U_k, W_{n,k}) +  \delta_\square(W_{n,k}, W_n)\\
&\leq \Vert U- U_k\Vert_1 +  \Vert U_k- W_{n,k}\Vert_1 +  \delta_\square(W_{n,k}, W_n) \leq \ve.
\end{align*}
\end{proof}

\begin{proof}[Proof of \Cref{lemma:semi-continuity-J}.]
 We modify the proof of \cite[Lemma 2.1]{Dhara2019} to allow for $W_0$ with values in $\{0,1\}$.
For $\ti{f} \in\tW$, let
\[
H(\ti{f}) = \inf_{g \in \W : \delta_\square(g, \ti{f}) =0} I_{W_0}(g).
\]
If $\ti f\notin \tW_\Omega$, then $g\notin \W_\Omega$ for all $g$ contributing to the infimum and by \Cref{finite-or-not},  $H ( \ti{f})= \infty$.
On the other hand, if $\ti f\in  \tW_\Omega$, there exists a $g\in \W_\Omega$ contributing to the infimum, so with the help of \Cref{finite-or-not}
we conclude that  $H(\cdot)$ is bounded  on $\tW_\Omega$.
Combined with the fact that
 $$
h\in B(\ti{f}, \delta)\iff\delta_\square(\ti{g},h)=0 \text{ for some } \ti{g} \in S(\ti{f}, \delta),
 $$
we obtain that for $\ti{f} \in \tW_\Omega$
\begin{align*}
J_{W_0}(\ti{f}) = \sup_{\delta > 0} \inf_{ h \in B( \ti{f}, \delta)} I_{W_0}(h)
= \sup_{\delta > 0} \inf_{ \ti{g} \in S(\ti{f}, \delta)} H(\ti{g})
= \sup_{\delta > 0} \inf_{ \ti{h} \in S( \ti{f}, \delta) \cap \tW_\Omega} H(\ti{h})
= \liminf_{\ti{h} \to \ti{f}} H(\ti{h}).
\end{align*}
Therefore, $J_{W_0}(\cdot)$ is the pointwise $\liminf$ of a bounded function. This implies that $J_{W_0}(\cdot)$ is lower semi-continuous on $\tW_\Omega$.

Note that $J_{W_0}(f)=\infty$ for all $f \in \tW \setminus \tW_\Omega$. Therefore the lower semi-continuity of $J_{W_0}(\cdot)$ on $\tW$ follows by \Cref{extension} and \Cref{closed tilde}.
\end{proof}

\noindent
We close this preliminary section with a proposition and a lemma which will be used  in the proofs of Theorems~\ref{sym-regime}, \ref{sym-regime-near-one} and
 \ref{nonsym opt}.

\begin{proposition}\label{b-gamma-compact}
Let $m \in \mathbb{Z}^+$ and $\gamma \in \Delta_m$. Let $\tau$ be a continuous graph parameter.  Then the following holds:

(i) The set $\{\ti{g} \in \ti{\mathcal{B}}^\gamma: \tau(\ti g) \geq t\}$ is a compact set in $(\tW, \delta_\square)$.

(ii) The set $\{{g} \in {\mathcal{B}}^\gamma: \tau( g) \geq t\}$ is a closed set in $(\W, d_\square)$.

\end{proposition}

\begin{proof}
Since $\tW$ is compact, it suffices to show that
$\{\ti{g} \in \ti{\mathcal{B}}^\gamma: \tau(\ti g) \geq t\}$
is closed.
 Let $\ti{f}_n \in  \ti{\mathcal{B}}^\gamma$ be such that
 $\tau( \tilde{f}_n) \geq t$ and $\ti{f_n}$ converges to some graphon $\ti{f}$. Since
 $\tau$
 is continuous,
 $\lim_{n \to \infty}
 \tau( \ti{f}_n) = \tau(  \ti{f}) \geq t$.

 It remains to show that $\ti{f} \in \ti{\mathcal{B}}^\gamma$.
Without loss of generality, we may assume $f_n \in \mathcal{B}^\gamma$, and write $f_n= (\alpha_{ij}^n)_{i,j \in [m]}$, where each $\alpha_{ij}^n \in [0,1]$. By the compactness of $[0,1]^{m^2}$, there exists a subsequence such that  $$\alpha_{ij}^{n_k} \to \beta_{ij} \text{ for all $i,j \in [m]$}.$$
Let $g=(\beta_{ij})_{i,j \in [m]} \in \mathcal{B}^\gamma$. Since $f_{n_k} \to g$ pointwise and $d_\square(f_{n_k}, g) \leq \Vert f_{n_k}-g\Vert_1$, the Dominated Convergence Theorem implies that $d_\square(f_{n_k}, g) \to 0$. Since $\delta_\square(\ti{f}_n, \ti{f}) \to 0$, we have $\delta_\square(\ti{g}, \ti{f})=0$ and thus $\ti{f} \in \ti{\mathcal{B}}^\gamma$.

The proof of the second statement is identical, except that it starts from a sequence
  ${f}_n \in  {\mathcal{B}}^\gamma$  such that
 $\tau({f}_n) \geq t$ and ${f_n}$ converges to some ${f}\in \W$ in the metric $d_\square$.
\end{proof}

\begin{lemma}\label{step-functions}
Suppose $f \in \W$  is of the form $f=\sum_{i,j\in [k]}\beta_{ij}1_{Y_i}1_{Y_j}$ where $\beta_{ij}=\beta_{ji}\in [0,1]$
and $Y_1,\dots,Y_k$ form a partition of $[0,1]$ into measurable sets. Let $g\in\W$.  Then $\delta_\square(f,g)=0$ if and only if
there exists a partition of $[0,1]$ into measurable subsets $Y'_1,\dots Y_k'$ such that
$g=\sum_{i,j\in [k]}\beta_{ij}1_{Y'_i}1_{Y'_j}$ and $\lambda(Y_i)=\lambda(Y_i')$ almost everywhere.
\end{lemma}

\begin{proof}
If $g$ can be written as $\sum_{i,j\in [k]}\beta_{ij}1_{Y'_i}1_{Y'_j}$ with $\lambda(Y_i)=\lambda(Y_i')$ almost everywhere, then clearly $\delta_{\square}(f,g) = 0$. To establish the other direction,
we will use Theorem 8.6 (vi)  from \cite{jansongraphonbook}.  This will require us to turn $f$ into what is called a twin-free graphon, defined as a graphon $W$ such that there exists no pair $(x,x')\in [0,1]$
such that $W(x,\cdot)=W(x',\cdot)$ almost everywhere.  Unfortunately, by their very definition, step functions are not twin-free.  To remedy this, we introduce graphons over a general probability space $(\Omega,\mathcal F,\mu)$, defined as measurable functions $W:\Omega^2\to [0,1]$ such that $W(x,y)=W(y,x)$ for all $x,y\in \Omega$.  We also need to define a cut distance between graphons $W_i$ on (potentially) different probability spaces $(\Omega_i,\mathcal F_i,\mu_i)$, $i=1,2$.  It is defined as
$$
\delta_\square(W_1,W_2)=\inf_\mu \sup_{S,T\subset\Omega_1\times\Omega_2}\left|\int (W_1(x,x')-W_2(y,y'))d\mu(x,y)d\mu(x',y')\right|
$$
where the $\inf$ goes over all couplings of $\mu_1$ and $\mu_2$. It is easy to see that for graphons defined on $[0,1]$, this definition agrees with the previous one (see, e.g., Lemma 3.5 in \cite{Borgs2008}).  With these new definitions, we define two graphons $W_1$, $W_2$ over two possibly different probability spaces to be equivalent if $\delta_\square(W_1,W_2)=0$.

With this definition, the graphon $f$ is equivalent to the ``discrete'' graphon $f_{ij}'=\beta_{i,j}$ where $i$ and $j$ lie in the  probability space
$([k],2^{[k]},\mu)$ with $\mu(i)=\lambda(Y_i)$  (use the coupling which pairs $i\in [k]$ with the uniform measure on $Y_i$).
 It is  also easy to turn $f'$ into a twin free graphon as follows: if $i$ and $i'$ are twins, i.e., if the $i^\text{th}$ and $j^\text{th}$ row of $\beta$ are identical, just merge the sets $Y_i$ and $Y_j$
into a new set of measure $\lambda(Y_i)+\lambda(Y_j)$, reducing $k$ by one.  Note that this does not change the function $f$, just the representation of the form $f=\sum_{i,j\in [k]}\beta_{ij}1_{Y_i}1_{Y_j}$.   Iterating this procedure, we eventually obtain a twin free graphon $f'$ which  has cut-distance zero from $f$, $\delta_\square(f',f)=0$. 

At this point, we use Theorem 8.6 (vi)  from \cite{jansongraphonbook} which says that $\delta_\square(f',g)=0$ if and only if there exists a measure preserving map $\phi:[0,1]\to [k]$
such that $g(x,y)=f'_{\phi(x),\phi(y)}$ almost everywhere.  Defining $Y'_i=\phi^{-1}(\{i\})$ proves the lemma, in the case where the rows of $\beta$ are pairwise distinct. Otherwise, the intervals of $g$ may be split in order to match the representation of $f$.
\end{proof}

\subsection{Lower bound}\label{sec:lower}

In order to prove Statement (1) of \Cref{ldp}, we closely follow the proof of Theorem 2.3 in \cite{Chatterjee2011}.

\begin{proof}[Proof of \Cref{ldp}, Statement (1)]

We will prove the bound in the form given in Lemma~\ref{rem:LDP-compact}.
Let $f^{G_{kn}}$ be the empirical graphon of a graph on $kn$ vertices drawn according to $\pr_{kn,W_0}$. First, we claim that if
\begin{align} \label{goal 0} \liminf_{n \to \infty} \frac{1}{(kn)^2} \log{ \pr_{kn, W_0}\brac{d_\square(f^{G_{kn}},g) \leq \ve}} \geq - I_{W_0}(g)\end{align}
holds for all $g \in \mathcal{W}$ and $\ve>0$, then the theorem follows.

To see this, we first observe that for $\ti h\in \tW$, $0<\eta\leq\ve/2$ and $g \in B(\ti{h}, \eta)$
$$
 \ti{\mathbb{P}}_{kn, W_0}(S(\ti{h}, \ve))
 ={\mathbb{P}}_{kn, W_0}(\delta_\square(f^{G_{kn}},\ti h)\leq\ve)
\geq {\mathbb{P}}_{kn, W_0}(d_\square(f^{G_{kn}},g)\leq \ve/2)
$$
where the identity follows from the definition of  $\ti{\mathbb{P}}_{kn, W_0}$ and $S(\ti{h}, \ve)$,
and the lower bound follows upon noting that
$\delta_\square(f^{G_{kn}},\ti h)\leq d_\square(f^{G_{kn}},g)+\delta_\square(g,\ti h)$.
Therefore, assuming (\ref{goal 0}) yields
$$
\liminf_{n \to \infty} \frac{1}{(kn)^2} \log  \ti{\mathbb{P}}_{kn, W_0}(S(\ti{h}, \ve))
\geq - I_{W_0}(g)
$$
for all $0<\eta\leq \ve/2$ and all $g\in B(\ti{h},\ve/2)$.
It follows that
$$\liminf_{n \to \infty}  \frac{1}{(kn)^2}  \log\ti{\mathbb{P}}_{kn, W_0}(S(\ti{h}, \ve)) \geq -\sup_{\eta\in(0,\ve/2]}\inf_{g \in B(\ti{h}, \eta)} I_{W_0}(g)
= -\sup_{\eta>0}\inf_{g \in B(\ti{h}, \eta)} I_{W_0}(g)=-J_{W_0}(\ti{h})
$$
as required.

We have shown that (\ref{goal 0}) implies the theorem; we now turn to its proof. By \Cref{finite-or-not}, (\ref{goal 0}) holds trivially for $g \not \in \W_\Omega$. We may assume $g \in \W_\Omega$.
Let $\epsilon > 0$. We define $(g_n)_{n\geq 1}$, a sequence of $kn \times kn$ block graphons that approximate $g$. Recall \eqref{eq:k-n}.  For $i,j \in [kn]$, let
$$p_{ij}^{(n)}= (kn)^2 \int \int _{[\frac{i-1}{kn},\frac{i}{kn}] \times [\frac{j-1}{kn},\frac{j}{kn}]} g(x,y) dx dy
\qquad\text{and}\qquad
g_n(x,y)= p_{\vartheta_{kn}(x), \vartheta_{kn}(y)}^{(n)}.$$

Since $\|g_n - g\|_{1} \to 0$, in order to prove (\ref{goal 0}) it suffices to show that
$$\liminf_{n \to \infty} \frac{1}{(kn)^2} \log{ \pr_{kn, W_0}\brac{ B_{\ve,n}}} \geq - I_{W_0}(g),$$
where $B_{\ve,n}= \{f: d_\square (f, g_n) \leq \ve/2\}$.
We will apply the following proposition, which is proved as part of Theorem 2.3 in \cite{Chatterjee2011}. For completion, we include a proof sketch in the Appendix (\Cref{other-useful-results}).

\begin{proposition}\label{in thm 2.3}
\label{proposition:zero-probability}
Let $f_n$ be a graphon drawn from the measure $\pr_{kn, g_n}$, with $g_n$ as defined above.
For any $\epsilon > 0$, it holds that
\[\lim_{n \to \infty} \pr_{kn, g_n} ( d_\square(f_n, g_n) \geq \ve)= 0.\]
\end{proposition}

\noindent
We now apply a tilting argument to establish the lower bound, following \cite{Chatterjee2011}. Note that since $g\in \W_\Omega$, $g_n \in \W_\Omega$, and so the support of $\pr_{kn, g_n}$ is contained in the support of $\pr_{kn, W_0}$. Therefore we may write
\begin{align*}
\pr_{kn, W_0}(B_{\epsilon, n}) &= \int_{B_{\epsilon, n}} d\pr_{kn, W_0}= \int_{B_{\epsilon, n}} \exp{- \log \frac{d\pr_{kn, g_n}}{d \pr_{kn, W_0}}} d\pr_{kn, g_n}\\
&= \pr_{kn, g_n}(B_{\epsilon,n}) \frac{1}{\pr_{kn, g_n} (B_{\epsilon, n})} \int_{B_{\epsilon, n}} \exp{- \log \frac{d\pr_{kn, g_n}}{d \pr_{kn, W_0}}} d\pr_{kn, g_n}.
\end{align*}
\Cref{proposition:zero-probability} implies that for $n$ sufficiently large, $\pr_{kn, W_0}(B_{\epsilon, n})>0$.
Taking the logarithm of both sides and applying Jensen's inequality, we obtain
\begin{align*}
\log \pr_{kn, W_0}(B_{\epsilon, n}) \geq \log \pr_{kn, g_n}(B_{\epsilon, n}) - \frac{1}{\pr_{kn, g_n}(B_{\epsilon, n})} \int_{B_{\epsilon, n}} \log \left(\frac{d \pr_{kn, g_n}}{d \pr_{kn, W_0}}\right) d \pr_{kn, g_n}.
\end{align*}
By Proposition \ref{proposition:zero-probability}, it holds that $\pr_{kn,g_n}(B_{\epsilon, n}) \to 1$. Furthermore, letting $p_{\text{min}}$ and $p_{\text{max}}$ be the minimal and maximal values taken on by $W_0$ in $(0,1)$, we have 
\[\frac{1}{(kn)^2}\int_{\mathcal{W}_\Omega\setminus B_{\ve, n}}\log \left(\frac{d \pr_{kn, g_n}}{d \pr_{kn, W_0}}\right) d \pr_{kn, g_n} \leq \frac{1}{(kn)^2} \int_{\mathcal{W}_\Omega\setminus B_{\ve, n}} \log\left(\left(\frac{1}{\min\{p_{\text{min}}, 1 - p_{\text{max}}\}}\right)^{\binom{kn}{2}}\right) d \pr_{kn, g_n} \to 0. \] Therefore,
\begin{align*}
\liminf_{n \to \infty} \frac{1}{(kn)^2}\log \pr_{kn, W_0}(B_{\epsilon, n}) &\geq - \limsup_{n \to \infty} \frac{1}{(kn)^2} \int  \log \left(\frac{d \pr_{kn, g_n}}{d \pr_{kn, W_0}}\right) d \pr_{kn, g_n}.
\end{align*}
Observe that the probability of sampling a graph $G$ on $kn$ vertices is
$$
\pr_{kn, g_n}(\{G\})=\prod_{\substack{i,j \in [kn]\\i<j}} \brac{ p_{ij}^{(n)}\1_{(i,j) \in E(G) } + (1-p_{ij}^{(n)})\1_{(i,j) \not \in E(G)}}.
$$
By construction if $W_0(i/(kn), j/(kn))=0$, then $p_{ij}^{(n)}=0$ and if $W_0(i/(kn), j/(kn))=1$, then $p_{ij}^{(n)}=1$. Let $V(G)$ denote the vertex set of $G$. Recalling the convention that $0 \log(0/0)=0$, we obtain 
\begin{align*}
 & \limsup_{n \to \infty} \frac{1}{(kn)^2} \int\log{ \left( \frac{d\pr_{kn, g_n}}{d\pr_{kn,W_0}}\right)} d\pr_{kn,g_n} \\
 &=
 \limsup_{n \to \infty}
 \frac{1}{(kn)^2} \sum_{G: V(G)=[kn]} \pr_{kn, g_n}(\{G\})\log \bfrac{ \pr_{kn, g_n}(\{G\})}{\pr_{kn, W_0}(\{G\})}\\
 &=
 \limsup_{n \to \infty} \frac{1}{(kn)^2} \sum_{G: V(G)=[kn]} \prod_{\substack{i,j \in [kn]\\i<j}} \brac{ p_{ij}^{(n)}\1_{(i,j) \in E(G) } + (1-p_{ij}^{(n)})\1_{(i,j) \not \in E(G)}} \times \\
  &~~~~~ \brac{\sum_{\substack{i,j \in [kn]\\i<j}} \log{ \frac{p_{ij}^{(n)}}{W_0(\frac{i}{kn},\frac{j}{kn})}} \1\{(i,j) \in E(G)\} +\log{ \frac{1-p_{ij}^{(n)}}{1-W_0(\frac{i}{kn},\frac{j}{kn})}} \1\{(i,j) \not\in E(G)\}} \\
  &=\limsup_{n \to \infty} \frac{1}{(kn)^2} \sum_{\substack{i,j \in [kn]\\i<j}} p_{ij}^{(n)}\log{ \frac{p_{ij}^{(n)}}{W_0(\frac{i}{kn},\frac{j}{kn})}} +(1-p_{ij}^{(n)})\log{ \frac{1-p_{ij}^{(n)}}{1-W_0(\frac{i}{kn},\frac{j}{kn})}}\\
  &= \limsup_{n \to \infty}I_{W_0}(g_n) = I_{W_0}(g). 
\end{align*}
The last equality follows from a straightforward adaptation of an argument from  \cite[Lemma 5.7]{Chatterjee2015}: for any $\varepsilon>0$, set  $A_{\varepsilon}=\{(x,y): \|g_n(x,y)-g(x,y)\|_1>\varepsilon\}$.  Observe that $\Vert g_n - g\Vert_1 \geq \lambda(A_{\ve}) \epsilon$, implying $\lambda(A_{\ve}) \to 0$ for any $\varepsilon>0$.
In addition, note that $Im(W_0)$ is finite and $h_p(\cdot)$ is uniformly continuous for any fixed $p$. Therefore, for every $\delta>0$, there exists $\varepsilon>0$ such that $|h_{W_0(x,y)}(g_n(x,y)) - h_{W_0(x,y)}(g(x,y))|<\delta$ whenever $|g_n(x,y) - g(x,y)| < \varepsilon$. Consequently, 
\begin{align}
    |I_{W_0}(g_n) - I_{W_0}(g)| &\leq  \frac{1}{2} \int_{A_{\varepsilon}} |h_{W_0(x,y)}(g_n(x,y)) - h_{W_0(x,y)}(g(x,y))| dx dy \nonumber \\
    &+ \frac{1}{2} \int_{[0,1]^2\backslash A_{\varepsilon}} |h_{W_0(x,y)}(g_n(x,y)) - h_{W_0(x,y)}(g(x,y))| dx dy \nonumber \\
    &\leq \frac{C}{2} \lambda(A_{\varepsilon}) + \frac{1}{2}\delta \nonumber 
\end{align}
where $C>0$ is such that $h_{W_0(x,y)}(a)<C$ for all $a\in[0,1]$ and $(x,y) \in [0,1]^2$. This completes the proof. 
\end{proof}

\subsection{Upper bound}\label{sec:upper}
In this section we prove the upper bound of \Cref{ldp}.
The proof requires two key lemmas.
The first one establishes that as long as we look at balls around block constant graphons, we can restrict to a finite number of measure preserving bijections, and the second one gives an upper bound on the probability that sampling a graphon and applying an invertible transformation yields a graphon in a particular $d_\square$ ball.
This is formalized in the following two lemmas. Let $\mathcal{M}_{kn}$ be the set of all permutations of $[0,1]$ corresponding to permuting the vertices of a graph on $kn$ vertices, i.e. if $\sigma \in \mathcal{M}_{kn}$, then for all $i \in [kn]$, the image of the interval $\vartheta_{kn}^{-1}(i)$ under $\sigma$ is an interval of the form $\vartheta_{kn}^{-1}(j)$ for some $j \in [kn]$.

\begin{lemma}\label{approx}
Let $s_h,s_w \in \mathbb{Z}^+$ and $\eta>0$.  Given $\mu \in \Delta_{s_h}$, $\gamma \in \Delta_{s_w}$,
$h \in \mathcal{B}^{\mu}$,
and a $k$-uniform block graphon $W_0 \in \mathcal{B}^\gamma$, there exists a finite set of invertible measure preserving transformations $T \subseteq \mathcal{M}$ with $|T|=N(\eta,s_w, s_h)$  such that the following holds for all
$n>12s_h^2 s_w /(k\eta)$:
For all $\sigma \in \mathcal{M}_{kn}$ there exists $\tau \in T$ such that for all $\ve >0$
 $$\pr_{kn,W_0}\brac{ d_\square(f^{G_{kn}^\sigma},h) \leq \ve} \leq \pr_{kn,W_0}\brac{ d_\square(f^{G_{kn}^\tau},h) \leq \ve + \eta},
 $$
where $f^{G_{kn}}$ is the empirical graphon obtained by sampling $W_0$ according to $\pr_{kn,W_0}$.
\end{lemma}

\begin{lemma}\label{bound on one} Let $W_0$ be a uniform $k$-block graphon, $\ve>0$ and $h \in \mathcal{W}$.
 Let $G_{kn}$ be the graph drawn from $\mathbb{P}_{kn, W_0}$, and let $f^{G_{kn}}$ denote the corresponding empirical graphon. For all invertible $\tau \in \mathcal{M}$, it holds that
\[ \limsup_{n \to \infty} \frac{1}{(kn)^2} \log \mathbb{P}_{kn, W_0} \left(d_\square\left (f^{G_{kn}^\tau},h \right) \leq \ve  \right) \leq  -\inf_{f: \delta_\square(f, h) \leq \ve } I_{W_0}(f).\]
\end{lemma}

To prove the second lemma,
we use the following LDP upper bound with respect to the weak topology.
Recall that the weak topology on $\W$ is the smallest topology under which the maps $f \mapsto \int_{[0,1]^2} f(x,y) g(x,y) dx dy$ are continuous for every $g \in L^2([0,1]^2)$.

\begin{theorem} \label{weak ldp}
Let $W_0$ be a uniform $k$-block graphon. For every weakly closed set $F \in \W$,
$$\limsup_{n \to \infty } \frac{1}{(kn)^2} \log \pr_{kn, W_0}\brac{ F} \leq - \inf_{f \in F} I_{W_0}(f).$$
\end{theorem}

The proof of \Cref{weak ldp} is a straightforward generalization of Theorem 5.1 of \cite{Chatterjee2015}. For completeness we include the proof in the Appendix. We delay the proofs of \Cref{approx,bound on one} to the end of the section and proceed to the proof of the upper bound in \Cref{ldp}.

In the proof of the upper bound in \Cref{ldp}, we also use the following version of the Weak Regularity Lemma, which follows directly from  Theorem 3.1 in \cite{Chatterjee2015}.

\begin{lemma}[Weak Regularity Lemma]\label{lem:weak-reg}
Given $\ve>0$ there exists a finite set $H(\ve)\subset\W$ of block-graphons, such that if $f$ is a uniform $n$-block graphon
there exists $\sigma \in \mathcal{M}_{n}$ and an $h\in H(\ve)$ such that
$$
d_\square(f^\sigma,h)<\ve. 
$$
We call such a set $H(\ve)\subset\W$ an $\ve$-net. 
\end{lemma}

\begin{proof}[Proof of \Cref{ldp}, Statement 2] We prove the bound in the form given in \Cref{rem:LDP-compact}. Fix $\ti{g}$ and $\alpha>0$. Let $\ve, \eta < \alpha/2$, let $H(\ve/2)$ be an $\ve/2$ net as given in \Cref{lem:weak-reg},
and let $T^*$ be the union of the sets of invertible transformations $T$ given by \Cref{approx} for each $h \in H(\ve/2)$. We index the finite set as $T^*=\{\tau_1, \tau_2, \dots \tau_{N(\eta, W_0, \ve)}\}$. Then

\begin{align*}
    \ti{\pr}_{kn,W_0}\brac{S(\ti{g}, \alpha)}
     &= \pr_{kn,W_0}\brac{ \delta_\square(f^{G_{kn}} ,g )\leq \alpha}
     \\
    &\leq \sum_{h \in H(\ve/2)}
    \pr_{kn,W_0}\brac{ \bigcup_{\sigma \in \mathcal{M}_{kn}}
        \left\{f^{G_{kn}^\sigma} \in \{ f: \delta_\square(  f, g ) \leq \alpha \} \cap \left\{ f : d_\square(f,h) \leq \frac\ve 2\right\}\right\}}
    \\
    &\leq \sum_{h \in H(\ve/2)}\sum_{\sigma \in \mathcal{M}_{kn}}
    \pr_{kn,W_0}\brac{\left\{f^{G_{kn}^\sigma} \in \{ f: \delta_\square(  f, g ) \leq \alpha \} \cap \left\{ f : d_\square(f,h) \leq \frac \ve 2\right\}\right\}}
    \\
      &\leq \sum_{\substack{h \in H(\ve/2)\\\delta_\square(g,h) \leq \alpha +\ve/2 } }\sum_{\sigma \in \mathcal{M}_{kn} }
      \pr_{kn,W_0}\brac{ d_\square(f^{G_{kn}^\sigma},h) \leq \frac \ve 2}
      \\
      & \leq (kn)!|H(\ve/2)| \max_{\substack{h \in H(\ve /2)\\\delta_\square(g,h) \leq \alpha +\ve/2 \\ i \in [N(\eta, W_0, \ve)]}}
      \pr_{kn,W_0}\brac{ d_\square(f^{G_{kn}^{\tau_i}},h) \leq \frac \ve 2 + \eta},
\end{align*}
where
the last inequality follows from \Cref{approx}.
\Cref{bound on one} implies that
\begin{align*}
\limsup_{n \to \infty} \frac{1}{(kn)^2} \log
\pr_{kn,W_0}\brac{ d_\square(f^{G_{kn}^{\tau_i}},h) \leq \frac\ve 2 + \eta}
&\leq -\inf_{f: \delta_\square(f,h) \leq \frac \ve 2 +\eta}I_{W_0}(f)
\\
&\leq - \inf_{f: \delta_\square(f,\ti{g}) \leq \alpha +\ve +\eta} I_{W_0}(f),
\end{align*}
when $h$ is such that $\delta_\square(g,h) \leq \alpha +\ve/2$ and $\tau_i$ is invertible.

Let $M \in \mathbb{Z}^+$, and consider the $M$ sequences $\{y_n^i\}_{n \geq 1}$ where $i \in [M]$. 
A basic analysis argument implies that if $\limsup_{n \to \infty} \frac 1{(kn)^2}\log y_n^{i} \leq C$ for all $i \in [M]$, then $\limsup_{n \to \infty} \frac 1{(kn)^2} \log\brac{ \max_{i \in [M]} y^{i}_n} \leq C$. 
It follows that
\begin{align*}
\limsup_{n \to \infty} &\frac{1}{(kn)^2} \log  \ti{\pr}_{kn,W_0}\brac{S(\ti{g}, \alpha)} \\
&\leq \limsup_{n \to \infty} \frac{1}{(kn)^2} \log \brac{ (kn)! 
 |H(\ve/2)| \max_{\substack{h \in H(\ve/2)\\\delta_\square(g,h) \leq \alpha +\ve/2 \\ i \in [N(\eta, W_0, \ve)]}}\pr_{kn,W_0}\brac{ d_\square(f^{G_{kn}^{\tau_i}},h) \leq \ve/2 + \eta}}\\
&= \limsup_{n \to \infty} \frac{1}{(kn)^2} \log \brac{  \max_{\substack{h \in H(\ve/2)\\\delta_\square(g,h) \leq \alpha +\ve/2 \\ i \in [N(\eta, W_0, \ve)]}}\pr_{kn,W_0}\brac{ d_\square(f^{G_{kn}^{\tau_i}},h) \leq \ve/2 + \eta}}\\
&\leq - \inf_{f: \delta_\square(f,\ti{g}) \leq \alpha +\ve +\eta} I_{W_0}(f).
\end{align*}
Since $\ve , \eta < \alpha/2$, we obtain
$$\limsup_{n \to \infty} \frac{1}{(kn)^2} \log  \ti{\pr}_{kn,W_0}\brac{S(\ti{g}, \alpha)}\leq - \inf_{f: \delta_\square(f,\ti{g}) \leq 2\alpha } I_{W_0}(f).$$
Since $(- \inf_{f: \delta_\square(f,\ti{g}) \leq 2\alpha} I_{W_0}(f))$ is a non-increasing function as $\alpha \to 0$,
\begin{align*}\lim_{\alpha\to 0} \limsup_{n \to \infty} \frac{1}{(kn)^2} \log  \ti{\pr}_{kn,W_0}\brac{S(\ti{g}, \alpha)} &\leq \inf_{\alpha >0 } \brac{- \inf_{f: \delta_\square(f,\ti{g})
\leq 2\alpha} I_{W_0}(f)}\\
&= - \sup_{\alpha >0 }  \inf_{f: \delta_\square(f,\ti{g}) \leq 2\alpha } I_{W_0}(f)\\
&= - J_{W_0} (g),
\end{align*}
as required.
\end{proof}

We will use the following definition and proposition
in the proof of \Cref{approx}.

\begin{definition}\label{defn-respectful}
Let $\phi \in \mathcal{M}$. Let $I_1 \dots, I_k$ be a partition of $[0,1]$. We say that $\phi$ \emph{respects the interval structure} of $I_1, \dots, I_k$ if for all $j$ and $X \subseteq I_j$,  $\phi(X) \subseteq I_j$.
\end{definition}

\begin{proposition}\label{respectful}
Let $h$ be a graphon that is constant on each block $I_i \times I_j$ for $i,j \in [k]$. If $\phi$ is invertible and respects the interval structure of $I_1 \dots I_k$, then for all $g \in \W$
$$d_\square(g^\phi, h) = d_\square(g,h).$$
\end{proposition}

\begin{proof}Note that since $h$ is constant on each block $I_i \times I_j$ and $\phi$ respects the interval structure, $$h^\phi(x,y)= h(\phi(x), \phi(y))= h(x,y).$$ It follows that
$$d_\square(g,h)=d_\square(g^\phi, h^\phi)= d_\square(g^\phi, h).$$
\end{proof}

\begin{proof}[Proof of \Cref{approx}.]
For ease of notation, we denote $f=f^{G_{kn}}$ as the empirical graphon obtained by sampling $W_0$ according to $\pr_{kn,W_0}$. Let $\{H_i: i \in [s_h]\}$ and $\{W_{0,j}: j \in [s_w]\}$ denote the intervals of the block structure governed by $\mu$ and $\gamma$ respectively;  formally, $H_1 = [0, \mu_1]$, $W_{0,1} = [0, \gamma_1]$,
 $$H_i =\left(\sum_{j=1}^{i-1} \mu_j, \sum_{j=1}^{i} \mu_j\right] \text{ for } 2 \leq i \leq s_h, \quad  \text{ and } \quad W_{0,i} =\left(\sum_{j=1}^{i-1} \gamma_j, \sum_{j=1}^{i} \gamma_j\right] \text{ for } 2 \leq i \leq s_w.$$
 $\{H_i\}$ are the intervals corresponding to the blocks of $h$ and $\{W_{0,i}\}$ are the intervals corresponding to the blocks of $W_0$.

We begin with a proof outline. First we construct a finite set of measure preserving transformations $T \subseteq \mathcal{M}$ such that each $\sigma \in \mathcal{M}_{kn}$ is ``close'' to a transformation $\tau$ in $T$.  Then, in order to compare $f^{\sigma}$ and $f^{\tau}$, we define $\alpha, \beta$ invertible transformations that respect the intervals $H_1, \dots H_{s_h}$. \Cref{respectful} implies that $d_\square(f^\sigma,h) = d_\square((f^{\sigma})^\alpha,h)$ and $d_\square(f^{\tau},h) = d_\square((f^{\tau})^\beta,h)$.
It follows that
\begin{align}
    \pr_{kn,W_0}\brac{ d_\square(f^\sigma,h) \leq \ve}&= \pr_{kn,W_0}\brac{ d_\square(\brac{f^\sigma}^\alpha,h) \leq \ve},\label{equality1}\\
    \pr_{kn,W_0}\brac{ d_\square(f^{\tau},h) \leq \ve + \eta}&=\pr_{kn,W_0}\brac{ d_\square(\brac{f^{\tau}}^\beta,h) \leq \ve + \eta}.\label{equality2}
\end{align}

Finally, we will describe a coupling of $f$ and $g$, each with marginal distribution  $\pr_{kn,W_0}$, that guarantees \begin{equation}\label{c-eq} d_\square\brac{ \brac{f^\sigma}^\alpha ,\brac{g^{\tau}}^\beta } \leq \eta\end{equation}
with probability one.
The triangle inequality implies
\begin{align}\label{middle inequality}
    \pr_{kn,W_0}\brac{ d_\square(\brac{f^\sigma}^\alpha,h) \leq \ve}\leq \pr_{kn,W_0}\brac{ d_\square(\brac{f^{\tau}}^\beta,h) \leq \ve + \eta}.
\end{align}
The statement follows directly from \eqref{equality1}, \eqref{equality2}, and \eqref{middle inequality}.

To complete the proof according to this outline, we must complete the following tasks:
\begin{enumerate}[(a)]
    \item Define a finite net of measure preserving transformations $T \subseteq \mathcal{M}$ with $|T|=N(\eta, s_h, s_w)$.
    \item For each transformation  $\sigma\in\mathcal{M}_{kn}$, define a ``close'' transformation $\tau\in T$. Informally, we will say that two transformations are close if they map approximately the same amount of mass from $W_{0,i}$ to $H_j$ for all $i \in [s_w]$ and $j \in [s_h]$.
    \item Define $\alpha$ and $\beta$, invertible transformations that respect the intervals $H_1 \dots H_{s_h}$.
    \item Exhibit a coupling of $f$ and $g$ each sampled according to $\pr_{kn,W_0}$ that guarantees \eqref{c-eq}.
\end{enumerate}
We begin with (a). For convenience we index vectors $v \in \R^{s_w\cdot s_h}$ by pairs $(i,j) \in [s_w] \times [s_h]$. Let $$V= \bigg\{ v \in\R_{\geq 0}^{s_w\cdot s_h} \bigg| \sum_{i=1}^{s_w} \sum_{j=1}^{s_h} v_{ij}=1, \sum_{i=1}^{s_w} v_{ij}= \mu_j, \sum_{j=1}^{s_h} v_{ij}= \gamma_i   \bigg\}.$$ Recall that $\{H_i\}$ are the intervals corresponding to the blocks of $h$ and $\{W_{0,i}\}$ are the intervals corresponding to the blocks of $W_0$. For each $v \in V$ we associate an invertible measure preserving transformation $\tau_v \in \mathcal{M}$ that maps an interval of length $v_{ij}$ contained in $W_{0,i}$ to an interval that is contained completely in $H_j$ for each $(i,j)$. To this end, let
$$I_{11} = \left[0, v_{11} \right] \qquad \text{and}\qquad I_{ij}=\left( \sum_{a=1}^{i-1} \gamma_a  + \sum_{b=1}^{j-1} v_{ib},\sum_{a=1}^{i-1} \gamma_a  + \sum_{b=1}^{j} v_{ib} \right], (i,j) \neq (1,1).$$
 Note that the intervals $I_{ij}$ are sorted first by the first index, then by the second index. Define $\tau:[0,1]\to [0,1]$ to be the transformation that translates the intervals so they are first sorted by the second index, then by the first index,
$$\tau(I_{11}) = [0, v_{11}] \qquad \text{and} \qquad \tau(I_{ij}) =\left( \sum_{b=1}^{j-1} \mu_b  + \sum_{a=1}^{i-1} v_{aj},\sum_{b=1}^{j-1} \mu_b  + \sum_{a=1}^{i} v_{aj} \right], (i,j) \neq (1,1).$$
For an illustration of this transformation,  see \Cref{fig:transformation}.
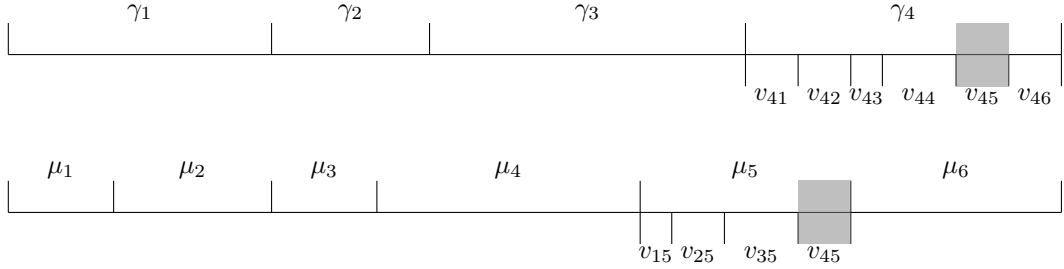
\begin{figure}[h]
    \centering
\begin{tikzpicture}[scale=1.4]
\draw (0,0) -- (10,0);
\draw (0,0) -- (0,0.3);
\draw (2.5,0) -- (2.5,0.3);
\draw (4,0) -- (4,0.3);
\draw (7,0) -- (7,0.3);
\draw (10,0) -- (10,0.3);
\draw (7,0) -- (7,-0.3);
\draw (7.5,0) -- (7.5,-0.3);
\draw (8,0) -- (8,-0.3);
\draw (8.3,0) -- (8.3,-0.3);
\draw (9,0) -- (9,-0.3);
\draw (9.5,0) -- (9.5,-0.3);
\draw (10,0) -- (10,-0.3);
\node at (1.25,0.4) {$\gamma_1$};
\node at (3.25,0.4) {$\gamma_2$};
\node at (5.5,0.4) {$\gamma_3$};
\node at (8.5,0.4) {$\gamma_4$};
\node at (7.25,-0.4) {$v_{41}$};
\node at (7.75,-0.4) {$v_{42}$};
\node at (8.15,-0.4) {$v_{43}$};
\node at (8.65,-0.4) {$v_{44}$};
\node at (9.25,-0.4) {$v_{45}$};
\node at (9.75,-0.4) {$v_{46}$};
\draw[opacity=0, fill=gray, fill opacity = 0.5] (9,-0.3) rectangle (9.5,0.3);
\draw (0,-1.5) -- (10,-1.5);
\draw (0,-1.5) -- (0,-1.2);
\draw (1,-1.5) -- (1,-1.2);
\draw (2.5,-1.5) -- (2.5,-1.2);
\draw (3.5,-1.5) -- (3.5,-1.2);
\draw (6,-1.5) -- (6,-1.2);
\draw (8,-1.5) -- (8,-1.2);
\draw (10,-1.5) -- (10,-1.2);
\draw (6,-1.5) -- (6,-1.8);
\draw (6.3,-1.5) -- (6.3,-1.8);
\draw (6.8,-1.5) -- (6.8,-1.8);
\draw (7.5,-1.5) -- (7.5,-1.8);
\draw (8,-1.5) -- (8,-1.8);
\node at (0.5,-1.1) {$\mu_1$};
\node at (1.75,-1.1) {$\mu_2$};
\node at (3,-1.1) {$\mu_3$};
\node at (4.75,-1.1) {$\mu_4$};
\node at (7,-1.1) {$\mu_5$};
\node at (9,-1.1) {$\mu_6$};
\node at (6.15,-1.9) {$v_{15}$};
\node at (6.55,-1.9) {$v_{25}$};
\node at (7.15,-1.9) {$v_{35}$};
\node at (7.75,-1.9) {$v_{45}$};
\draw[opacity = 0, fill=gray, fill opacity = 0.5] (7.5,-1.8) rectangle (8,-1.2);
\end{tikzpicture}
\caption{Illustration of the transformation $I_{45} \mapsto \tau(I_{45})$.}
\label{fig:transformation}
\end{figure}

Observe that $I_{ij} \subseteq W_{0,i}$ and $\tau(I_{ij}) \subseteq H_j$.
Note that $V$ is a compact set, and thus we can construct a finite net $V^* \subseteq V$ such that for all $v \in V$, there exists $v^* \in V^*$ such that $\Vert v- v^* \Vert_\infty \leq \eta/(8s_ws_h)$.  Let $T= \{\tau_v | v \in V^*\} \cup \{\tau_v^{-1} | v \in V^*\} $.

\medskip

Next we address (b). Let $M= \{\bar{m}_1, \dots \bar{m}_{kn}\}$ be the set of intervals corresponding to vertices in an empirical graphon with $kn$ vertices,  where $\bar{m}_1 = \left[0, \frac{1}{kn}\right]$ and
$$\bar{m}_i= \left( \frac{i-1}{kn}, \frac{i}{kn}\right] \text{ for } 2 \leq i \leq kn.$$
We call the intervals in $M$ ``vertex intervals''.

For each transformation in $\mathcal{M}_{kn}$, we find a transformation in $T$ that moves roughly the same amount of mass between intervals $W_{0,i}$ and $H_j$ for all $i \in [s_w]$ and $j \in [s_h]$.
We construct an element in $T$ that is close to  $\sigma^{-1}$ rather than $\sigma$ to make the construction in the next section more convenient. (Note that both $\sigma$ and $\sigma^{-1}$ are in $\mathcal{M}_{kn}$.)

Let $N_{ij}$ be the set of vertex intervals that are mapped from $W_{0,i}$ to $H_j$ under $\sigma^{-1}$,
$$
N_{ij}= \{ \bar{m}_\ell \in M | \bar{m}_\ell \subseteq W_{0,i} \text{ and } \sigma(\bar{m}_\ell) \subseteq H_j\},
$$
and let $n_{ij} =|N_{ij}|$.

Each vertex interval is contained in some $W_{0,i}$ by construction. Since $\sigma^{-1}$ maps at most $s_h -1$ vertex intervals  to the boundary between intervals of $h$,
$$\sum_{i=1}^{s_w}\sum_{j=1}^{s_h} n_{ij} \geq kn - s_h+1.$$
Define $v \in \R^{s_w\cdot s_h}$, $v_{ij}= n_{ij}/(kn)$.
We claim that there exists $v' \in V$ with $v'_{ij}\geq v_{ij}$ for all $i,j$ such that
	$$
	\Vert v - v' \Vert_\infty\leq \frac{s_h-1}{kn}< \frac{\eta}{8s_w s_h}.
	$$
	To see this, let $\tilde\mu_j= \sum_{i=1}^{s_w} v_{ij}$ and $\tilde \gamma_i=\sum_{j=1}^{s_h} v_{ij}$.
	Then $\tilde\mu_j\leq \mu_j$ and $\tilde\gamma_i\leq\gamma_i$ and
	$$
	\sum_i(\gamma_i-\tilde\gamma_i)=\sum_j(\mu_j-\tilde\mu_j)=1-\sum_{i=1}^{s_w} \sum_{j=1}^{s_h} v_{ij}=:\Delta,
	\qquad \Delta\leq  \frac{s_h-1}{kn}.
	$$
	Taking a coupling $\zeta_{ij}$ of the probability distributions $\left(\frac 1\Delta(\gamma_i-\tilde\gamma_i)\right)_{i\in[s_w]}$
	and $\left(\frac 1\Delta(\mu_j-\tilde\mu_j)\right)_{j\in[s_h]}$
	and setting $v_{ij}'=v_{ij}+\zeta_{ij}\Delta$ proves the claim. (Observe that  $n\geq 8{s_h}^2 s_w /(k\eta)$ by the assumption of the lemma.)

Choosing
$v^* \in V^*$ such that $\Vert v' - v^* \Vert_\infty \leq \eta/(8s_w s_h)$, we then have
$$\Vert v- v^* \Vert_\infty < \frac{\eta}{4s_ws_h}.$$
 We associate $\sigma^{-1}$ with $\tau_{v^*}\in T$.
Let $\tau= \tau_{v^*}^{-1}$, and note $\tau \in T$.

\medskip
Next we address (c), defining the transformations $\alpha$ and $\beta$. We must define $\alpha$ and $\beta$ in a way that conveniently facilitates a coupling satisfying \eqref{c-eq} in step (d). In particular, we will define a coupling so that $\brac{f^\sigma}^\alpha$ and $\brac{g^{\tau}}^\beta $ are identical on many sets of the form $\bar{m} \times \bar{m}'$ where $\bar{m}$ and $\bar{m}'$ are vertex intervals. We can exactly couple the values  $f(\sigma(\alpha(x)),\sigma(\alpha(y)))$ and $g(\tau(\beta(x)),\tau(\beta(y)))$ on $\bar{m} \times \bar{m}'$, provided $\sigma(\alpha(x))$ and $\tau(\beta(x))$ are in the same interval $W_{0,i}$ and $\sigma(\alpha(y))$ and $\tau(\beta(y))$ are in the same interval $W_{0,j}$. In this case, on both $\bar{m}$ and $\bar{m}'$ $f(\sigma(\alpha(x)),\sigma(\alpha(y))), g(\tau(\beta(x)),\tau(\beta(y))) \sim Bern(p_{ij})$ where $p_{ij}$ is the value of $W_0$ on $W_{0,i} \times W_{0,j}$, and thus they can be coupled.

We say that a vertex interval $\bar{m} \in M$ is ``synchronized'' if  $\bar{m} \subseteq H_j$ for some $j \in [s_h]$, and  $\sigma (\alpha(\bar{m}))$ and $\tau ( \beta(\bar{m}))$ are vertex intervals that belong to the same interval $W_{0,i}$ for some $i \in [s_w]$.
We construct $\alpha$ and $\beta$ so that at least a $(1-\eta/2)$ fraction of the vertex intervals $\bar{m} \in M$ are synchronized. In step (d), we will couple the behavior of the vertices corresponding to $\sigma (\alpha(\bar{m}))$ and $\tau ( \beta(\bar{m}))$ for each synchronized vertex interval $\bar{m}$.

Let $n_{ij}, I_{ij}, v^*$ be as defined in parts (a) and (b).
Let $k_{ij}= \min \{ n_{ij}, \lfloor v^*_{ij}kn \rfloor -1 \}$. We will construct $\alpha$ and $\beta$ so that there are $k_{ij}$ synchronized vertex intervals contained in $H_j$ whose images under
$(\sigma \circ \alpha)$ and $(\tau\circ \beta)$ are contained in $W_{0,i}$.

The transformations $\sigma^{-1}$ and $\tau^{-1}$  map approximately the same amount of mass from $W_{0,i}$ to $H_j$ for all $i \in [s_w]$ and $j \in [s_h]$, but the intersection of the image of $W_{0,i}$ and $H_j$ may be be very different under the two maps.
We design $\alpha^{-1}$ and $\beta^{-1}$ so that $\alpha^{-1} \circ \sigma^{-1}$ and $\beta^{-1} \circ \tau^{-1}$ both map mass from $W_{0,i}$ to the same subinterval of $H_j$. Working with the inverse functions allows us to think of $\alpha^{-1}$ and $\beta^{-1}$ as functions that reorganize the images of $W_{0,i}$ under $\sigma^{-1}$  and $\tau^{-1}$ (respectively) within each interval $H_j$. We now formally construct $\alpha$ and $\beta$ by constructing their inverses.

First we construct $\alpha^{-1}$, as illustrated in \Cref{fig:alpha}. There are $n_{ij}$ vertex intervals contained in $W_{0,i}$ that are mapped to vertex intervals in $H_j$ under $\sigma^{-1}$.  Informally, $\alpha^{-1}$ will rearrange the images of these vertex intervals within $H_j$ by sorting them by their origin interval $W_{0,i}$. Under $\alpha^{-1}$, the image of vertex intervals originating in {$W_{0,1}$}
 are mapped to the leftmost vertex intervals contained completely in $H_j$.

 Formally, let $a^{ij}_1,a^{ij}_2 \dots a^{ij}_{k_{ij}}$ enumerate $k_{ij}$ of the $n_{ij}$ vertex intervals contained in $W_{0,i}$ that are mapped to vertex intervals in $H_j$ under $\sigma^{-1}$.
Let $m_j(i)$ be the $i^{th}$ interval of $M$ that is entirely contained in $H_j$. Define $\alpha \in \mathcal{M}$ so that $\alpha^{-1}$ translates the interval $\sigma^{-1}(a_\ell^{ij})$ to the vertex interval specified as follows:
$$\alpha^{-1}(\sigma^{-1}(a_\ell^{ij}))=m_j\brac{\sum_{z=1}^{i-1} k_{zj} + \ell},$$
and $H_j\setminus \brac{\bigcup_{i=1}^{s_w}\bigcup_{\ell=1}^{k_{ij}}\sigma^{-1}(a^{ij}_\ell)}$ maps to $H_j\setminus \brac{\bigcup_{i=1}^{s_w}\bigcup_{\ell=1}^{k_{ij}}\alpha^{-1}( \sigma^{-1}(a^{ij}_\ell))}$ under $\alpha^{-1}$ in any invertible manner. Since $\alpha^{-1}(H_i)=H_i$, $\alpha$ and $\alpha^{-1}$ respect the intervals $H_1, \dots H_{s_h}$.

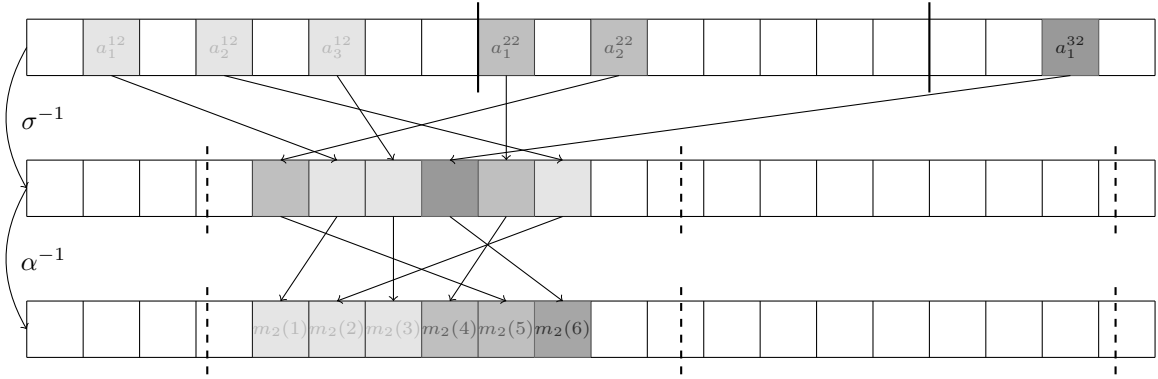
\begin{figure}
    \centering
    \begin{tikzpicture}[scale=0.75]
    \draw (0,0) -- (20,0);
    \draw (0,-1) -- (20,-1);
    \foreach \i in {0,...,20}{
    \draw (\i,0) -- (\i,-1);
    }
    \draw (0,-2.5) -- (20,-2.5);
    \draw (0,-3.5) -- (20,-3.5);
    \foreach \i in {0,...,20}{
    \draw (\i,-2.5) -- (\i,-3.5);
    }
    \draw (0,-5) -- (20,-5);
    \draw (0,-6) -- (20,-6);
    \foreach \i in {0,...,20}{
    \draw (\i,-5) -- (\i,-6);
    }
    \fill [fill=gray,fill opacity=0.2,text opacity=1] (1,-1) rectangle (2,0) node[pos=.5] {\scriptsize{$a_1^{12}$}};
    \fill [fill=gray,fill opacity=0.2,text opacity=1] (3,-1) rectangle (4,0) node[pos=.5] {\scriptsize{$a_2^{12}$}};
    \fill [fill=gray,fill opacity=0.2,text opacity=1] (5,-1) rectangle (6,0) node[pos=.5] {\scriptsize{$a_3^{12}$}};
    \fill [fill=gray,fill opacity=0.5,text opacity=1] (8,-1) rectangle (9,0) node[pos=.5] {\scriptsize{$a_1^{22}$}};
    \fill [fill=gray,fill opacity=0.5,text opacity=1] (10,-1) rectangle (11,0) node[pos=.5] {\scriptsize{$a_2^{22}$}};
    \fill [fill=gray,fill opacity=0.8,text opacity=1] (18,-1) rectangle (19,0) node[pos=.5] {\scriptsize{$a_1^{32}$}};
    \fill [fill=gray,fill opacity=0.5] (4,-3.5) rectangle (5,-2.5);
    \fill [fill=gray,fill opacity=0.2] (5,-3.5) rectangle (6,-2.5);
    \fill [fill=gray,fill opacity=0.2] (6,-3.5) rectangle (7,-2.5);
    \fill [fill=gray,fill opacity=0.8] (7,-3.5) rectangle (8,-2.5);
    \fill [fill=gray,fill opacity=0.5] (8,-3.5) rectangle (9,-2.5);
    \fill [fill=gray,fill opacity=0.2] (9,-3.5) rectangle (10,-2.5);
    \fill [fill=gray,fill opacity=0.2, text opacity=1] (4,-6) rectangle (5,-5) node[pos=.5] {\scriptsize{$m_2(1)$}};
    \fill [fill=gray,fill opacity=0.2, text opacity=1] (5,-6) rectangle (6,-5) node[pos=.5] {\scriptsize{$m_2(2)$}};
    \fill [fill=gray,fill opacity=0.2, text opacity=1] (6,-6) rectangle (7,-5) node[pos=.5] {\scriptsize{$m_2(3)$}};
    \fill [fill=gray,fill opacity=0.5, text opacity=1] (7,-6) rectangle (8,-5) node[pos=.5] {\scriptsize{$m_2(4)$}};
    \fill [fill=gray,fill opacity=0.5, text opacity=1] (8,-6) rectangle (9,-5) node[pos=.5] {\scriptsize{$m_2(5)$}};
    \fill [fill=gray,fill opacity=0.7, text opacity=1] (9,-6) rectangle (10,-5) node[pos=.5] {\scriptsize{$m_2(6)$}};
    \draw[->](1.5,-1) -- (5.5,-2.5);
    \draw[->](3.5,-1) -- (9.5,-2.5);
    \draw[->](5.5,-1) -- (6.5,-2.5);
    \draw[->](8.5,-1) -- (8.5,-2.5);
    \draw[->](10.5,-1) -- (4.5,-2.5);
    \draw[->](18.5,-1) -- (7.5,-2.5);
    \draw[->](4.5,-3.5) -- (8.5,-5);
    \draw[->](5.5,-3.5) -- (4.5,-5);
    \draw[->](6.5,-3.5) -- (6.5,-5);
    \draw[->](7.5,-3.5) -- (9.5,-5);
    \draw[->](8.5,-3.5) -- (7.5,-5);
    \draw[->](9.5,-3.5) -- (5.5,-5);
    \draw[thick] (8,-1.3) -- (8,0.3);
    \draw[thick] (16,-1.3) -- (16,0.3);
    \draw[thick,dashed] (3.2,-3.8) -- (3.2,-2.2);
    \draw[thick,dashed] (11.6,-3.8) -- (11.6,-2.2);
    \draw[thick,dashed] (19.3,-3.8) -- (19.3,-2.2);
    \draw[thick,dashed] (3.2,-6.3) -- (3.2,-4.7);
    \draw[thick,dashed] (11.6,-6.3) -- (11.6,-4.7);
    \draw[thick,dashed] (19.3,-6.3) -- (19.3,-4.7);
    \draw [->] (0,-0.5) to [out=240,in=120] (0,-3);
    \node at (0.3,-1.75) {$\sigma^{-1}$};
    \draw [->] (0,-3) to [out=240,in=120] (0,-5.5);
    \node at (0.3,-4.25) {$\alpha^{-1}$};
    \end{tikzpicture}
    \caption{The construction of $\alpha^{-1}$.  The tall solid vertical lines represent the divisions between the intervals $W_{0,1}, \dots W_{0,3}$, and the tall dashed vertical lines represent the divisions between the intervals $H_1, \dots H_4$. All arrows indicate that the respective transformations map the specified vertex intervals to vertex intervals by translation.}
    \label{fig:alpha}
\end{figure}

Next we construct the map $\beta^{-1}$, as illustrated in \Cref{fig:beta}. Recall the definition of $I_{ij}$ described in the construction of $ \tau_{v^*}=\tau^{-1}$. Each interval $I_{ij} $ is contained in $W_{0,i}$ and  $\tau^{-1}(I_{1j})$, $\tau^{-1}(I_{2j})$, \dots $\tau^{-1}(I_{s_h j})$ are consecutive intervals (in that order) whose union is $H_j$.
Unlike $\sigma^{-1}$, $\tau^{-1}$ may not be in  $\mathcal{M}_{kn}$, and so the image of vertex intervals under $\tau^{-1}$ are not necessarily vertex intervals. Informally, $\beta^{-1}$ will map the images of vertex intervals under $\tau^{-1}$ to vertex intervals in a way that maintains their relative order in $H_j$.

We now formally describe $\beta^{-1}$. Since $I_{ij}$ has length
$v^*_{ij}$, there are at least $\lfloor v^*_{ij}kn \rfloor -1$ vertex intervals contained in $I_{ij}\subseteq W_{0,i}$, all of which are mapped to $H_j$ under $\tau^{-1}$.   Let $b^{ij}_1,b^{ij}_2 \dots b^{ij}_{k_{ij}}$ enumerate $k_{ij}$ of these vertex intervals contained in $I_{ij}$.
Define $\beta \in \mathcal{M}$ so that $\beta^{-1}$ translates  the interval (which is not necessarily a vertex interval) $\tau^{-1}(b_\ell^{ij})$ to the vertex interval specified as follows: $$\beta^{-1}(\tau^{-1}(b_\ell^{ij}))=m_j\brac{\sum_{z=1}^{i-1} k_{zj} + \ell},$$
and $H_j\setminus \brac{\bigcup_{i=1}^{s_w}\bigcup_{\ell=1}^{k_{ij}}\tau^{-1}(b^{ij}_\ell)}$ maps to $H_j\setminus \brac{\bigcup_{i=1}^{s_w}\bigcup_{\ell=1}^{k_{ij}}\beta( \tau^{-1}(b^{ij}_\ell))}$ under $\beta^{-1}$ in any invertible manner. Since $\beta(H_i)=H_i$, $\beta$ and $\beta^{-1}$ respect the intervals $H_1, \dots H_{s_h}$.

\begin{figure}
    \centering
    \begin{tikzpicture}[scale=0.75]
    \draw (0,0) -- (20,0);
    \draw (0,-1) -- (20,-1);
    \foreach \i in {0,...,20}{
    \draw (\i,0) -- (\i,-1);
    }
    \draw (0,-2.5) -- (20,-2.5);
    \draw (0,-3.5) -- (20,-3.5);
    \foreach \i in {0,...,20}{
    \draw (\i,-2.5) -- (\i,-3.5);
    }
    \draw (0,-5) -- (20,-5);
    \draw (0,-6) -- (20,-6);
    \foreach \i in {0,...,20}{
    \draw (\i,-5) -- (\i,-6);
    }
    \fill [fill=gray,fill opacity=0.2] (1.1,-1) rectangle (2,0);
    \fill [fill=gray,fill opacity=0.2,text opacity=1] (2,-1) rectangle (3,0) node[pos=.5] {\scriptsize{$b_1^{12}$}};
    \fill [fill=gray,fill opacity=0.2,text opacity=1] (3,-1) rectangle (4,0) node[pos=.5] {\scriptsize{$b_2^{12}$}};
    \fill [fill=gray,fill opacity=0.2,text opacity=1] (4,-1) rectangle (5,0) node[pos=.5] {\scriptsize{$b_3^{12}$}};
    \fill [fill=gray,fill opacity=0.2] (5,-1) rectangle (5.6,0);
    \fill [fill=gray,fill opacity=0.5] (9.4,-1) rectangle (10,0);
    \fill [fill=gray,fill opacity=0.5,text opacity=1] (10,-1) rectangle (11,0) node[pos=.5] {\scriptsize{$b_1^{22}$}};
    \fill [fill=gray,fill opacity=0.5,text opacity=1] (11,-1) rectangle (12,0) node[pos=.5]{\scriptsize{$b_2^{22}$}};
    \fill [fill=gray,fill opacity=0.5] (12,-1) rectangle (12.1,0);
    \fill [fill=gray,fill opacity=0.7] (16.9,-1) rectangle (17,0);
    \fill [fill=gray,fill opacity=0.7,text opacity=1] (17,-1) rectangle (18,0) node[pos=.5]{\scriptsize{$b_1^{32}$}};
    \fill [fill=gray,fill opacity=0.7] (18,-1) rectangle (18.1,0);
    \fill [fill=gray,fill opacity=0.2] (3.2,-3.5) rectangle (7.7,-2.5);
    \fill [fill=gray,fill opacity=0.5] (7.7,-3.5) rectangle (10.4,-2.5);
    \fill [fill=gray,fill opacity=0.7] (10.4,-3.5) rectangle (11.6,-2.5);
    \fill [fill=gray,fill opacity=0.2, text opacity=1] (4,-6) rectangle (5,-5) node[pos=.5] {\scriptsize{$m_2(1)$}};
    \fill [fill=gray,fill opacity=0.2, text opacity=1] (5,-6) rectangle (6,-5) node[pos=.5] {\scriptsize{$m_2(2)$}};
    \fill [fill=gray,fill opacity=0.2, text opacity=1] (6,-6) rectangle (7,-5) node[pos=.5] {\scriptsize{$m_2(3)$}};
    \fill [fill=gray,fill opacity=0.5, text opacity=1] (7,-6) rectangle (8,-5) node[pos=.5] {\scriptsize{$m_2(4)$}};
    \fill [fill=gray,fill opacity=0.5, text opacity=1] (8,-6) rectangle (9,-5) node[pos=.5] {\scriptsize{$m_2(5)$}};
    \fill [fill=gray,fill opacity=0.7, text opacity=1] (9,-6) rectangle (10,-5) node[pos=.5] {\scriptsize{$m_2(6)$}};
    \draw[->](1.1,-1) -- (3.2,-2.5);
    \draw[->](5.6,-1) -- (7.7,-2.5);
    \draw[->](9.4,-1) -- (7.7,-2.5);
    \draw[->](12.1,-1) -- (10.4,-2.5);
    \draw[->](16.9,-1) -- (10.4,-2.5);
    \draw[->](18.1,-1) -- (11.6,-2.5);
    \draw[-](4.1,-3.9) -- (7.1,-3.9);
    \draw[-](4.1,-4.1)--(4.1,-3.7);
    \draw[-](7.1,-4.1)--(7.1,-3.7);
    \draw[-](8.3,-3.9) -- (10.3,-3.9);
    \draw[-](8.3,-4.1)--(8.3,-3.7);
    \draw[-](10.3,-4.1)--(10.3,-3.7);
    \draw[-](10.5,-3.9) -- (11.5,-3.9);
    \draw[-](10.5,-4.1)--(10.5,-3.7);
    \draw[-](11.5,-4.1)--(11.5,-3.7);
    \draw[->](4.1,-3.9) -- (4,-5);
    \draw[->](7.1,-3.9) -- (7,-5);
    \draw[->](8.3,-3.9) -- (7,-5);
    \draw[->](10.3,-3.9) -- (9,-5);
    \draw[->](10.5,-3.9) -- (9,-5);
    \draw[->](11.5,-3.9) -- (10,-5);
    \draw[thick] (8,-1.3) -- (8,0.3);
    \draw[thick] (16,-1.3) -- (16,0.3);
    \draw[thick,dashed] (3.2,-3.8) -- (3.2,-2.2);
    \draw[thick,dashed] (11.6,-3.8) -- (11.6,-2.2);
    \draw[thick,dashed] (19.3,-3.8) -- (19.3,-2.2);
    \draw[thick,dashed] (3.2,-6.3) -- (3.2,-4.7);
    \draw[thick,dashed] (11.6,-6.3) -- (11.6,-4.7);
    \draw[thick,dashed] (19.3,-6.3) -- (19.3,-4.7);
    \draw [->] (0,-0.5) to [out=240,in=120] (0,-3);
    \node at (0.3,-1.75) {$\tau^{-1}$};
    \draw [->] (0,-3) to [out=240,in=120] (0,-5.5);
    \node at (0.3,-4.25) {$\beta^{-1}$};
    \end{tikzpicture}
    \caption{The construction of $\beta^{-1}$. The tall solid vertical lines represent the divisions between the intervals $W_{0,1}, \dots W_{0,3}$, and the tall dashed vertical lines represent the divisions between the intervals $H_1, \dots H_4$. The arrows corresponding to $\tau^{-1}$ illustrate that $\tau^{-1}$ map intervals to intervals by translation. The arrows depicting $\beta^{-1}$ show that $\beta^{-1}$ maps adjacent intervals of the form $\tau^{-1}(b_{\ell}^{ij})$ (shown by horizontal line segments) to vertex intervals by translation.}
    \label{fig:beta}
\end{figure}
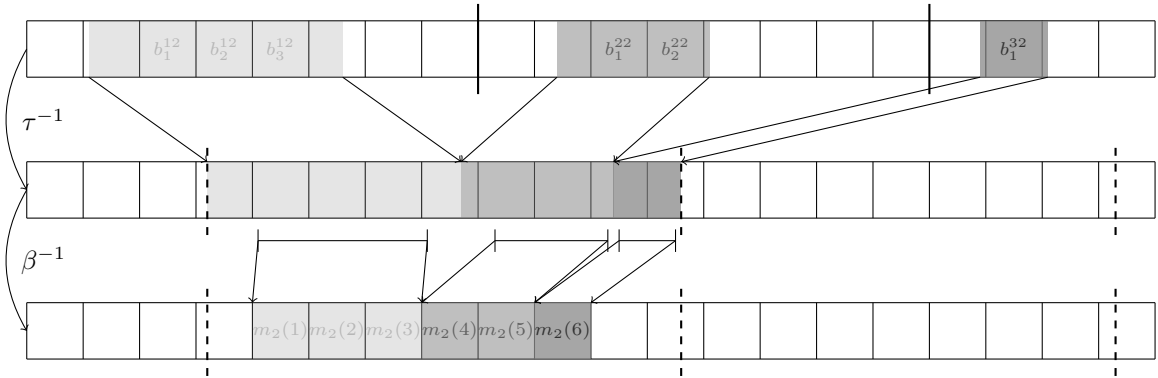

Next we construct $K$, a set of synchronized vertex intervals. Note that for all triples $i,j, \ell$ with $i \in [s_w]$, $j \in [s_h]$, and $\ell \in [k_{ij}]$,  $m_j\brac{\sum_{z=1}^{i-1} k_{zj} + \ell}$ is a synchronized vertex interval since $a_\ell^{ij}$ and $b_\ell^{ij}$ are vertex intervals contained in $W_{0,i}$. Let $$K=\left\{ m_j\brac{\sum_{z=1}^{i-1} k_{zj} + \ell} |i \in [s_w], j \in [s_h], \text{ and  }\ell \in [k_{ij}]\right\}.$$

Finally, we bound the size of $K$. Recall  that by construction $v_{ij}=n_{ij}/(kn)$ and $\Vert v- v^*\Vert_\infty \leq \eta/(4s_w s_h)$. It follows that
$$|n_{ij} - v^*_{ij}kn| \leq \frac{\eta k n}{4s_ws_h}.$$
Since $| \lfloor v^*_{ij}kn \rfloor -1-v^*_{ij}kn| \leq 2$, it follows that
$$k_{ij}= \min \{ n_{ij}, \lfloor v^*_{ij}kn \rfloor -1 \} \geq n_{ij} -2 - \frac{ \eta kn }{4 s_w s_h}.$$
 We use this to lower bound the total number of synchronized intervals in $K$,
 \begin{align*}
 |K|&=\sum_{i=1}^{s_w}\sum_{j=1}^{s_h}k_{ij}\geq  \brac{\sum_{i=1}^{s_w}\sum_{j=1}^{s_h}n_{ij}} - 2s_ws_h  -\frac{ \eta  k n }{4}\geq kn -s_h+1
 - 2s_ws_h  -\frac{ \eta k n }{4} \\
 &\geq kn \brac{1 - \frac{3s_w s_h}{kn}- \frac{\eta }{4}}\geq kn \brac{1 -\eta/2},
  \end{align*}
since  $n\geq 12s_h^2 s_w /(k\eta)\geq 12s_h  s_w /(k\eta)$ by the assumption of the lemma.

\medskip

 Finally, we address (d). We construct a coupling of $f$ and $g$ so that $(f^\sigma)^\alpha$ and $(g^\tau)^\beta$ agree on sets of the form $\bar{m} \times \bar{m}'$ where $\bar{m}, \bar{m}' \in K$ are synchronized intervals.  Let $v$ and $v'$ be the vertices in $f$ corresponding to the vertex intervals that are mapped to $\bar{m}$ and $\bar{m}'$ respectively under $\sigma \circ \alpha$.
 Let $u$ and $u'$ be the vertices in $g$ corresponding to the vertex intervals that are mapped to $\bar{m}$ and $\bar{m}'$ respectively under $\tau \circ \beta$. By construction, $v$ and $u$ correspond to vertex intervals contained in the same interval $W_{0,i}$, and likewise $v'$ and $u'$ correspond to vertex intervals contained in the same interval $W_{0,j}$.
 Let $X_{ab}$ and $Y_{ab}$ be the indicator random variables for the events that there is an edge between vertices $a$ and $b$ in $f$ and $g$ respectively. Recall that $p_{ij}$ denotes the value $W_0$ takes on $W_{0,i} \times W_{0,j}$. Since $X_{v v'} \sim Bern(p_{ij})$ and $Y_{u u'} \sim Bern(p_{ij})$, we can couple $X_{v v'}$ and $Y_{u u'}$ exactly, which then guarantees that $(f^\sigma)^\alpha$ and $(g^\tau)^\beta$ agree on the set $\bar{m} \times \bar{m}'$. 

Since $(f^\sigma)^\alpha$ and $(g^\tau)^\beta$ agree on the synchronized vertex intervals $\brac{\bigcup_{\bar{m} \in K} \bar{m}}^2$ and $|\brac{\bigcup_{\bar{m} \in K} \bar{m}}|\geq 1-\eta/2$, it follows that $d_\square\brac{(f^\sigma)^\alpha,(g^\tau)^\beta} \leq \eta$, as desired.
 \end{proof}

\begin{proof}[Proof of \Cref{bound on one}.]
Fix $\tau \in \mathcal{M}$. Note that $\tau$ is invertible and $d_\square(f^\tau, g)= d_\square(f, g^{\tau^{-1}})$. It follows that
    $$\mathbb{P}_{kn, W_0} \left(d_\square\left (f^{G_{kn}^\tau},h \right) \leq \ve  \right)=\mathbb{P}_{kn, W_0} \left(d_\square\left (f^{G_{kn}},h^{\tau^{-1}} \right) \leq \ve  \right)=\pr_{kn, W_0} \left( f^{G_{kn}} \in \left\{g: d_\square \left(g, h^{\tau^{-1}} \right) \leq \ve \right\}\right).$$
Note that \cite[Lemma 5.4]{Chatterjee2015} implies that the set $\{g: d_\square(g, h^{\tau^{-1}} ) \leq \ve\}$ is closed in the weak topology.
    Applying \Cref{weak ldp}, we obtain
   \begin{align*}
\limsup_{n \to \infty} \frac{1}{(kn)^2} \log \mathbb{P}_{kn, W_0} \left(d_\square(f^{G_{kn}^\tau},h) \leq \ve \right)&=
\limsup_{n \to \infty } \frac{1}{(kn)^2} \log \pr_{kn, W_0}\brac{ f^{G_{kn}} \in \left\{g: d_\square \left(g, h^{\tau^{-1}} \right) \leq \ve \right\}}\\
&\leq - \inf_{f \in \{g: d_\square(g, h^{\tau^{-1}} ) \leq \ve \}} I_{W_0}(f)  \leq -\inf_{f: \delta_\square(f, h) \leq \ve } I_{W_0}(f),
    \end{align*}
    where the last line follows from the observation that $\delta_\square(g,h)\leq d_\square (g, h^{\tau^{-1}} )$.
\end{proof}

\subsection{Proof of \Cref{theorem:variational-problem-nice}}\label{sec:pf-vp-nice}
We begin with the following theorem, which is a direct adaptation of  \cite[Theorem 3.1]{Chatterjee2011} to general $k$-block graphons $W_0$.  As usual, for $\ti{f} \in \mathcal{W}$ and $\ti{H} \subseteq\tW$, define $\delta_{\square}(\ti{f}, \ti{H}) \triangleq \inf_{\ti{h} \in \ti{H}} \delta_{\square}(\ti{f}, \ti{h})$.
\begin{theorem}\label{theorem:conditional}
Let $W_0$ be a uniform  $k$-block graphon. Let $\ti{F}$ be a closed subset of $\tW$, and let $\ti{F}^0$ be its interior. Suppose
\begin{align}
\inf_{\ti{h} \in \ti{F}^0} J_{W_0}(\ti{h}) &= \inf_{\ti{h} \in \ti{F}} J_{W_0}(\ti{h}).
\label{interior_assump} \end{align}
Let $\ti{F}^{\star}$ be the subset of $\ti{F}$ where $J_{W_0}$ is minimized. Then $\ti{F}^{\star}$ is non-empty and compact, and
\begin{equation}\label{inf=sup}
\min_{\ti{h} \in \ti{F}} J_{W_0}(\ti{h})
=-\lim_{n\to\infty} \frac{1}{(kn)^2}
\log \tilde{\mathbb{P}}_{kn, W_0}( \ti{F}).
\end{equation}
If $\min_{\ti{h} \in \ti{F}} J_{W_0}(\ti{h})<\infty$, then for all sufficiently large $n$ and all
$\epsilon > 0$, $\ti{\pr}_{kn,W_0}(\ti f^{G_{kn}}\in  \ti{F})>0$ and
\[\ti{\pr}_{kn,W_0}\left(\delta_{\square} (\ti{f}^{G_{kn}}, \ti{F}^{\star})  \geq \ve \Big| {\ti f^{G_{kn}}} \in \ti{F} \right) \leq e^{-C(\epsilon, \ti{F}) (kn)^2},\]
where $C(\epsilon, \ti{F})$ is a positive constant depending only on $\epsilon$ and $\ti{F}$. In particular, if $\ti{F}^{\star}$ contains only one element $\ti{h}^{\star}$ (and $J_{W_0}(\ti{h}^{\star})<\infty$), then the conditional distribution of ${\ti f^{G_{kn}}}$ given ${\ti f^{G_{kn}}} \in \ti{F}$ converges to the point mass at $\ti{h}^{\star}$ as $n \to \infty$.
\end{theorem}
\begin{proof}
Since $\tW$ is compact and $\ti{F}$ is  closed, $\ti{F}$ is also compact. By Lemma \ref{lemma:semi-continuity-J}, the function $J_{W_0}$ is lower semi-continuous on $\ti{F}$. Since $\ti{F}$ is compact, $J_{W_0}$ must attain its minimum on $\ti{F}$. Therefore, $\ti{F}^{\star}$ is non-empty. Moreover, by the lower semi-continuity of $J_{W_0}$, $\ti{F}^{\star}$ is closed, and hence compact.
Finally, by Theorem \ref{ldp},
$$
\begin{aligned}
-\inf_{\ti{h} \in \ti{F}^0} J_{W_0}(\ti{h})
&\leq
\liminf_{n \to \infty} \frac{1}{(kn)^2} \log \tilde{\mathbb{P}}_{kn, W_0}(\tilde{F}^0)
\leq
\liminf_{n \to \infty} \frac{1}{(kn)^2} \log \tilde{\mathbb{P}}_{kn, W_0}(\tilde{F})
\\
&\leq \limsup_{n \to \infty} \frac{1}{(kn)^2} \log \tilde{\mathbb{P}}_{kn, W_0}(\tilde{F})
 \leq - \inf_{\tilde{h} \in \tilde{F}} J_{W_0}(\tilde{h})=-\min_{\ti{h} \in \ti{F}} J_{W_0}(\ti{h}).
\end{aligned}
$$
Combined with  \eqref{interior_assump} this  proves \eqref{inf=sup}.

Next, assume that the $\inf$ in \eqref{interior_assump} is finite.  This is only compatible with
\eqref{inf=sup} if $ \tilde{\mathbb{P}}_{kn, W_0}(\tilde{F^0})>0$
for all $n$ larger than some $n_0$. Fix $\epsilon > 0$ and let
\[\ti{F}_{\epsilon} \triangleq \left\{ \ti{h} \in \ti{F} : \delta_{\square} (\ti{h}, \ti{F}^{\star})  \geq \ve \right\}, \]
which is also a closed subset. Observe that $\ti{F}_{\epsilon} \cap \ti{F}^{\star} = \emptyset$.
Then
\begin{align*}
\ti{\pr}_{kn,W_0} \left(\delta_{\square}({ \ti f^{G_{kn}}} , \ti{F}^{\star}) \geq \epsilon | {\ti f^{G_{kn}}} \in \ti{F} \right) &= \frac{\ti{\pr}_{kn,W_0} \left({ \ti f^{G_{kn}}} \in \ti{F}_{\epsilon} \right)}{\ti{\pr}_{kn,W_0}\left( {\ti f^{G_{kn}}} \in \ti{F} \right)}\leq \frac{\ti{\pr}_{kn,W_0} \left( {\ti f^{G_{kn}}}  \in \ti{F}_{\epsilon} \right)}{\ti{\pr}_{kn,W_0} \left( {\ti f^{G_{kn}}} \in \ti{F}^0 \right)}.
\end{align*}
Using Theorem \ref{ldp} again, this shows that,
\begin{align*}
&\limsup_{n \to \infty} \frac{1}{(kn)^2} \log \ti{\pr}_{kn,W_0} \left(\delta_{\square}( {\ti f^{G_{kn}}} , \ti{F}^{\star}) \geq \epsilon | {\ti f^{G_{kn}}} \in \ti{F} \right) \\
&\leq \limsup_{n \to \infty} \frac{1}{(kn)^2} \left[\log  \ti{\pr}_{kn,W_0} \left( {\ti f^{G_{kn}}} \in \ti{F}_{\epsilon} \right) - \log \ti{\pr}_{kn,W_0} \left( {\ti f^{G_{kn}}} \in \ti{F}^0 \right) \right]\\
&= \limsup_{n \to \infty} \frac{1}{(kn)^2} \log \ti{\pr}_{kn,W_0} \left( {\ti f^{G_{kn}}} \in \ti{F}_{\epsilon} \right) - \liminf_{n \to \infty} \frac{1}{(kn)^2} \log \ti{\pr}_{kn,W_0} \left( {\ti f^{G_{kn}}} \in \ti{F}^0 \right) \\
&\leq \inf_{\ti{h} \in \ti{F}^0} J_{W_0}(\ti{h}) - \inf_{\ti{h} \in \ti{F}_{\epsilon}} J_{W_0}(\ti{h})\\
&=\inf_{\ti{h} \in \ti{F}} J_{W_0}(\ti{h}) - \inf_{\ti{h} \in \ti{F}_{\epsilon}} J_{W_0}(\ti{h}).
\end{align*}
It now suffices to show that
$\inf_{\ti{h} \in \ti{F}} J_{W_0}(\ti{h}) < \inf_{\ti{h} \in \ti{F}_{\epsilon}} J_{W_0}(\ti{h})$. Clearly, $\inf_{\ti{h} \in \ti{F}} J_{W_0}(\ti{h}) \leq \inf_{\ti{h} \in \ti{F}_{\epsilon}} J_{W_0}(\ti{h})$. Suppose that equality holds. The compactness of $\ti{F}_{\epsilon}$ and the lower semi-continuity of $J_{W_0}$ (\Cref{lemma:semi-continuity-J}) imply that there exists $\ti{g} \in \ti{F}_{\epsilon}$ that attains the infimum. It follows that  $J_{W_0}(\ti{g}) = \inf_{\ti{h} \in \ti{F}_{\epsilon}} J_{W_0}(\ti{h}) = \inf_{\ti{h} \in \ti{F}}J_{W_0}(\ti{h})$. But then $\ti{g} \in \ti{F}^{\star}$, and so $\ti{F}_{\epsilon} \cap \ti{F}^{\star} \neq \emptyset$, which is a contradiction.
\end{proof}

\begin{proof}[Proof of Theorem \ref{theorem:variational-problem-nice}]
We will prove the theorem by establishing condition  \eqref{interior_assump} in Theorem \ref{theorem:conditional}.
We first note that the continuity of $\phi_{\tau}(W_0,\cdot)$ at $t$ excludes the trivial case $\tau(\W)=\{t\}$, since then $\tau(\tW)=\{t\}$ as well, which shows that  $\phi_{\tau}(W_0,\cdot)$ jumps from a finite constant to $\infty$ at $t$. Therefore, we may assume the graph parameter $\tau$ is not constant. 

Next we recall that $\phi_{\tau}(W_0,t) = \min \{J_{W_0}(\tilde{f}) : \ti{f}\in \ti F\}$, where
$\ti{F} = \{ \ti{f}\in \tW : \tau(\ti f) \geq t \}$.
By the continuity of $\tau$, $\ti{F}$ is  closed.
Also by the continuity of $\tau$,
the open set $\ti O=\{\ti{f}\in \tW : \tau(\ti f)>t\} \subseteq \ti{F}^0$ contains
 $\ti F_\ve=\{ \ti{f}\in \tW : \tau(\ti f) \geq t+\ve \}$ for all $\ve > 0$.
As a consequence,
 $$
 \phi_{\tau}(W_0,t) =\inf_{\ti{h} \in \ti{F}} J_{W_0}(\ti{h})\leq  \inf_{\ti{h} \in \ti{F}^0} J_{W_0}(\ti{h})
 \leq \inf_{\ti{h} \in \ti{O}} J_{W_0}(\ti{h}) \leq  \inf_{\ti{h} \in \ti{F_\ve}} J_{W_0}(\ti{h})= \phi_{\tau}(W_0,t+\ve).
 $$
 Sending $\ve\to 0$ and using the continuity of $\phi_{\tau}(W_0,\cdot)$ at $t$, we see that the first inequality is saturated, proving \eqref{interior_assump}.

 Recall that $\pr_{kn,W_0}$ is supported in $\W_\Omega$. The proof is completed by noting that $\phi_\tau(W_0,t)<\infty$ if and only if $t\leq t_{\max}$.
\end{proof}

\section{$\phi_{\tau}$: Monotonicity, continuity, and examples}
\label{sec:properties}
In this section, we establish some analytical properties of the function $\phi_{\tau}$, which will be critical for our discussion of symmetry/symmetry breaking in the subsequent sections. \Cref{vp-prelim} collects some preliminary properties of homomorphism densities and the cut distance. In  \Cref{sec:useful-phi}, we introduce the ``sufficient increase property'', which guarantees the continuity of $\phi_{\tau}$. Further, we establish that homomorphism densities satisfy this property, and the operator norm satisfies this property under additional assumptions. Finally, \Cref{prop-of-phi} establishes an alternative variational representation of $\phi_\tau$ at points of continuity. Using this representation, we identify a class of  parameters $\tau$ such that $\phi_{\tau}$ is strictly increasing.

\subsection{Preliminaries}\label{vp-prelim}
In  subsequent sections, it will be necessary to express the homomorphism density as the sum of interval labeled homomorphisms and identify interval labeled homomorphisms that are always zero on $\W_\Omega$. Given $m \in \mathbb{Z}^+$, $\gamma \in \Delta_m$ and $W_0 \in \mathcal{B}^\gamma$, let  $I_1, I_2, \dots I_m$ be the intervals of $\gamma$, i.e. $I_1 = [0, \gamma_1]$ and $I_j= (\sum_{i=1}^{j-1}\gamma_i, \sum_{i=1}^{j}\gamma_i]$ for $j \geq 2$. When we proved \Cref{theorem:variational-problem-nice}, we used $k$ to denote the total number of intervals and assumed that all intervals have the same length. Here the intervals need not be the same length--- to emphasize this we now use $m$ to denote the total number of intervals.
The blocks of $W_0$ have the form $I_i \times I_j$ for $i,j \in [m]$.
Let $v= |V(H)|$ be the number of vertices in $H$, and let $Y \in [m]^{v}$ be a vector of vertex interval indices. Define the \emph{interval-labeled homomorphism density} as
$$t(H, g, Y)= \int_{x_1 \in I_{Y_1}} \int_{x_2 \in I_{Y_2}} \dots \int_{x_v \in I_{Y_v}} \prod_{\{i,j\} \in E(H)} g(x_i, x_j) \, dx_v \dots dx_{2} dx_1.$$
In other words, $t(H,g,Y)$ accounts for the homomorphisms in which the $j^{th}$ vertex is in $I_{Y_j}$ for all $j \in [v]$ and so $$t(H,g) = \sum_{Y \in [m]^v} t(H,g,Y).$$

\tikzset{blockr/.style= {rectangle, draw=black!50, fill=black!20, thick}}
\begin{figure}
\centering
\begin{tikzpicture}[scale=0.7]
\draw [] (0,0) rectangle node {$0$} (2,-2);
\draw [blockr] (2,0) rectangle node {$p$} (4,-2);
\draw [](4,0) rectangle node {$0$} (6,-2);
\draw [blockr] (0,-2) rectangle node {$p$} (2,-4);
\draw [] (2,-2) rectangle node {$0$} (4,-4);
\draw [blockr] (4,-2) rectangle node {$p$} (6,-4);
\draw [] (0,-4) rectangle node {$0$} (2,-6);
\draw [blockr] (2,-4) rectangle node {$p$} (4,-6);
\draw [blockr] (4,-4) rectangle node {$p$} (6,-6);
\draw [|-|] (-0.2,0) -- (-0.2,-2) node[midway,left] {$I_1$};
\draw [|-|] (-0.2,-2) -- (-0.2,-4) node[midway,left] {$I_2$};
\draw [|-|] (-0.2,-4) -- (-0.2,-6) node[midway,left] {$I_3$};
\end{tikzpicture}
\caption{For $W_0$ the graphon above and $H$ a triangle, the block $I_1 \times I_2$ is not relevant.  Increasing the number of edges between $I_1$ and $I_2$ will not increase the number of triangles in the graph since any triangle has at least two vertices coming only from $I_3$.}
\label{irr}
\end{figure}
Next, we define \emph{relevant blocks} to be the blocks whose values may affect the homomorphism density of a graphon in $\W_\Omega$. Increasing the value of a graphon $g$ in $\W_\Omega$ on a relevant block has the potential to increase $t(H,g)$.  \Cref{irr} gives an example of a block that is not relevant.

\begin{definition}\label{relevant} Fix a finite graph $H$ and $\W_\Omega$.
We say an interval labeling vector $Y$ is \emph{irrelevant} with respect to $W_0$ if there exists $\{i,j \} \in E(H)$ such that the block $I_{Y_i} \times I_{Y_j}$ takes value zero on $W_0$. Equivalently, $Y$ is irrelevant if $t(H,g,Y)=0$ for all $g \in \W_\Omega$. We say $Y$ is relevant otherwise.

We say a block $I_a \times I_b$ contributes to the interval-labeled homomorphism density $t(H,g,Y)$ if $Y_i=a$ and $Y_j=b$ for some $\{i,j\} \in E(H)$.  We say the block $I_a \times I_b$ is \emph{relevant} if $I_a \times I_b$ contributes to some $t(H,g,Y)$ with $Y$  relevant.  Let $R \subseteq [0,1]^2$ be the union of all relevant blocks.
\end{definition}

\noindent
Note that a block $I_a\times I_b$ is relevant if and only if $p_{ab}>0$
and $t(H,W_0)$ strictly decreases when $p_{ab}$ is lowered.

Our next result establishes that if the cut distance between two graphons is at least a constant, one can find a region where the values on the  graphons differ by at least a constant. This result will be crucially used to establish the ``sufficient increase property''  in this section. In our subsequent discussion, we will  use this result to establish the existence of nearby graphons with lower entropy.

\begin{lemma}\label{cut-sets}
Let $f,g \in \W$. Let $S_\beta^+=\{(x,y) \in [0,1]^2: f(x,y) -g(x,y) \geq \beta\}$ and $S_\beta^-=\{(x,y) \in [0,1]^2: g(x,y) -f(x,y) \geq \beta\}$.
\begin{enumerate}
    \item If $f \geq g$ pointwise and $d_\square(f,g) \geq \ve$, then $|S_{\ve/2}^+| > \ve/2$.
    \item If $d_\square(f,g) > \ve$, then $|S_{\ve/4}^+|\geq \ve/4$ or $|S_{\ve/4}^-|\geq \ve/4$.
\end{enumerate}
\end{lemma}

\begin{proof}
Suppose $f \geq g$ pointwise and $d_\square(f,g) \geq \ve$. Since $f \geq g$ pointwise, $d_\square(f,g)= \Vert f-g \Vert_1$.
It follows that
$$\ve \leq d_\square(f,g) = \Vert f-g \Vert_1 =\int_{[0,1]^2}f-g \leq |S_{\ve/2}^+| + \frac{\ve}{2} (1-|S_{\ve/2}^+|),$$
and so $|S_{\ve/2}^+| \geq \ve/(2-\ve) > \ve/2$.

Next suppose $d_\square(f,g) > \ve$ with no additional assumptions on $f,g \in \W$. Let $S^+=\{(x,y) \in [0,1]^2: f(x,y) \geq g(x,y)\}$ and $S^-=\{(x,y) \in [0,1]^2: f(x,y) < g(x,y)\}$.
Since $d_\square(f,g) > \ve$, there exists $A,B \subseteq [0,1]$ such that
$$\left| \int_{A \times B} f- g \right | \geq \ve.$$
Observe $$\left| \int_{A \times B} f- g \right | \leq  \int_{A \times B}\left| f- g \right | = \int_{(A \times B )\cap S^+} f-g + \int_{(A \times B) \cap S^-} g-f.
$$
It follows that $\int_{(A \times B )\cap S^+} f-g \geq \ve/2$ or $\int_{(A \times B) \cap S^-} g-f \geq\ve/2$. In the first case, let $\overline{f}(x,y) = f(x,y)$ for $(x,y) \in (A \times B) \cap S^+$, and otherwise set $\overline{f}(x,y) = g(x,y)$. Then $\overline{f} \geq g$ pointwise, and $d_{\square}(\overline{f},g) \geq \frac{\epsilon}{2}$. By the first statement, we obtain $|S^+_{\frac{\varepsilon}{4}}| \geq \frac{\varepsilon}{4}$. A similar argument applies to the second case.
\end{proof}

\subsection{Establishing the continuity of $\phi_\tau$}\label{sec:useful-phi}
In this subsection we establish that $\phi_\tau$ is continuous for certain graph parameters $\tau$.  When $\tau$ is clear from context, we let $\phi(t)= \phi_\tau(W_0, t).$

\begin{lemma}\label{homomorphism-cont} Let $H$ be a finite graph, let $\tau= t(H, \cdot)$, let  $m \in \mathbb{Z}^+$, $\gamma \in \Delta_m$, and $W_0 \in \mathcal{B}^\gamma$. Then $\phi$  is continuous
on $\R\setminus\{t^{\tau}_{\max}(\tW_{\Omega})\}$.
\end{lemma}

\begin{lemma}\label{operator-cont}  Let $\tau(g)= \Vert g\Vert_{\emph{op}}$,   let $W_0$ be a two-block bipartite graphon with $W_0 \in \mathcal{B}^{(\gamma, 1-\gamma)}$ where $\gamma \in (0,1) \cap \mathbb{Q}$. Then $\phi$  is continuous
on $\R\setminus\{t^{\tau}_{\max}(\tW_{\Omega})\}$.
\end{lemma}

In order to establish the above lemmas, we describe the sufficient increase property of $\tau$ and $W_0$, and show that having this property guarantees that $\phi_\tau$ is continuous. Then we show that homomorphism densities  have this property (with any block constant base graphon $W_0$), and the operator norm has this property when $W_0$ is a two block bipartite graphon.

\begin{definition}\label{defn:p-star} We say that $\tau$ has \emph{the sufficient increase property} on $\W_\Omega$ if the following is true. Fix any $\eta >0$ and $t^{\tau}_{\min}(\tW_{\Omega}) \leq t < t^{\tau}_{\max}(\tW_{\Omega})$. Then there exist $\alpha= \alpha(t, t_{\max},\eta), \beta=\beta(t, t_{\max},\eta) >0$,  such that the following holds for all $g \in \W_\Omega$. If $\tau(g) \geq t- \alpha$, then there exists $g^* \in \{ f: \Vert f-g \Vert_\infty \leq \eta\} \cap \W_{\Omega}$ such that $\tau(g^*) \geq t+ \beta$.
\end{definition}

\begin{lemma}\label{p-star} Let $m \in \mathbb{Z}^+$, $\gamma \in \Delta_m$, $W_0 \in \mathcal{B}^\gamma$, and let $\tau$ be a continuous graph parameter that has the sufficient increase property on $\W_\Omega$.
 Then $\phi$ is continuous on $\R\setminus\{t^{\tau}_{\max}(\tW_{\Omega})\}$.
\end{lemma}

In \cite{Lubetzky2015}, Lubetzky and Zhao studied the variational problem \eqref{min_eq} when $W_0$ is a constant graphon. They defined a ``nice graph parameter'' as a graph parameter $\tau$ that is (i) continuous with respect to $\delta_\square$ and (ii) has the property that every local extremum of $\tau$ with respect to $L_\infty$ is necessarily a global extremum.
They showed that for any nice graph parameter $\tau$, $\phi_\tau$ is continuous in the setting where the base graphon $W_0$ is constant.  Their proof technique cannot be directly adapted to the setting where $W_0$ is a block constant graphon with a zero or one block. When $W_0$ take values zero or one, the entropy function $I_{W_0}$ can be infinite, creating a technical hurdle. In particular, it is not clear how to establish right continuity of  $\phi_\tau$ for arbitrary nice graph parameters. We instead use \Cref{defn:p-star} as a sufficient condition for the right continuity of $\phi_\tau$.

Before proving \Cref{p-star}, we establish the left continuity of $\phi_\tau$ without any assumptions on the block graphon $W_0$ or the continuous graph parameter $\tau$.

\begin{lemma}\label{phi left} Let $\tau$ be a continuous graph parameter, and let $m \in \mathbb{Z}^+$, $\gamma \in \Delta_m$, and $W_0 \in \mathcal{B}^\gamma$. Then
$\phi$ is left-continuous.
\end{lemma}
\begin{proof}
We first note that we may assume that  $t\leq t^{\tau}_{\max}(\tW_\Omega)$, since $\phi=\infty$ and hence constant above $t_{\max}$.
Let $t_n \nearrow t$. Since $\phi$ is non-decreasing in $t$,
the sequence  $ \phi(t_n)$ has a limit, and
$\lim_{n \to \infty} \phi(t_n) \leq \phi(t)$.
To prove an upper bound on $\phi(t)$, recall the definition \eqref{phi-defn} of $\phi(t)$ as a minimum.
For each $k \geq 1$, there exists $\ti{g}_k$ such that $\tau( \ti{g}_k)\geq t_k$ and
$J_{W_0}(\ti{g}_k)= \phi(t_k)$.
By the compactness of $\tW$, there exists a convergent subsequence $\ti{g}_{k_j}$ such that $\delta_\square(\ti{g}_{k_j}, \ti{g}) \to 0$ for some $\ti{g} \in \widetilde{\mathcal{W}}$. Since $\tau(\ti{g}_{k_j}) \geq t_{k_j}$ and $t_{k_j} \nearrow t$, it follows that $\tau(\ti{g}) \geq t$,
and thus $ \phi (t)\leq J_{W_0}(\ti{g})$.
Combined with
the
lower semi-continuity of $J_{W_0}$ (\Cref{lemma:semi-continuity-J}),
we get
$$ \phi (t)\leq
J_{W_0}(\ti{g})\leq
\liminf_{j \to \infty} J_{W_0} ( \ti{g}_{k_j})
=\liminf_{j \to \infty} \phi(t_{k_j})
=\lim_{n \to \infty} \phi(t_n). \qedhere
$$
\end{proof}

\begin{remark}\label{countable}
Since $\phi$ is left-continuous (\Cref{phi left}) and non-decreasing, $\phi$ can have at most countably many points of discontinuity.
\end{remark}

\noindent
We now prove \Cref{p-star}, which establishes the  continuity of $\phi_\tau$ when $\tau$ has the sufficient increase property.

\begin{proof}[Proof of \Cref{p-star}.] By \Cref{phi left}, it suffices to establish the right-continuity of $\phi$ at $t$.
By assumption,
 $t\neq t_{\max}$.  Since
 $\phi$ is constant on $(-\infty,t_{\min}]$ and $(t_{\max},\infty)$ (where it is $0$ and $\infty$, respectively),
 we may  assume that $t_{\min}\leq t <t_{\max}$.
  Consider  a sequence $t_n \searrow t$, and an arbitrary $\ve>0$.  We need to show that there exists $n$ sufficiently large such that $\phi(t_n) \leq \phi(t) +\ve$.

Let $\eta>0$ be such that if $f, g \in \W_\Omega$ and $\Vert f-g \Vert_\infty \leq \eta$, then $|I_{W_0}(f)- I_{W_0}(g)|< \ve$; \Cref{uniform-cont} guarantees the existence of such an $\eta$. Let $\ti{g} \in \tW_\Omega$ be such that $\tau(\ti{g})\geq t$
and $\phi(t)=J_{W_0}(\ti g)$. By definition of $J_{W_0}( \cdot)$ there exists a sequence $f_k \in \W_\Omega$ such that $$I_{W_0}(f_k) \to J_{W_0}(\ti{g}) \quad \text{ and } \quad \delta_\square(f_k, \ti{g}) \to 0.$$
Since $\tau$ has the sufficient increase property on $\W_\Omega$, there exist $\alpha, \beta>0$ such that 
$$\tau(h) \geq t- \alpha \implies \exists \, h^* \text{ with } \Vert h^*-h \Vert_\infty \leq \eta \text{ and } \tau(h^*) \geq t+ \beta.$$
Since $\tau$ is continuous in $( \tW_\Omega, \delta_\square)$ and $\tau(g)\geq t$, there exists $k_0$ sufficiently large such that for all $k\geq k_0$, $\tau(f_k)  \geq  t-\alpha$.
Thus, for all $k \geq k_0$, there exists $f_k'$ such that
$$\tau(f_k')  \geq  t+\beta \quad \text{ and } \quad \Vert f_k' - f_k \Vert_\infty \leq \eta.$$
The choice of $\eta$ implies
$$|I_{W_0}(f_k')-I_{W_0}(f_k)| \leq \ve.$$
By compactness of $(\tW_\Omega, \delta_\square)$, there exists a convergent subsequence such that  $\ti{f}_{k_j}' \to \ti{h}$ for some $\ti{h} \in \tW_\Omega$. Since $\tau$ is continuous with respect to $\delta_\square$, $\tau(\ti{f}_{k_j}') \to \tau(\ti{h})$, and so $\tau(\ti{h}) \geq t+\beta$. It follows that
\begin{align*}
    \phi(t +\beta) \leq J_{W_0}(\ti{h})& \leq \liminf_{j \to \infty} I_{W_0}(f_{k_j}')\leq \liminf_{j \to \infty} I_{W_0}(f_{k_j}) +\ve= J_{W_0}(\ti{g}) +\ve
    =\phi(t)+\ve.
\end{align*}
Taking $n$ sufficiently large such that $t_n \leq t+ \beta$ and noting $\phi(t_n) \leq \phi(t +\beta)$  yields the desired statement.
\end{proof}

Next we establish that homomorphism densities have the sufficient increase property.
To this end, we introduce the following graphon $g^{+\eta}$,
\begin{align}\label{eq:gplus_defn}
    g^{+\eta}(x,y) &= \begin{cases}
     g(x,y) & (x,y) \not\in \Omega\\
    \min \left\{g(x,y) + \eta, 1 \right\} & \text{otherwise}
    \end{cases}
\end{align}
and note the following fact.

\begin{fact}\label{fun fact}
Let $\ell, u,w,z \in \R$ with $u \geq 0$. Suppose that $\ell \geq uz$ and $\ell \geq w-z$. Then $\ell \geq \frac{uw}{u+1}$.
\end{fact}

\begin{proof}
Note that for all $z \in \R$, $uz \geq \frac{uw}{u+1}$ or $w-z \geq \frac{uw}{u+1}$. The fact follows directly.
\end{proof}

\begin{lemma} \label{star helper} Let $\tau= t(H, \cdot)$ where $H$ is a finite graph,  let $m \in \mathbb{Z}^+$, $\gamma \in \Delta_m$, and $W_0 \in \mathcal{B}^\gamma$, and let $t_{\max}$ denote $t^{\tau}_{\max}(\tW_\Omega)$.
Fix $\eta \in (0,1]$.  Then there exists $c= c(\gamma, H, \eta) \in (0,1]$ such that
$
\tau(g^{+\eta}) \geq \tau(g)+c(t_{\max}-\tau(g))^2
$
for all $g \in \W_\Omega$.
\end{lemma}

\begin{proof}
Define $g^{\text{max}}$ as follows
$$g^{\text{max}}(x,y)=
\begin{cases}
g(x,y) & (x,y) \in ([0,1]^2\setminus R)\cup \Omega^c\\
1 & \text{ otherwise},
\end{cases}
$$
where $R$ is the union of relevant blocks, as defined above.
Note that since $g=W_0$ on $[0,1]^2 \setminus \Omega$, $g^{\text{max}} \in \W_\Omega$. Also note that $\tau(g^{\text{max}})= \max_{f \in \W_\Omega}\tau(f)= t_{\max}.$

Let $e_H=|E(H)|$ and $v= |V(H)|$ be the numbers of edges and vertices in $H$.  Let
$$
d=
 \frac{t(H, g^{\text{max}})- t(H,g)}{e_H}= \frac{t_{\max}- \tau(g)}{e_H}.
 $$
Since the statement of the lemma is trivial if $\tau(g)=t_{\max}$  we may assume w.l.o.g. that
$d>0$.
The Counting Lemma \cite[Theorem 3.7]{Borgs2008}\cite[Lemma 10.23]{Lovasz2012}  implies that $\delta_\square(g^{\text{max}}, g)\geq d$, and so
$$d_\square(g, g^{\text{max}})\geq \delta_\square(g,g^{\text{max}}) \geq d.$$
Let $S=\{(x,y) \in [0,1]^2: g^{\text{max}}-g \geq d/2\}$.
Since $g^{\text{max}} \geq g$ pointwise,
\Cref{cut-sets} implies that $|S| \geq d/2$.
 Let $\eta'\triangleq \min\{\eta, d/2\}$. It follows that $g^{+\eta} -g \geq \eta'$ on $S$. By construction, $S \subseteq R$. Recall that $m$ denotes the number of blocks in $W_0$. Therefore, there are at most $m^2$ relevant blocks of $[0,1]^2$ of the form $I_i \times I_j$ for $i,j \in [m]$. Thus, there exists $a,b \in [m]$ such that $I_a \times I_b$ is relevant and $|(I_{a} \times I_{b})\cap S |\geq  d/(2m^2)$.

It suffices to show that increasing $g$ to $g^{+\eta}$ on $(I_{a} \times I_{b})\cap S$ yields a constant increase in the homomorphism density. Since $I_{a} \times I_{b}$ is a relevant block, there exists a relevant $Y \in [m]^v$ such that $Y_p=a$, $Y_q=b$ for some $\{p,q\} \in E(H)$.

Define
$$Z_{\overline{S}}(H,g,Y)= \int_{x_1 \in I_{Y_1}}  \dots \int_{x_v \in I_{Y_v}} \prod_{\{i,j\} \in E(H)} g(x_i, x_j) \1\{(x_p, x_q) \not \in S\} \, dx_v \dots dx_1. $$
$$Z_{S}(H,g,Y)= \int_{x_1 \in I_{Y_1}}  \dots \int_{x_v \in I_{Y_v}} \prod_{\{i,j\} \in E(H)} g(x_i, x_j)\1\{(x_p, x_q) \in S\} \, dx_v \dots dx_1. $$
In other words, $Z_{S}(H,g,Y)$ accounts for the  homomorphisms in which the $\{p,q\}$ edge lies in $S$, and $Z_{\overline{S}}(H,g,Y)$ accounts for the homomorphisms in which the $\{p,q\}$ edge does not lie in $S$.
Therefore,
$$t(H,g,Y)= Z_{\overline{S}}(H,g,Y)+Z_{S}(H,g,Y).$$
Since $g^{+\eta}\geq g$ pointwise, $Z_{\overline{S}}(H,g^{+\eta},Y) \geq Z_{\overline{S}}(H,g,Y)$. Next we derive two lower bounds on $Z_{S}(H,g^{+\eta},Y)$.

First observe
\begin{align*}
    Z_{S}(H,g^{+\eta},Y)&= \int_{x_1 \in I_{Y_1}} \dots \int_{x_v \in I_{Y_v}} \prod_{\{i,j\} \in E(H)} g^{+\eta}(x_i, x_j)\1\{(x_p, x_q) \in S\} \, dx_v \dots dx_1\\
    &= \int_{x_1 \in I_{Y_1}}  \dots \int_{x_v \in I_{Y_v}} g^{+\eta}(x_p, x_q) \prod_{\substack{\{i,j\} \in E(H)\\\{i,j\} \not = \{p,q\}}} g^{+\eta}(x_i, x_j)\1\{(x_p, x_q) \in S\} \, dx_v \dots dx_1\\
    &\geq\int_{x_1 \in I_{Y_1}}  \dots \int_{x_v \in I_{Y_v}} (g(x_p, x_q)+\eta') \prod_{\substack{\{i,j\} \in E(H)\\\{i,j\} \not = \{p,q\}}} g(x_i, x_j)\1\{(x_p, x_q) \in S\} \, dx_v \dots dx_1\\
    &\geq Z_{S}(H,g,Y)( 1+\eta').
    \end{align*}
The final inequality follows from noting that $g + \eta' \geq (1+ \eta')g$.
Our goal is to lower bound the  difference  $t(H,g^{+\eta})- t(H,g)$ by a constant. The above computation implies that  $t(H,g^{+\eta})- t(H,g) \geq \eta'Z_{S}(H,g,Y)$. This lower bound may not be sufficient if $Z_{S}(H,g,Y)$ is too small. We derive another lower bound for this case.

Let $\beta= \min_{j \in [m]} |I_j|$. Observe that for all $\{i,j\}\in E(H)$,  $g^{+\eta}(x_i,x_j)\geq \eta$ when $x_i \in I_{Y_i}$ and $x_j \in I_{Y_j}$, as $Y$ is relevant. It follows that $$Z_S(H, g^{+\eta},Y) \geq \eta^{e_H} |S\cap( I_{a} \times I_{b})|\prod_{j \in [m]\setminus \{p,q\}} |I_{Y_j}| \geq \eta^{e_H} |S\cap( I_{a} \times I_{b})| \beta^{v-2}.$$
We have shown that
$$Z_{S}(H,g^{+\eta},Y)- Z_{S}(H,g,Y) \geq \eta' Z_{S}(H,g,Y)$$
and
$$Z_{S}(H,g^{+\eta},Y)- Z_{S}(H,g,Y) \geq \eta^{e_H} |S\cap( I_{a} \times I_{b})| \beta^{v-2} - Z_{S}(H,g,Y).$$
Applying Fact \ref{fun fact} with $u=\eta', w= \eta^{e_H}|S\cap( I_{a} \times I_{b})|\beta^{v-2}, z=Z_{S}(H,g,Y) , $ and $\ell=Z_{S}(H,g^{+\eta},Y)- Z_{S}(H,g,Y)$, we obtain $$Z_{S}(H,g^{+\eta},Y)- Z_{S}(H,g,Y) \geq  \frac{\eta'(\eta^{e_H}|S\cap( I_{a} \times I_{b})|\beta^{v-2})}{\eta'+1}.$$
We now simplify our lower bound using the facts that $|S\cap( I_{a} \times I_{b})| \geq d/(2m^2)$, $\eta'/(1+\eta') \geq \eta'/2$,  $\eta' \geq \eta d/2,$  and
$d= (t_{\max}-\tau(g))/e_H$, obtaining
$$\frac{\eta'(\eta^{e_H}|S\cap( I_{a} \times I_{b})|\beta^{v-2})}{\eta'+1}\geq \frac{\eta^{e_H+1} d^2\beta^{v-2}}{8m^2} = \frac{\eta^{e_H+1} (t_{\max}-\tau(g))^2\beta^{v-2}}{8e_H^2 m^2}. $$

Set $c=\frac{\eta^{e_H+1} \beta^{v-2}}{8e_H^2 m^2}.$ Note that $c$ is a function only of $H$, $\gamma$, and $\eta$.
It follows that
\begin{align*}
    t(H,g^{+\eta},Y)&= Z_{\overline{S}}(H,g^{+\eta},Y)+Z_{S}(H,g^{+\eta},Y)\\
    &\geq Z_{\overline{S}}(H,g,Y)+Z_{S}(H,g,Y)+c(t_{\max}-\tau(g))^2\\
    &= t(H,g,Y) +c(t_{\max}-\tau(g))^2.
\end{align*}
Thus,
$t(H,g^{+\eta}) \geq t(H,g) +c(t_{\max}- \tau(g))^2$.
\end{proof}

Next we establish that $\phi_\tau$ is continuous when $\tau$ is a homomorphism density by using \Cref{star helper} to show that homomorphism densities have the sufficient increase property.
\begin{proof}[Proof of \Cref{homomorphism-cont}.]
By \Cref{p-star} it suffices to show that $\tau$ has the sufficient increase property.
Fix $\eta$ and $t$. Let $\alpha = \beta = \frac{c}{8}(t_{\max} -t)^2$, where $c \in (0,1]$ is from \Cref{star helper}. If
$\tau(g)> t+\frac 12 (t_{\max}-t)$, we can choose $g^*=g$.  Next, suppose
$\tau(g)\leq t+\frac 12 (t_{\max}-t)$. But then
$
t_{\max}-\tau(g)\geq \frac 12 (t_{\max}-t)$ and by \Cref{star helper},
$
\tau(g^{+\eta})\geq\tau(g)+\frac c4 (t_{\max}-t)^2.
$
The assumption $\tau(g)\geq t-\alpha$ then implies
$$
\tau(g^{+\eta})\geq t+\frac c8 (t_{\max}-t)^2 = t + \beta.
$$
Noting that $\beta\leq \frac 12  (t_{\max}-t)^2\leq  \frac 12  (t_{\max}-t)$
completes the proof.
\end{proof}
\noindent
Next we show  that the operator norm has the sufficient increase property. We begin with the following lemma.

\begin{lemma}\label{operator helper}
Let $\tau(g)= \Vert g \Vert_{\emph{\op}}$ and let $W_0 \in \mathcal{B}^{(\gamma, 1-\gamma)}$ where $\gamma \in (0,1)\cap \mathbb{Q}$ be a bipartite graphon. Let $g^{+\eta}$ be as defined in \eqref{eq:gplus_defn}. Fix any $\eta \in (0,1]$ and $g \in \W_{\Omega}$ such that  $\tau(g) < t^{\tau}_{\max}(\tW_\Omega)$. Then
$$\tau(g^{+\eta})^2 \geq
\max\left\{ \eta^2 t^2_{\max},\left(1 + \frac{\eta^4}{2^{25}} \brac{t_{\max} - \tau(g)}^{20}\right)\tau(g)^2\right\}.$$
\end{lemma}

\begin{proof}
Recalling the definition of a bipartite graphon $f_p^{\gamma}$ from \Cref{b-er-results}, we
note that except for the trivial case $p\in \{0,1\}$ (in which case $\W_\Omega=\{W_0\}$
and $\tau(g)=t_{\max}$ for all $g\in \W_\Omega$), the set $\Omega$ is  $[0,\gamma]\times (\gamma,1]\cup  (\gamma,1]\times [0,\gamma]$.
Let $h^{\text{max}} =f_1^\gamma$ be the graphon that takes value $1$ on $\Omega$ and agrees with $W_0$ on $\Omega^c$ (where both are $0$).
Note that if $f \leq g$ pointwise, then $\Vert f\Vert_{\text{op}} \leq \Vert g \Vert_{\text{op}}$. It follows that
$\tau(h^{\text{max}})= t_{\max}$. The graphon $h^{\text{max}}$ satisfies $f \leq h^{\text{max}}$ for every $f \in \mathcal{W}_{\Omega}$.

To prove the first lower bound, we note that $g^{+\eta}\geq f_\eta^\gamma$ pointwise, implying that
 $\Vert g^{+\eta}\Vert_{\text{op}}\geq \Vert f_\eta^\gamma\Vert_{\text{op}}=
\eta\Vert f_1^\gamma\Vert_{\text{op}}=\eta t_{\max}$.

To prove the second lower bound, we note that it follows from the proof of \cite[Lemma 3.6]{Lubetzky2015} that for $f, g \in \W$
$$\frac{( \Vert f \Vert_{\op}- \Vert g \Vert_{\op})^4}{4}\leq \delta_\square (f,g).$$
Let $d=\frac{( t_{\max}- \Vert g \Vert_{\op})^4}{4} $. It follows that
$$d \leq \delta_\square (h^{\text{max}},g) \leq d_\square(h^{\text{max}},g).$$

Next, let $\eta'= \min\{\eta, d/2\}$ and $S= \{(x,y) \in \Omega: h^{\text{max}}(x,y)- g(x,y) \geq \eta' \}.$ Since $h^{\text{max}}=g$ on $[0,1]^2 \setminus \Omega$ and $h^{\text{max}} \geq g$ pointwise, \Cref{cut-sets} implies that $|S| \geq d/2$.

Since $T_g$ is a self-adjoint compact nonnegative linear operator, there exists $u \in L_2([0,1])$ such that $u(x) \geq 0$, $\Vert u \Vert_2 =1$ and $T_gu(x)= \Vert g \Vert_\op u(x)$ for almost all $x \in [0,1]$  \cite[Proposition 2.11]{Chatterjee2015}. Let $P \subseteq [0,1]$ be a measure one subset where this proposition holds. We will derive a lower bound on $\tau(g^{+\eta})^2$ by showing that for some $c>0$,
\begin{align}
    T_{g^{+\eta}}u(x) &\geq T_{g}u(x) +c, \label{case-goal1}
\end{align}
for $x$ in some subset of $[0,1]$. The construction of this subset will depend on $S$ and $u$. For ease of notation, we let $\ve= \eta / 2$. Define
$$A_\ve= \{ x \in [0,1]: u(x) \geq \ve\} \quad \text{ and } \quad A_\ve^c = [0,1] \setminus A_\ve.$$
We will consider two cases that depend on the size of $A_\ve^c$. In each case we find a subset of $[0,1]$ and $c>0$ satisfying \eqref{case-goal1}.

Before proceeding to the cases, we establish a useful property of $u$.
Define $$u_1 = \int_{0}^\gamma u(x) dx \quad \text{ and } \quad u_2 = \int_{\gamma}^1 u(x) dx.$$
Since $\Vert u \Vert_2=1$, there exists a subset of $[0,1]$ with positive measure where $u \geq 1$ on the subset. Let $z$ be an element of the intersection of this subset with $P$. We assume that $z \in [0, \gamma]$. After completing the proof under this assumption, we will discuss how a similar argument applies when $z \in (\gamma, 1]$. Since $z \in [0,\gamma]$, $g$ is zero on $\{z\} \times [0,\gamma]$ and so
\begin{align}\label{u2 bound}
    \Vert g \Vert_\op \leq \Vert g \Vert_\op u(z)= T_gu(z) = \int_0^1 g(z,y) u(y) dy = \int_\gamma^1 g(z,y) u(y) dy \leq u_2.
\end{align}

\noindent \textbf{Case 1:} $ \vert A_\ve^c \cap [0, \gamma] \vert \geq d/16$.\\
Let $x \in A_\ve^c \cap [0, \gamma] \cap P$. Observe
$$T_gu(x)= \Vert g \Vert_\op u(x) \leq \Vert g \Vert_\op \ve = \frac{ \eta \Vert g \Vert_\op }{2}.$$
Note that $g^{+\eta}(x,y) \geq \eta$ for all $y \in (\gamma,1]$ and $g^{+\eta}(x,y) =0 $ for all $y \in [0, \gamma]$ by construction. It follows that
$$T_{g^{+\eta}}u(x) = \int_{\gamma}^1 g^{+\eta}(x, y) u(y) dy\geq \eta u_2 \geq \eta \Vert g \Vert_\op.$$
Thus for all $x \in A_\ve^c \cap [0, \gamma] \cap P$ $$T_{g^{+\eta}}u(x) - T_gu(x) \geq \frac{\eta \Vert g \Vert_\op}{2}.$$
Observe that
\begin{align*}
    \tau(g^{+\eta})^2 &= \Vert g^{+\eta} \Vert_\op^2\geq \Vert T_{g^{+\eta}}u \Vert_2^2= \int_{0}^1 (T_{g^{+\eta}}u(x))^2 dx\\
    &\geq \int_{ A_\ve^c \cap [0, \gamma]} \brac{T_gu(x)+ \frac{\eta \Vert g \Vert_\op}{2}}^2 dx + \int_{( A_\ve^c \cap [0, \gamma] )^c}T_gu(x)^2 dx\\
    &\geq\Vert T_{g}u \Vert_2^2 + |A_\ve^c \cap [0, \gamma]|\frac{\eta^2 \Vert g \Vert_\op^2}{4}\geq
 \tau(g)^2 + \frac{d\eta^2 \Vert g \Vert_\op^2}{64}.
\end{align*}

\noindent \textbf{Case 2:} $ \vert A_\ve^c \cap [0, \gamma] \vert < d/16$.\\
Recall $|S| \geq d/2$.
For each $x \in [0,1]$, define $S_x= \{y \in [0,1]: (x,y)\in S\}$, and let $X=\{x \in (\gamma,1]: |S_x| > d/8\}$. Since $g$ is symmetric, $(x,y) \in S$ if and only if $(y,x) \in S$.  It follows that
\begin{align}\label{x eq}
    \frac{d}{4}\leq \frac{|S|}{2}= \int_\gamma^1 |S_x| \leq |X| + \frac{d}{8}(1-\gamma -|X|) \leq  |X| + \frac{d}{8}(1 -|X|) \leq  \frac{d}{8} + |X|.
\end{align}
Therefore $|X|  \geq d/8$.
Note that for all $x \in X$, $|S_x \cap A_\ve| >  d/16$ because $|S_x| > d/8$, $S_x \subseteq [0, \gamma]$ and $|A_\ve^c \cap [0,\gamma]|< d/16$. Recall that for $y \in S_x$, $g^{+\eta}(x,y) \geq  g(x,y)+ \eta'$ and that for $y \in A_\ve$, $u(y) \geq \ve$.
Therefore, for all $x \in X$,
\begin{align*}
    T_{g^{+\eta}}u(x) &= \int_0^1 g^{+\eta}(x,y) u(y) dy
\geq \int_{ S_x \cap A_\ve} (g(x,y)+ \eta') u(y) dy + \int_{( S_x \cap A_\ve)^c} g(x,y) u(y) dy\\
&\geq T_gu(x) + \ve \eta' |S_x \cap A_\ve|
\geq T_gu(x) + \frac{d\ve \eta'}{16}.
\end{align*}
Finally, observe that
\begin{align*}
    \tau(g^{+\eta})^2
    &\geq \int_{0}^1 (T_{g^{+\eta}}u(x))^2 dx
    \geq \int_{ X} \brac{ T_gu(x) +\frac{d \ve \eta'}{16}}^2 dx + \int_{ X^c}T_gu(x)^2 dx\\
    &\geq \Vert T_{g}u \Vert_2^2 + |X|\bfrac{d \ve \eta'}{16}^2
    \geq \tau(g)^2 + \frac{d^3\ve^2 \eta'^2}{2^{11}}.
\end{align*}

\noindent Recalling that $\ve= \eta/2$, $\eta' \geq \eta d/2$, and $d= ( t_{\max} - \tau(g))^4/4$, we see that in
the first case, we have
\begin{align*}
\tau(g^{+\eta})^2- \tau(g)^2 \geq \frac{d \eta^2 \Vert g \Vert_{\op}^2}{64} = \frac{\eta^2 (t_{\max} - \tau(g))^4 \tau(g)^2}{2^{8}}
\end{align*}
and in the second case,
\begin{align*}
\tau(g^{+\eta})^2- \tau(g)^2 \geq \frac{d^3 \ve^2 \eta'^2}{2^{11}} \geq \frac{d^3 \eta^2}{2^{13}} \left(\frac{\eta d}{2}\right)^2 = \frac{\eta^4 d^5}{2^{15}} = \frac{\eta^4 (t_{\max} - \tau(g))^{20}}{2^{25}}.
\end{align*}
Therefore, in both cases
$$\tau(g^{+\eta})^2- \tau(g)^2 \geq \frac{\eta^4}{2^{25}} ( t_{\max} - \tau(g))^{20}\tau(g)^2.$$

Finally we revisit our assumption prior to the case work that the value $z \in P$ such that $u(z) \geq 1$ is in $[0,\gamma]$. Suppose instead that $z \in (\gamma, 1]$. The equation analogous to  \eqref{u2 bound} implies that $u_1 \geq \Vert g \Vert_\op$. Now switching the roles of $[0, \gamma]$ and $(\gamma, 1]$ in Case 1 gives the same lower bounds on $\tau(g^{+\eta})^2$.
\end{proof}

\noindent
Next we establish the continuity of  $\phi_\tau$ when $\tau$ is the operator norm and $W_0$ is a bipartite graphon with $W_0 \in \mathcal{B}^{(\gamma, 1-\gamma)}$. We use \Cref{operator helper} to establish the sufficient increase property.

\begin{proof}[Proof of \Cref{operator-cont}.]
By \Cref{p-star}, it suffices to show that $\tau$ has the sufficient increase property.
Fix $\eta$ and $t$. If $t=0$, the choice $\beta=\eta t_{\max}$ and the first bound in  \Cref{operator helper} implies that $\tau(g^{+\eta})\geq \beta$.

 If $t>0$, the proof proceeds along the same lines as the proof of \Cref{homomorphism-cont}.
 Set $\alpha=\min\{t/2,ct/4\}$ and $\beta=\min\{ct/4, (t_{\max}-t)/2\}$, where \[c=\sqrt{1 + \frac{\eta^4}{2^{25}} \brac{t_{\max} - t}^{20}}-1.\] As before, the required bound is easy if
$\tau(g)> t+\frac 12 (t_{\max}-t)$.  In this case, we again choose $g^*=g$. Next, suppose
$\tau(g)\leq t+\frac 12 (t_{\max}-t)$.  But then
$
t_{\max}-\tau(g)\geq \frac 12 (t_{\max}-t)$ and by the second bound in \Cref{operator helper} ,
$\tau(g^{+\eta}) \geq
(1+c)\tau(g)$. 
The assumption
$\tau(g)\geq t-\alpha$ implies
$$
\tau(g^{+\eta}) \geq (t-\alpha) +(t-\alpha)c\geq t-ct/4+ct/2\geq t+ct/4\geq t+\beta.
$$
\end{proof}

\subsection{Properties of $\phi_\tau$ at points of continuity}\label{prop-of-phi}
\Cref{phi is h} below states that at points of continuity, the function $\phi_\tau(W_0,t)$ can be alternatively expressed as the minimum of $I_{W_0}$ over a subset of $\W$. We use this characterization of $\phi_\tau$ to establish that $\phi_\tau$ is strictly increasing when $\tau$ is an increasing, uniformly continuous graph parameter and $\phi_\tau$ is continuous (\Cref{strictly-increasing}).

\begin{lemma} \label{phi is h} Fix $m \in \mathbb{Z}^+$, $\gamma \in \Delta_m$, $W_0 \in \mathcal{B}^\gamma$, and let  $t\in\R$.
If $\phi_{\tau}(W_0,\cdot)$ is continuous at $t$, then
\begin{equation}
    \phi_{\tau}(W_0,t) = \inf \{I_{W_0}(f) : f \in \mathcal{W}, \tau(f) \geq t\}.
\end{equation}
\end{lemma}

\begin{lemma}\label{strictly-increasing}
Fix $m \in \mathbb{Z}^+$, $\gamma \in \Delta_m$, $W_0 \in \mathcal{B}^\gamma$, and let $\tau$ be a graph parameter that is uniformly continuous (with respect to $\delta_\square$) and increasing, meaning that if $ f \geq g$ pointwise then $\tau(f) \geq \tau(g)$. Suppose that
$\phi_{\tau}(W_0, \cdot)$ is continuous
on the open interval $(\tau(W_0), t^{\tau}_{\max}(\tW_\Omega))$.
Then $\phi_{\tau}(W_0,\cdot)$
is strictly increasing on
$[\tau(W_0), t_{\max}]$.

Therefore, if $\tau$ is a uniformly continuous increasing graph parameter, $\phi_\tau(W_0,\cdot)$ is continuous on $(\tau(W_0), t_{\max})$,  and $\ti{f}$ is a minimizer of the variational problem \eqref{phi-defn}, then $\tau(f)=t$
for all $t\in [\tau(W_0), t_{\max}]$.
\end{lemma}

\begin{remark}\label{equality at min}
Note that the Counting Lemma
\cite[Lemma 10.23]{Lovasz2007} and \cite[Lemma 3.6]{Lubetzky2015} imply that homomorphism densities and the operator norm are each uniformly continuous with respect to $\delta_\square$. 
\end{remark}

\begin{proof}[Proof of \Cref{phi is h}.]
Let $h(t)= \inf\{I_{W_0}(f) : f \in \W, \tau(f) \geq t\}$.
It is clear from the definition that $\phi \leq h$. We will show that the right continuity of $\phi$ at $t$ implies that $\phi(t) \geq h(t)$.

To this end, we claim that for all $a \in \R$,
\begin{align} \label{strict}
\inf\{J_{W_0} ( \ti{f}) : \tau(\ti{f}) > a\}=\inf\{I_{W_0} ( f) : \tau(f) > a\}.
\end{align}
Indeed, it is clear the left hand side is at most the right hand side.
Since both sides are infinite if the set  $\{\ti f\in \tW_\Omega:\tau(\ti f)>a\}$ is empty,  we may assume that this set contains at least one
$\ti{f}\in\tW_\Omega$ such that $\tau(\ti{f}) > a$.
 By the definition of $J_{W_0}$, there exists $g_n$ such that $\delta_\square(g_n , \ti{f}) \to 0$ and $I_{W_0}(g_n) \to J_{W_0}( \ti{f})$. By continuity of $\tau$, there exists $n_0$ sufficiently large such that for all $n >n_0$, $\tau(g_n) >a$. Thus $$J_{W_0}(\ti{f}) \geq \inf \{ I_{W_0}(h) : \tau(f) >a\},$$ and so \eqref{strict} follows.

We now turn to showing that if $\phi$ is right continuous at $t$, then $\phi(t) \geq h(t)$. Applying \eqref{strict}, we obtain
\begin{align*}
\phi(t + \ve) &\geq \inf \{ J_{W_0}(\ti{f}) : \tau(\ti{f}) > t+ \ve/2\}\\
&=\inf\{ I_{W_0}(f) : \tau(f) > t + \ve/2\}\\
&\geq \inf \{ I_{W_0}(f) : \tau(f) \geq t\}\\
&= h(t)
\end{align*} for any $\ve>0$.
It follows by right continuity of $\phi$ at $t$ that $$\phi(t) = \lim_{\ve \to 0} \phi(t +\ve) \geq h(t).$$\end{proof}

\begin{lemma} \label{increasing-helper}
Fix $m \in \mathbb{Z}^+$, $\gamma \in \Delta_m$, and $W_0 \in \mathcal{B}^\gamma$,
and let $f, g \in \W_\Omega$. For any $\ve > 0$, there exists $\eta= \eta(\ve,W_0)>0$ such that if $f \geq g \geq W_0$ pointwise and $d_\square(f,g) \geq \ve$, then
$$I_{W_0}(g) \leq I_{W_0}(f)-\eta.$$
\end{lemma}

\begin{proof}
We may assume that $\ve < 2(1-p)$ for $p \in Im(W_0) \setminus \{0,1\}$.
Recall the definition
$$S_{\ve/2}^+= \{(x,y) \in [0,1]^2 : f-g \geq \ve/2\}.$$
\Cref{cut-sets} implies that $|S_{\ve/2}^+| >\ve/2$.
Let $$\eta'= \min_{p \in Im(W_0)\setminus \{0,1\}} \min_{x \in [p,1-\ve/2]} [h_p(x +\ve/2)- h_p(x)]>0.$$
Since $f \geq g \geq W_0$ and $h_{p}(\cdot)$ is increasing on $[p,1]$, we have that
$$
h_{W_0(x,y)}(f(x,y)) \geq h_{W_0(x,y)}(g(x,y) +\ve/2) \geq h_{W_0(x,y)}(g(x,y)) +\eta'
$$
for all $(x,y) \in S_{\ve/2}^+$.  As a consequence,
\begin{align*}
    I_{W_0}(f) &= \int_{[0,1]^2} h_{W_0(x,y)}(f(x,y))\, dx \,dy\\
    &\geq \int_{[0,1]^2\setminus S_{\ve/2}^+} h_{W_0(x,y)}(g(x,y))\, dx \,dy+ \int_{S_{\ve/2}^+} (h_{W_0(x,y)}(g(x,y))+\eta')\, dx \,dy \\
    &\geq I_{W_0}(g) +\eta' |S_{\ve/2}^+| \geq I_{W_0}(g)+\frac{\ve\eta'}{2}.
\end{align*}
Taking $\eta= \eta' \ve/2$ completes the proof.
\end{proof}

\begin{proof}[Proof of \Cref{strictly-increasing}]
The lemma is trivial if $\tau(W_0) =t^{\tau}_{\max}(\tW_\Omega)$ so we may assume that
$\tau(W_0) < t_{\max}$.  Furthermore, since $\phi_\tau(W_0,\cdot)$ is
non-decreasing, it is enough to prove that it is strictly increasing on the open interval
$(\tau(W_0), t_{\max})$.
Let $\tau(W_0)< t_1 < t_2<t_{\max}$. We will prove that $\phi_\tau(W_0,t_1)< \phi_\tau(W_0,t_2)$ by applying \Cref{phi is h} and showing that
$$ \inf\{I_{W_0}(f): \tau(f) \geq t_1\} <  \inf\{I_{W_0}(f): \tau(f) \geq t_2\}.$$
To establish the above statement, it suffices to show that there exists $\eta= \eta(t_1,t_2)>0$ such that the following is true. If $f \in \W_\Omega$ is such that $\tau(f) \geq t_2$, then there exists $g' \in \W_\Omega$ such that $\tau(g') = t_1$ 
and $I_{W_0}(g') \leq I_{W_0}(f)-\eta$.

Given $f \in \mathcal{W}_{\Omega}$ with $\tau(f) \geq t_2$, we will define $g$ and $g'$ satisfying $\delta_{\square}(g, g') \geq \beta$ where $\beta$ is a function of $t_2 - t_1$. We will then show that $I_{W_0}(g) \leq I_{W_0}(f)$ and $I_{W_0}(g') \leq I_{W_0}(g) - \eta$, for $\eta = \eta(\beta)$.
Define $g \in \W_\Omega$ such that $g(x,y) = \max\{ W_0(x,y), f(x,y)\}$. Since $g \geq f$ pointwise and $\tau$ is increasing, $\tau(g) \geq \tau(f) \geq t_2$. Moreover $g \geq W_0$ pointwise. Next define $g_\alpha \in \W_\Omega$ where
\[g_\alpha(x,y) = W_0(x,y) + \alpha(g(x,y)-W_0(x,y)).  \]

By construction $g_1=g $ and $g_0=W_0$, and so
$\tau(g_1)\geq t_2$, and $\tau(g_0) \leq t_1$. Since $\tau( g_\alpha)$ decreases continuously as $\alpha \to 0$, there exists some $\alpha_0$ such that $\tau(g_{\alpha_0}) = t_1$.
Let $g'= g_{\alpha_0}$. Since $\tau(g)- \tau(g') \geq t_2-t_1$, the uniform continuity of $\tau$ implies that $\delta_\square(g,g') \geq \beta$ for some positive $\beta= \beta(t_2-t_1)$. It follows that $d_\square(g,g') \geq \beta$. Note also that $g,g' \in \W_\Omega$ and $g \geq g'$ pointwise. Therefore \Cref{increasing-helper} implies that there exists $\eta=\eta(\beta)$ such that $I_{W_0} (g') \leq I_{W_0} (g) -\eta.$ By construction, $I_{W_0}(g) \leq I_{W_0}(f)$ and so $I_{W_0} (g') \leq I_{W_0} (f) -\eta$.
\end{proof}

\section{The symmetric regime in general block models}
\label{sec:symmetry}
In this section we prove \Cref{sym-regime,sym-regime-near-one}, which establish that for any $d$-regular graph $H$ and $W_0\in \mathcal{B}^\gamma$, there exists a symmetric regime for $W_0$ and $t(H, \cdot)$.

\paragraph{Notation.}
Throughout this section we fix a particular $d$-regular graph $H$ with $d\geq 2$. Since we consider $d$ to be fixed, we suppress the dependence on $d$ in our notation.

\begin{definition} Let $p \in (0,1)$ and $d \geq 2$. We define $\psi_p: [0,1] \to \R$ as
$$\psi_p(x) = h_p(x^{1/d}),$$
and let $\hat{\psi}_p(x)$ denote the convex minorant of $\psi_p(x)$.
\end{definition}
\Cref{all-about-psi} in the Appendix collects some useful properties of $\psi_p$.

Next, we introduce notation that allows us to reason about individual blocks within a graphon.
Let $m \in \mathbb{Z}^+$ and let $\gamma \in \Delta^m$ be a vector of interval widths that determines block membership. Let $I_j$ for $j \in [m]$ be as given in \Cref{block-graphon}.
For $x \in [0,1]$, recall that $\vartheta(x)$ denotes the membership of $x$, so that  $x \in I_{\vartheta(x)}$.

Let $f \in \W$. For each $i,j \in [m]$, we define a
 function $f_{ij}:[0,1]^2\to[0,1]$
 that describes $f$ restricted to the block $ I_i \times  I_j$ by \footnote{
 Strictly speaking, the function $f_{ij}$ contains a little more information than is contained in $f$ restricted to $I_i\times I_j$, namely, it represents the function $f$ restricted to the closure of $I_i\times I_j$.    But this  difference only appears on a set of measure zero, and is thus inconsequential; furthermore, the relation \eqref{r-defn-2} holds for all $(x,y)\in [0,1]^2$. }
 \begin{equation}\label{fij}
 f_{ij}(x,y)=f\left(\sum_{k=0}^{i-1}\gamma_k+x\gamma_i,\sum_{k=0}^{j-1}\gamma_k +y\gamma_j\right)
\end{equation}
with $\gamma_0 = 0$.  We write $f = (f_{ij})_{i,j \in [m]}$ to indicate that
\begin{align}\label{r-defn-2}
f(x,y) &= f_{\vartheta(x), \vartheta(y)}(r(x), r(y)),
\end{align}
where
\begin{align}\label{r-defn}
r(x) = \frac{x - \sum_{i=0}^{\vartheta(x) - 1} \gamma_i}{\gamma_{\vartheta(x)}}
\end{align}
 (see Figure \ref{fig:block-graphon}). By an abuse of notation, when a graphon $f$ takes constant values on the blocks defined by $\gamma$ (as in  \Cref{block-graphon}), we write $f=(f_{ij})_{i,j \in [m]}$ where each $f_{ij} \in \mathbb{R}$ is a constant rather than a constant function.

In this section, we will utilize the restricted functions $f_{ij}$. Note that in contrast to the original graphon $f \in \W$, the $f_{ij}$ functions are not necessarily symmetric. However, we will continue to use the cut distance $d_{\square}$ on these functions. In particular, we will crucially use Lemma \ref{cut-sets}---we note that the proof of this result does not utilize the symmetry of the functions, and thus continues to hold in this extended setting.

\tikzset{blockr/.style= {rectangle, draw=black!50, fill=black!20, thick}}
\tikzset{blockr/.style= {rectangle, draw=black!50, fill=black!20, thick}}
\tikzset{blockr/.style= {rectangle, draw=black!50, fill=black!20, thick}}
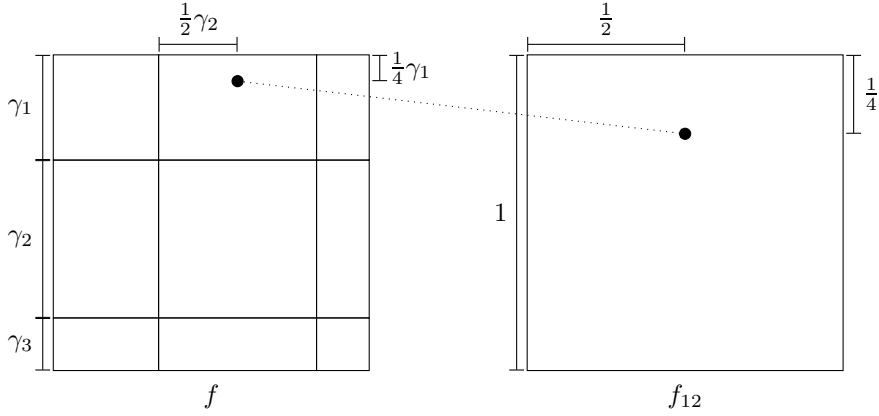
\begin{figure}[h]
\centering
\begin{tikzpicture}[scale=0.7]
\draw [] (0,0) rectangle node {} (2,-2);
\draw [] (2,0) rectangle node {} (5,-2);
\draw [] (5,0) rectangle node {} (6,-2);
\draw [] (0,-2) rectangle node {} (2,-5);
\draw [] (2,-2) rectangle node {} (5,-5);
\draw [] (5,-2) rectangle node {} (6,-5);
\draw [] (0,-5) rectangle node {} (2,-6);
\draw [] (2,-5) rectangle node {} (5,-6);
\draw [] (5,-5) rectangle node {} (6,-6);
\draw [|-|] (-0.2,0) -- (-0.2,-2) node[midway,left] {$\gamma_1$};
\draw [|-|] (-0.2,-2) -- (-0.2,-5) node[midway,left] {$\gamma_2$};
\draw [|-|] (-0.2,-5) -- (-0.2,-6) node[midway,left] {$\gamma_3$};
\filldraw (3.5,-0.5) circle (3pt);
\draw [|-|] (2,0.2) -- (3.5,0.2) node[midway,above] {$\frac{1}{2} \gamma_2$};
\draw [|-|] (6.2,0) -- (6.2,-0.5) node[midway,right] {$\frac{1}{4} \gamma_1$};
\draw [] (9,0) rectangle node {} (15,-6);
\filldraw (12,-1.5) circle (3pt);
\draw [|-|] (8.8,0) -- (8.8,-6) node[midway,left] {$1$};
\draw [|-|] (9,0.2) -- (12,0.2) node[midway,above] {$\frac{1}{2}$};
\draw [|-|] (15.2,0) -- (15.2,-1.5) node[midway,right] {$\frac{1}{4}$};
\draw[dotted,->] (3.5,-0.5) -- (12,-1.5);
\node at (3,-6.5) {$f$};
\node at (12,-6.5) {$f_{12}$};
\end{tikzpicture}
\caption{Illustration of the graphon $f = (f_{ij})_{i,j \in [m]}$ for $m = 3$. The indicated point in the graphon $f$ is in the $(1,2)$ block. The point is mapped to a point in the function $f_{12}$, with a scaled position.}\label{fig:block-graphon}
\end{figure}

\begin{definition}\label{defn-averaged}
Let $d>0$ and $m \in \mathbb{Z}^+$. For a graphon $f = (f_{ij})_{i,j \in [m]}$, we define the corresponding $d$-averaged block constant graphon $$f^{*} = (\|f_{ij}\|_d)_{i,j \in [m]},$$  where the $d$-norm is defined as $\Vert g \Vert_d = \left(\int_{[0,1]^2} g(x,y)^d dx dy \right)^{\frac{1}{d}}$.
\end{definition}

\begin{definition}
\label{def:convex-min}
 Let $m \in \mathbb{Z}^+$, $\gamma \in \Delta_m$, and $W_0\in \mathcal{B}^{\gamma, *}$ with $W_0= (p_{ij})_{i,j \in [m]}$. We say that the graphon $f= (f_{ij})_{i,j \in [m]}$ satisfies the \emph{$\ve$-neighborhood minorant condition} with respect to $W_0$ if for all $(i,j)$ such that $p_{ij} \in (0,1)$,  $\psi_{p_{ij}} = \hat{\psi}_{p_{ij}}$ in an $\ve$-neighborhood around $\Vert f_{ij} \Vert_d^d$, meaning
\[x \in \left(\Vert f_{ij} \Vert_d^d - \epsilon, \Vert f_{ij} \Vert_d^d + \epsilon \right) \cap[0,1] \implies  \psi_{p_{ij}}(x) = \hat{\psi}_{p_{ij}}(x).\] 
\end{definition}

\begin{remark}
To avoid ambiguity, we define the $\ve$-neighborhood minorant condition for $W_0$ with respect to the minimal block structure describing $W_0$. For this reason we require $W_0 \in {B}^{\gamma,*}$ in \Cref{j-sym,def:convex-min} and \Cref{sym-regime,sym-regime-near-one}.
\end{remark}

\subsection{A key lemma for establishing the symmetric regime}
We will establish the symmetric regimes for $\delta$ small and $\delta$ large (\Cref{sym-regime,sym-regime-near-one} respectively) by applying the following lemma.

\begin{lemma} \label{sym-condition}
Let $m \in \mathbb{Z}^+$, $\gamma \in \Delta_m$, and $W_0 \in \mathcal{B}^{\gamma, *}$, $\ve>0$, and $H$ be a $d$-regular graph. Suppose $\ti{f}$ is a minimizer of the variational problem \eqref{min_eq} for $\tau= t(H, \cdot)$ and there exists a sequence of graphons $f_n \in \W_\Omega$ such that each $f_n$ satisfies the $\ve$-neighborhood minorant condition, $\delta_\square(f_n, \ti{f}) \to 0$, and $I_{W_0}(f_n) \to J_{W_0}(\ti{f}).$ Then $\ti{f} \in \ti{\mathcal{B}}^\gamma$ and $J_{W_0}(\ti{f})= I_{W_0}(g)$ for some $g \in \mathcal{B}^\gamma$ with $\ti{g}= \ti{f}$.
\end{lemma}

\noindent
The proof relies on the following two lemmas. We delay their proofs until the end of the subsection.

\begin{lemma}\label{lemma:new- holder-general}
Let $H$ be a $d$-regular graph. Let $f = (f_{ij})_{i,j \in [m]}$ and $f^*=( \Vert f_{ij} \Vert_d)_{i,j \in [m]}.$
Then
\begin{align*}
    t(H,f) \leq t(H,f^*).
\end{align*}
\end{lemma}

\begin{lemma}\label{averaged-is-better}
Let $H$ be a $d$-regular graph, $m \in \mathbb{Z}^+$, $\gamma \in \Delta_m$, and $W_0 \in \mathcal{B}^{\gamma, *}$.
Let $f= (f_{ij})_{i,j \in [m]} \in \mathcal{W}_{\Omega}$.  
Assume that $f \in \mathcal{W}_{\Omega}$ satisfies the $\ve$-neighborhood minorant condition with respect to $W_0$ for some $\varepsilon>0$. If $d_\square(f, f^*) \geq \alpha>0$, there exists $\eta=\eta(\ve, d, W_0, \alpha)>0$, such that $$I_{W_0}(f)\geq I_{W_0}(f^*) +\eta.$$
Consequently, $I_{W_0}(f) = I_{W_0}(f^*)$ if and only if $f \in \mathcal{B}^\gamma$.
\end{lemma}

\begin{proof}[Proof of \Cref{sym-condition}.]
Let $\ti{f}$ be a minimizer of the variational problem \eqref{min_eq}. Let $f_n \in \W_\Omega$ be such that each $f_n$ satisfies the $\ve$-neighborhood minorant condition, $\delta_\square(f_n, \ti{f}) \to 0$, and $I_{W_0}(f_n) \to J_{W_0}(\ti{f}).$ For each $f_n$, define 
$f_n^*$ 
to be the corresponding $d$-averaged graphon. We claim that $d_\square (f_n^*, f_n) \to 0$.

Indeed, suppose for the sake of contradiction that there exists $\beta>0$ and a subsequence $\{f_{n_i}\}_{i \geq 1}$ such that $d_\square (f_{n_i}^*, f_{n_i}) \geq \beta$ for all $i$. By \Cref{averaged-is-better}, there exists some $\eta > 0$ such that
$I_{W_0}(f_{n_i}) \geq I_{W_0}(f_{n_i}^*) +\eta$. Consider the sequence $\{\ti{f}_{n_i}^*\}_{i\geq 1}$. By the compactness of $\tW_\Omega$, there exists a convergent subsequence $\ti{f}_{n_{i_k}}^* \to \ti{f'}$ for some $f' \in \W_\Omega$.
By \Cref{lemma:new- holder-general}, $t(H,f_{n_{i_k}}^*) \geq t(H,f_{n_{i_k}})$ and so $t(H, f' ) \geq t(H,f) \geq t$. It follows that $J_{W_0}(\ti{f'}) \geq \min\{J_{W_0}(\ti{g}) : \tau(g) \geq t\}$.
Observe
\begin{align*}
    J_{W_0}(\ti{f}) &= \liminf_{k \to \infty} {I_{W_0}(f_{n_{i_k}})} \geq \liminf_{k\to \infty}{I_{W_0}(f^*_{n_{i_k}})} +\eta\geq J_{W_0}(\ti{f'})+\eta\\ &\geq \min\{J_{W_0}(\ti{g}) : \tau(g) \geq t\} +\eta,
\end{align*}
and thus we have reached a contradiction.

We have shown that $d_\square (f_n^*, f_n) \to 0$. It follows that $\delta_\square(f_n^*, \ti{f} ) \to 0$. Since each $f_n^* \in \mathcal{B}^\gamma$, we can write $f_n^*= (\alpha_{ij}^n)_{i,j \in [m]}$ where each $\alpha_{ij}^n \in [0,1]$. By the compactness of $[0,1]^{m^2}$, there exists a subsequence such that  $$\alpha_{ij}^{n_k} \to \beta_{ij} \text{ for all $i,j \in [m]$}.$$
Let $g=(\beta_{ij})_{i,j \in [m]}$, $ g \in \mathcal{B}^\gamma$. Since $f_{n_k}^* \to g$ pointwise and $d_\square(f_{n_k}^*, g) \leq \Vert f_{n_k}^*- g \Vert_1$, the Dominated Convergence Theorem implies that $d_\square(f_{n_k}^*, g) \to 0$. Since $\delta_\square(f_{n_k}^*, \ti{f}) \to 0$, we have $\delta_\square(\ti{g}, \ti{f})=0$. Thus $\ti{f} \in \ti{\mathcal{B}}^\gamma$. Note $\tau(g) \geq t$.

Next we show that $I_{W_0}(g)= J_{W_0}(\ti{f}).$ Since $\delta_\square(f_n^*, \ti{f}) \to 0$,
$\liminf_{n \to \infty} I_{W_0}(f_{n}^*)\geq J_{W_0}(\ti{f}).$ Further, since each $f_{n}$ satisfies the $\varepsilon$-neighborhood minorant condition, then by \Cref{averaged-is-better},
$\limsup_{n \to \infty} I_{W_0}(f_{n}^*) \leq \lim_{n \to \infty} I_{W_0}(f_{n})=J_{W_0}(\ti{f})$.
Thus, $\lim_{n \to \infty} I_{W_0}(f_{n}^*)=J_{W_0}(\ti{f})$.
Since $f_{n_k}^* \to g$ pointwise and $f_{n_k}, g \in \W_\Omega$, the continuity of $h_p$ for any fixed $p \in (0,1)$ implies that  $$I_{W_0}(g)= \lim_{k \to \infty} I_{W_0}(f_{n_k}^*)=J_{W_0}(\ti{f}).$$
\end{proof}

\subsubsection{Proofs of supporting lemmas}
We turn to the proofs of  \Cref{lemma:new- holder-general,averaged-is-better}.

\begin{proof}[Proof of \Cref{lemma:new- holder-general}]
The lemma is a direct consequence of  the generalized H\"older inequality of  \cite{Finner1992}
(stated as Theorem \ref{theorem:holder} in the Appendix).
Let $v = |V(H)|$ be the number of vertices in $H$.  Recall the definition of $t(H,f)$:
\begin{align*}
t(H,f) &= \int_{[0,1]^{v}} \prod_{(i,j) \in E(H)} f(x_i, x_j) dx_1 \cdots dx_{v}.
\end{align*}
We break up the integration over the blocks specified by the vector $\gamma$. Recall the definition of $I_j$ (\Cref{block-graphon}). We have
\begin{align*}
t(H,f)&= \sum_{i_1 = 1}^m  \dots \sum_{i_{v} = 1}^m  \int_{x_1 \in I_{i_1}} \dots \int_{x_{v} \in I_{i_v}} \prod_{(a,b) \in E(H)}f(x_a, x_b) dx_1 \cdots dx_{v}.
\end{align*}
Recall that $f(x,y) = f_{\vartheta(x),\vartheta(y)}(r(x),r(y))$ as stated in \eqref{r-defn} and \eqref{r-defn-2}. It follows that
\begin{align*}
f(x_a, x_b) &= f_{\vartheta(x_a),\vartheta(x_b)}(r(x_a), r(x_b))= f_{i_a,i_b}(r(x_a), r(x_b)).
\end{align*}
Substituting and applying a change of variables, we obtain
\begin{align*}
t(H,f)&= \sum_{i_1 = 1}^m  \dots \sum_{i_{v} = 1}^m  \int_{x_1 \in I_{i_1}} \dots \int_{x_{v} \in I_{i_v}} \prod_{(a,b) \in E(H)}f_{i_a,i_b}(r(x_a), r(x_b)) dx_1 \cdots dx_{v}\\
&= \sum_{i_1 = 1}^m  \dots \sum_{i_{v} = 1}^m \left(\prod_{j = 1}^{v} \gamma_{i_j}\right) \int_{x \in [0,1]^{v}} \prod_{(a,b) \in E(H)}f_{i_a,i_b}(x_a, x_b) dx_1 \cdots dx_{v}.
\end{align*}
By the generalized H\"older inequality (Theorem \ref{theorem:holder}),
\begin{align*}
t(H,f)&\leq \sum_{i_1 = 1}^m  \dots \sum_{i_{v} = 1}^m \left(\prod_{j = 1}^{v} \gamma_{i_j}\right) \prod_{(a,b) \in E(H)} \left \Vert f_{i_a,i_b} \right \Vert_d = t(H, f^*).
\end{align*}
\end{proof}

\noindent
We will apply \Cref{a-set-lemma,1-norm-on-blocks,linear-approx} in the proof of \Cref{averaged-is-better}.

\begin{lemma}\label{1-norm-on-blocks}
Fix $m \in \mathbb{Z}^+$, $\gamma \in \Delta_m$, and let
$f,g \in \mathcal{B}^\gamma$ be  two graphons with
$f = (f_{ij})_{(i,j) \in [m]^2}$ and $g = (g_{ij})_{(i,j) \in [m]^2}$.
Then $d_\square(f,g)\leq\max_{i,j}d_\square(f_{ij},g_{ij})$.
\end{lemma}
\begin{proof}
Recalling the definition of $d_\square$, we will need to bound
$\sup_{S,T}|\int_{S\times T}(f-g)|$ over all measurable subsets $S,T\subset [0,1]$.  Fix two such subsets, and recall that $I_i$ is the $i$th block. By the triangle inequality,
\begin{align*}
\left| \int_{S \times T} \left(f(x,y) - g(x,y) \right) dx dy \right| &= \left| \sum_{(i,j) \in [m]^2} \int_{\{S \times T\} \cap \{I_i \times I_j\}} \left(f(x,y) - g(x,y) \right) dx dy \right|\\
&\leq \sum_{(i,j) \in [m]^2}\left|  \int_{(S \times T) \cap (I_i \times I_j)} \left(f(x,y) - g(x,y) \right) dx dy \right|.
\end{align*}
Setting $T_i=r(T\cap I_i)$ and $S_i=r(T\cap I_i)$, where $r$ is as described in \eqref{r-defn} and \eqref{r-defn-2}, we write the right hand side as
$$
 \sum_{(i,j) \in [m]^2}\left| \gamma_i\gamma_j \int_{(S_i \times T_i) } \left(f_{ij}(x,y) - g_{ij}(x,y) \right) dx dy\right|
 \leq\sum_{i,j}\gamma_i\gamma_j\max_{k,\ell}d_\square(f_{k\ell},g_{k\ell})=\max_{k,\ell}d_\square(f_{k\ell},g_{k\ell}).
$$
Since $S,T\subset [0,1]$ were arbitrary, this completes the proof.
\end{proof}

\begin{definition}\label{a-set}
Given $f:[0,1]^2\to [0,1]$ measurable, $c \in \R$, and $\ve \in [0,1]$, define the sets $$A^+_\ve (f, c)=\{ (x,y) \in [0,1]^2: f(x,y)-c \geq \ve\}$$
$$A^-_\ve (f, c)=\{ (x,y)\in [0,1]^2: c- f(x,y) \geq \ve\}.$$
\end{definition}

\begin{lemma}\label{a-set-lemma}
Let $f:[0,1]^2 \to [0,1]$  be measurable, and let $g$ be the constant graphon that takes value $\Vert f \Vert_d$. There exists $\beta=\beta(\ve)$ such that if $d_\square(f , g) \geq \ve$, then
$$|A^+_\beta (f^d, \Vert f\Vert_d^d)| \geq \beta \quad \text{ and } \quad |A^-_\beta (f^d, \Vert f\Vert_d^d)| \geq \beta $$
where $f^d$ denotes the function $f^d(x,y)= f(x,y)^d$.
\end{lemma}

\begin{proof}
Without loss of generality, we may assume that $\ve$ is small enough such that $1\geq \nicefrac{\ve}{4} \geq d \left(\nicefrac{\ve}{8}\right)^d.$ We begin by observing that for $u,v \geq 0$ satisfying $u- v\geq \ve/4$ and $a(x) =x^d$, \begin{align}\label{easy-calc}
    a(u)-a(v) = \int_{v}^u a'(x) dx \geq \int_{v+\ve/8}^{v+\ve/4} a'(x) dx \geq \frac{\ve}{8} a'\bfrac{\ve}{8} = d\bfrac{\ve}{8}^d.
\end{align}

 Let $f:[0,1]^2 \to [0,1]$ satisfy the hypotheses of the lemma. \Cref{cut-sets} implies that $|A^+_{\ve/4}(f, \Vert f \Vert_d) | \geq \ve/4$  or $|A^-_{\ve/4}(f, \Vert f \Vert_d) | \geq \ve/4$. We will establish the result in the case that $|A^+_{\ve/4}(f, \Vert f \Vert_d) | \geq \ve/4$. The other case follows by an analogous argument.
 Note that if $f(x,y) - \Vert f\Vert_d \geq \ve/4$, then \eqref{easy-calc} implies $$f^d (x,y) -\Vert f\Vert_d^d \geq  d\bfrac{\ve}{8}^d.$$
 Let $c= d(\ve/8)^d$.
Since $|A^+_{\ve/4}(f, \Vert f \Vert_d) | \geq \ve/4$ and $A^+_{\ve/4}(f, \Vert f \Vert_d) \subseteq A^+_{c}(f^d, \Vert f \Vert_d^d)$,  it follows that $|A^+_{c}(f^d, \Vert f \Vert_d^d) | \geq \ve/4\geq c$. For ease of notation, let  $A^+=A^+_{0}(f^d, \Vert f \Vert_d^d)$ and $A^-=A^-_{0}(f^d, \Vert f \Vert_d^d)$. Observe
$$ c^2 \leq \int_{A^+_{c}(f^d, \Vert f \Vert_d^d)} f^d- \Vert f \Vert_d^d \leq \int_{A^+} f^d- \Vert f \Vert_d^d= \int_{A^-}  \Vert f \Vert_d^d-f^d \leq |A^-_{c^2/2}(f^d, \Vert f \Vert_d^d)| + \frac{c^2}{2}.$$
It follows that $|A^-_{c^2/2}(f^d, \Vert f \Vert_d^d)|  \geq c^2/2$.  Since $|A^+_{c^2/2}(f^d, \Vert f \Vert_d^d)| \geq |A^+_{c}(f, \Vert f \Vert_d) | \geq c \geq c^2/2$, taking $\beta=c^2/2$ completes the proof.
\end{proof}

\begin{lemma}\label{linear-approx}
Let $p \in (0,1)$, $\varepsilon>0$, and $f:[0,1]^2 \to [0,1]$ measurable.
There exists $\eta=\eta(p, \ve, \beta)>0$ such that if $|A^-_\beta(f, \Vert f \Vert_1)|,|A^+_\beta(f, \Vert f \Vert_1)| \geq \beta$ and $\psi_p(x) = \hat{\psi}_p(x)$ for all $x \in \left(\Vert f \Vert_1 -\ve, \Vert f \Vert_1+\ve\right)$,
then $$\int_{[0,1]^2} \hat{\psi}_p(f)  \geq \hat{\psi}_p \brac{\Vert f \Vert_1} +\eta.$$
\end{lemma}

\begin{proof}
 For ease of notation let $z= \Vert f\Vert_1$. Since $\psi_p$ is differentiable, it follows that $\hat{\psi}_p$ is differentiable at $z$, and so $\hat{\psi}_p'(z)=\psi_p'(z)$ is a subdifferential of $\hat{\psi}_p$ at $z$. Let $$g(w)= \psi_p(z) + \psi_p'(z)(w-z).$$
Moreover, since $\psi_p'(z)$ is a subdifferential of $\hat{\psi}_p$ at $z$, $\hat{\psi}_p(w) \geq g(w)$ for all $w \in [0,1]$.
Since $g$ is a linear function,
\begin{align}\label{g-good}\int_{[0,1]^2} g(f(x,y))= g\brac{\int_{[0,1]^2} f(x,y)}=g\brac{\Vert f \Vert_1}= \psi_p\brac{\Vert f \Vert_1}= \hat{\psi}_p\brac{\Vert f \Vert_1}.\end{align}

\noindent
Define $$d(w)= \psi_p(w)-g(w),$$
and note that $d'(w)=\psi_p'(w)-\psi_p'(z)$. Assuming $w\geq z + \beta/2$ and applying the Fundamental Theorem of Calculus twice, we obtain
\begin{align*}
    d(w)&= d(w)-d(z)= \int_{z}^w d'(a)\, da
    = \int_z^w \int_z^a \psi_p''(b) \,db\, da \\
    &= \int_z^{z + \beta/2} \int_z^a \psi_p''(b) \,db\, da +  \int_{z+\beta/2}^w \int_z^a \psi_p''(b) \,db\, da,\\
    &\geq \frac{(w-z-\beta/2)^2}{2}  \min \{ \psi_p''(x): x \in  [z+\beta/2,w]\},
\end{align*}
provided that  $\psi_p''(x)$ is  non-negative on $[z, w]$. The same argument applies for $w\leq z$, and one obtains a similar lower bound, with $\min\{\psi_p''(x): x \in [w,z-\beta/2]\}$ instead.

\noindent
Next, we construct a set $S \subseteq[0,1]^2$ and choose $\eta'>0$ such that for all $(x,y) \in S$,
\begin{itemize}
\item[(i)] $\psi_p(f(x,y))= \hat{\psi}_p(f(x,y))$,
\item[(ii)]  $|f(x,y)- z | \geq \beta$,
\item[(iii)] $\psi_p''(b) \geq \eta'$ for all $b \in [f(x,y),z-\beta/2]$ if $f(x,y) \leq z$, or for all $b \in [z+\beta/2,f(x,y)]$ if $f(x,y) \geq z$,
\item[(iv)] $\psi_p''$ is non-negative on  $[f(x,y),z]$ if $f(x,y) \leq z$ or non-negative on $[z,f(x,y)]$ if $f(x,y) \geq z$, and
\item[(v)] $|S| \geq \beta$.
\end{itemize}
Our construction of $S$ depends on $\psi_p$ and $z$. Let $p_0$ be as given in \Cref{all-about-psi}. There are three cases concerning $p_0$. In each case, $\eta'$ is well-defined because it is the minimum of a continuous function over a compact set.
\begin{enumerate}[(1)]
    \item If $p > p_0$, then $\psi_p''$ is positive  on $[0,1]$. Let $S= A^+_\beta(f, z)$ and $\eta'= \min\{ \psi_p''(x): x \in [0,1]\}$.
    \item If $p < p_0$, then the function $\psi_p$ has two inflection points $r_1$ and $r_2$, and  $\psi_p''(x) >0$ on $[0,r_1)$ and $(r_2, 1]$ and $\psi_p''(x) <0$ on $(r_1,r_2)$. Note that $z \not \in (r_1 - \epsilon, r_2 + \epsilon)$, since $\psi_p=\hat{\psi}_p$ is convex in an $\ve$-neighborhood around $z$.
    \begin{itemize}[-]
   \item If $z\leq r_1 - \ve$, let $\eta'= \min\{ \psi_p''(x): x \in [0,r_1-\ve]\}$. Note $\eta' >0$ since $\psi_p''$ is positive on $[0,r_1-\ve]$. Let $S= A^-_\beta(f, z)$.
   \item If $z \geq r_2 +\ve$, let $\eta'= \min\{ \psi_p''(x): x \in [r_2+\ve,1]\}$. Note $\eta' >0$ since $\psi_p''$ is positive on $[r_2+\ve,1]$. Let $S= A^+_\beta(f, z)$.
    \end{itemize}
\item If $p=p_0$,  then the function $\psi_p$ has one point $r$ such that $\psi_p''(r)=0$ and $\psi_p''(x) >0$ on $[0,r)$ and $(r, 1]$. If $z < r$, let $\eta'= \min\{ \psi_p''(x): x \in [0, z - \beta/2]\}$ and $S= A^-_\beta(f, z)$. If $z \geq r$, let $\eta'= \min\{ \psi_p''(x): x \in [z + \beta/2, 1]\}$ and $S= A^+_\beta(f, z)$. Since $A^+_\beta(f,z), A^{-}_\beta(f,z)$ are non-empty, $0< z - \beta/2< z+\beta/2<1$.
\end{enumerate}

\noindent By our choice of $S = A_{\beta}^-(f,z)$ or $S = A_{\beta}^+(f,z)$ as needed, it is easy to see that properties (i), (ii), and (v) are satisfied. To see how the remaining properties are satisfied, consider for example Case (2), where $z \leq r_1 - \varepsilon$. In this case, $(x,y) \in S = A_{\beta}^-(f,z)$ implies $f(x,y) \leq z - \beta \leq r_1 - \varepsilon - \beta$. Since $\psi_p''(x) > 0$ on $[0,r_1)$, property (iv) is satisfied. Property (iii) is satisfied since $\psi_p''(b) \geq \eta'$ for all $b \in [0,r_1 - \varepsilon] \supset [f(x,y), z - \beta/2]$, where the inclusion holds for any $(x,y) \in S$.

Note that if $(x,y) \in S$ and $f=f(x,y) \geq z$, then properties (i)-(iv) imply that \begin{equation}\label{s-good}
 \hat{\psi}_p(f)- g(f)= \psi_p(f)- g(f)= d( f) \geq \frac{(f-z-\beta/2)^2}{2}  \min \{ \psi_p''(b): b \in [z+\beta/2,f]\} \geq \frac{\eta'\beta^2}{8}.
 \end{equation}
The same bound holds if $f(x,y) \leq z$.

\noindent
Recall that $\hat{\psi}_p(w)\geq g(w)$ for all $w \in [0,1]$. It follows by \eqref{g-good}, \eqref{s-good} and property (v) that
\begin{align*}
    \int_{[0,1]^2} \hat{\psi}_p(f(x,y)) \, dx \, dy - \hat{\psi}_p\brac{\Vert f \Vert_1}&= \int_{[0,1]^2} \hat{\psi}_p(f(x,y))- g(f(x,y)) \, dx \, dy \\
    &\geq \int_S \hat{\psi}_p(f(x,y))- g(f(x,y))\, dx \, dy
    \geq |S| \frac{\eta'\beta^2}{8} \geq \frac{\eta' \beta^3}{8}.
\end{align*}
Taking $\eta=\eta' \beta^3 / 8 $ completes the proof.
\end{proof}

\begin{proof}[Proof of \Cref{averaged-is-better}]
Suppose $f$ is a graphon that satisfies the $\ve$-neighborhood minorant condition,  and has the property that $d_\square(f, f^* ) \geq \alpha$.
Since $d_\square(f, f^* ) \geq \alpha$,   \Cref{1-norm-on-blocks} implies that there exists
some $a,b \in [m]^2$ such that $d_\square(f_{ab}, \Vert f_{ab}\Vert_d) \geq \alpha$
(where by an abuse of notation, $\Vert f_{ab}\Vert_d$ denotes the constant function that takes that value $\Vert f_{ab}\Vert_d$). Since $f \in \W_\Omega$, $f_{ij}$ is constant whenever $p_{ij}\in \{0,1\}$ and the values $a,b$ are such that $ p_{ab} \in (0,1)$.
 \Cref{a-set-lemma}
implies that there exists $\beta= \beta(\alpha)$
 such that
$|A^+_\beta (f_{ab}^d, \Vert f_{ab}\Vert_d^d)| \geq \beta$ and $|A^-_\beta (f_{ab}^d, \Vert f_{ab}\Vert_d^d)| \geq \beta$.

For each $(i,j)$ such that $p_{ij} \in (0,1)$, \Cref{linear-approx} implies that there exists $\eta_{ij}=\eta_{ij}(p_{ij}, \ve, \beta)$ such that if  $g : [0,1]^2 \to [0,1]$, $|A^-_\beta(g, \Vert g \Vert_1)|,|A^+_\beta(g, \Vert g \Vert_1)| \geq \beta$  and $\psi_{p_{ij}}(x) = \hat{\psi}_{p_{ij}}(x)$ for all $x \in \left(\Vert g \Vert_1 -\ve, \Vert g \Vert_1+\ve\right)$,
then \begin{align}\label{gap}\int_{[0,1]^2} \hat{\psi}_{p_{ij}}(g)  \geq \hat{\psi}_{p_{ij}} \brac{\Vert g \Vert_1} +\eta_{ij}.\end{align}
Let $\eta'= \min_{(i,j)} \{ \eta_{ij} : p_{ij} \in (0,1)\} $.

Since $f$ satisfies the $\ve$-neighborhood minorant condition, $\psi_{p_{ab}}= \hat{\psi}_{p_{ab}}$ in an $\ve$-neighborhood around the value $\Vert f_{ab}\Vert_d^d=\Vert f_{ab}^d \Vert_1$.
Applying \eqref{gap} to $f_{ab}^d$,  we obtain
\begin{align*}
 \int_{[0,1]^2} &h_{p_{ab}}(f_{ab}(x,y)) dx dy
= \int_{[0,1]^2} \psi_{p_{ab}}\left(f^d_{ab}(x,y)\right) dx dy
\geq \int_{[0,1]^2} \hat{\psi}_{p_{ab}}\left(f^d_{ab}(x,y)\right) dx dy\\
&\geq  \hat{\psi}_{p_{ab}}\left(\left \Vert f_{ab} \right \Vert_d^d\right) + \eta'
=  \psi_{p_{ab}}\left(\left \Vert f_{ab} \right \Vert_d^d\right) +\eta'
=  h_{p_{ab}}\left(\left \Vert f_{ab} \right \Vert_d\right)+ \eta'.
\end{align*}

For all $i,j$ such that $p_{ij}\in (0,1)$, the point $\left( \Vert f_{ij} \Vert_d^d, h_p(\Vert f_{ij} \Vert_d)\right)$ lies on the convex minorant of $\psi_{p_{ij}}(x)= h_{p_{ij}}(x^{1/d})$, and so applying Jensen's Inequality we obtain
\begin{align*}
 \int_{[0,1]^2} h_{p_{ij}}(f_{ij}(x,y))& dx dy
= \int_{[0,1]^2} \psi_{p_{ij}}\left(f^d_{ij}(x,y)\right) dx dy
\geq \int_{[0,1]^2} \hat{\psi}_{p_{ij}}\left(f^d_{ij}(x,y)\right) dx dy\\
&\geq  \hat{\psi}_{p_{ij}}\left(\int_{[0,1]^2} f^d_{ij}(x,y) dx dy\right)
=  \hat{\psi}_{p_{ij}}\left(\left \Vert f_{ij} \right \Vert_d^d\right)
=  \psi_{p_{ij}}\left(\left \Vert f_{ij} \right \Vert_d^d\right)
=  h_{p_{ij}}\left(\left \Vert f_{ij} \right \Vert_d\right).
\end{align*}

\noindent
 Since $f \in \mathcal{W}_{\Omega}$, it holds that $f(x,y) = W_0(x,y)$ for all $(x,y)$ such that $W_0(x,y) \in \{0,1\}$. Let $c_{ij}$ be the indicator that $p_{ij} \in (0,1)$.
Observe
\begin{align*}
I_{W_0}(f) &= \sum_{i=1}^m \sum_{j=1}^m \gamma_i \gamma_j c_{ij} \int_{[0,1]^2} h_{p_{ij}}(f_{ij}(x,y)) dx dy\\
&\geq  \eta' \gamma_a \gamma_b +\sum_{i=1}^m \sum_{j=1}^m \gamma_i \gamma_j c_{ij} h_{p_{ij}}\left(\left \Vert f_{ij} \right \Vert_d\right)\\
&=\eta' \gamma_a \gamma_b + I_{W_0}\left(f^* \right).
\end{align*}
Taking $\eta= \eta'\min_{(i,j) } \{\gamma_i \gamma_j : p_{ij} \in (0,1) \}$ yields the desired result.
\end{proof}

\subsection{Symmetry for small $\delta$}
We now prove \Cref{sym-regime}, which establishes the existence of a symmetric regime when $\delta$ is small. We state the key lemmas used in the proof and defer the proofs of these lemmas to the end of the subsection.

\begin{lemma}\label{CMC-seq}
Given $m \in \mathbb{Z}^+$, $\gamma \in \Delta_m$, $W_0 \in \mathcal{B}^{\gamma, *}$ and a finite $d$-regular graph $H$, there exist
 $\delta_{0}=\delta_{0}(H, W_0)$ and $\ve=\ve(W_0)>0$
 such that the following is true.
For all $0 < \delta < \delta_0$, if $\ti{f}$ is a minimizer of the variational problem \eqref{min_eq} with $t=(1+\delta)t(H,W_0)$, then there exists a sequence of graphons $f_n \in \W_\Omega$ such that  $\delta_\square(f_n, \ti{f}) \to 0$, $I_{W_0}(f_n) \to J_{W_0}(\ti{f})$ and each $f_n$ satisfies the $\ve$-neighborhood minorant condition.
\end{lemma}


\noindent
We now use \Cref{CMC-seq} to prove \Cref{sym-regime}.

\begin{proof}[Proof of \Cref{sym-regime}.]
Let $W_0 \in \mathcal{B}^{\gamma,*}$, and
let $H$ be a finite $d$-regular graph.  Let $\delta_0>0$ and $\ve>0$ be as in
\Cref{CMC-seq}, and
assume that $\delta < \delta_0$.  Suppose that $\ti{f}$ is a minimizer of the variational problem \eqref{min_eq} with $t=(1+\delta)t(H,W_0)$. By \Cref{CMC-seq}, there exists a sequence of graphons $f_n \in \W_\Omega$ such that each $f_n$ satisfies the $\ve$-neighborhood minorant condition, $\delta_\square(f_n, \ti{f}) \to 0$, and $I_{W_0}(f_n) \to J_{W_0}(\ti{f}).$
\Cref{sym-condition} implies that
$\ti{f} \in \ti{\mathcal{B}}^\gamma$ and $J_{W_0}(\ti{f})= I_{W_0}(g)$ for some $g \in \mathcal{B}^\gamma$ with $\ti{g}= \ti{f}$. Thus, the problem is in the symmetric regime.


Let $\ti{h}$ be a minimizer of the variational problem \eqref{min_eq} with $t=(1+\delta)t(H,W_0)$. By the above argument we may assume that $h \in \mathcal{B}^\gamma$ and $I_{W_0}(h) = J_{W_0}(\ti{h})$. Note that
 $$\min\{I_{W_0}(g): g \in \mathcal{B}^\gamma, t(H, g) \geq (1+\delta)t(H, W_0)\} \leq I_{W_0}(h)=J_{W_0}(\ti{h})$$ and
\begin{align*} \min\{I_{W_0}(g): g \in \mathcal{B}^\gamma, t(H, g) \geq (1+\delta)t(H, W_0)\} &\geq \min\{J_{W_0}(\ti{g}): \ti{g }\in \ti{\mathcal{B}}^\gamma, t(H, g)\geq(1+\delta)t(H, W_0)\}\\
&= J_{W_0}(\ti{h})=I_{W_0}(h).
\end{align*}
It follows that
$$J_{W_0}(\ti{h})=I_{W_0}(h)=\min\{I_{W_0}(g): g \in \mathcal{B}^\gamma, t(H, g) \geq (1+\delta)t(H, W_0)\}.$$
\end{proof}

\subsubsection{Proofs of supporting lemmas}
We now prove \Cref{CMC-seq}. We first show that graphons with homomorphism density close to $t(H,W_0)$ must be close to $W_0$ pointwise except possibly on a small set.

\begin{lemma}\label{new}
Let  $m \in \mathbb{Z}^+$, $\gamma \in \Delta_m$, $W_0 \in \mathcal{B}^\gamma$ such that $\gamma \in \R^m$, and  let $H$ be a finite $d$-regular graph. Suppose $f \geq W_0$ pointwise and $t(H,f) \leq (1+\delta) t(H,W_0)$ for $\delta > 0$. Let $Y$ be a relevant interval labeling vector with respect to $W_0$. Suppose $I_a \times I_b$ contributes to $Y$, i.e. $Y_i=a$ and $Y_j=b$ for some $\{i,j \} \in E(H)$. (See \Cref{relevant}.) Let $S_\ve=\{ (x,y)  \in I_a \times I_b : f(x,y) \geq W_0(x,y) +\ve\}$.
Then
$$\ve |S_\ve| \leq \frac{\delta t(H, W_0)p_{ab} \gamma_a \gamma_b}{t(H,W_0, Y)},$$
where $p_{ab}$ is the value of $W_0$ on $I_a \times I_b$, $\gamma_a=|I_a|$, and $\gamma_b= |I_b|$.
\end{lemma}

\begin{proof}
Let $v= |V(H)|$ be the number of vertices in $H$. Observe
\begin{align*}
t(H,f,Y)&= \int_{x_1 \in I_{Y_1}} \dots \int_{x_v \in I_{Y_v}} \prod_{\{u,w\} \in E(H)} f(x_u, x_w)  \, dx_v \dots dx_1\\
&\geq\brac{\prod_{u \in [v]\setminus\{i,j\}}|I_{Y_u}| \prod_{\{u,w\}\in E(H)\setminus\{i,j\}} p_{Y_uY_w}} \int_{x_i \in I_a} \int_{x_j \in I_b} f(x_i,x_j) dx_i \, dx_j\\
&\geq\brac{\prod_{u \in [v]\setminus\{i,j\}}|I_{Y_u}| \prod_{\{u,w\}\in E(H)\setminus\{i,j\}} p_{Y_uY_w}} \brac{\int_{(x_i,x_j) \in I_a \times I_b \setminus S_\ve} p_{ab}+\int_{(x_i,x_j) \in S_\ve} (\ve + p_{ab})  }\\
&= \brac{\prod_{u \in [v]\setminus\{i,j\}}|I_{Y_u}| \prod_{\{u,w\}\in E(H)\setminus\{i,j\}} p_{Y_uY_w}} \brac{  |I_a| |I_b| p_{ab}+\ve |S_\ve| }\\
&= t(H,W_0,Y) +  \frac{\ve |S_\ve|}{p_{ab} \gamma_a \gamma_b}t(H,W_0,Y) .
\end{align*}
We use this to lower bound the homomorphism density of $f$, and obtain
\begin{align*}(1+\delta) t(H,W_0) \geq t(H,f) &= \sum_{Z \in [m]^v} t(H,f,Z) \\
&\geq \frac{\ve |S_\ve|}{p_{ab} \gamma_a \gamma_b}t(H,W_0,Y)+ \sum_{Z \in [m]^v} t(H, W_0, Z)\\
&= \frac{\ve |S_\ve|}{p_{ab} \gamma_a \gamma_b}t(H,W_0,Y)+ t(H,W_0).
\end{align*}
The statement of the lemma follows directly.
\end{proof}

\begin{lemma}\label{delta-small-star-holds}
 Fix $m \in \mathbb{Z}^+$, $\gamma \in \Delta_m$, $W_0 \in \mathcal{B}^{\gamma}$ and $H$ a finite $d$-regular graph.
There exist $\delta_0=\delta_0(H, W_0)$ and $\ve=\ve(W_0)>0$ such that the following is true.
If $f \in \W_\Omega$, $f\geq W_0$ pointwise, $f=W_0$ on irrelevant blocks (\Cref{relevant}), and $t(H, f) \leq (1+\delta_0) t(H, W_0)$, then $f$ satisfies the $\ve$-neighborhood minorant condition.
\end{lemma}

\begin{proof}
Let $W_0=(p_{ij})_{i,j \in [m]}$.
By \Cref{all-about-psi}, for $p_{ij} \in (0,1)$ the function $\psi_{p_{ij}}:(0,1)\to \R$ is either convex 
or $\hat{\psi}_{p_{ij}}$ is constructed by replacing $\psi_{p_{ij}}$ with its lower common tangent on exactly one interval. 
If $\psi_{p_{ij}}$ is convex, then the  minorant condition is trivially satisfied for the $(i,j)$ block. Otherwise, let $q_{ij}$ be
such that $q_{ij}^d$ is the left point of intersection between $\psi_{p_{ij}}$ and its lower common tangent. Let $P=\{(i,j): \psi_{p_{ij}}''(x)=0 \text{ for some } x \in [p_{ij}^d,1] \}$ and define 
\begin{align*}
\ve &= \min_{(i,j)\in P }q_{ij}^d- p_{ij}^d.
\end{align*}
Since $\psi'_{p_{ij}}(p_{ij}^d) = 0$ for all $(i,j)$, we have $q_{ij} > p_{ij}$ for all $(i,j) \in P$, so that $\ve > 0$.

Let $f \in \W_\Omega$, $f\geq W_0$ pointwise, $f=W_0$ on irrelevant blocks (\Cref{relevant}), and $t(H, f) \leq (1+\delta) t(H, W_0)$. Let $(a,b) \in P$. When $I_a \times I_b$ is irrelevant, $\Vert f_{ab} \Vert_d^d=p_{ab}^d \leq q_{ab}^d - \ve$. Since $\psi_{p_{ab}} = \hat{\psi}_{p_{ab}}$ on $[0, q_{ab}^d]$, $\psi_{p_{ab}} = \hat{\psi}_{p_{ab}}$ in an $\ve$-neighborhood around $\Vert f_{ab}\Vert_d^d$. It suffices to show that for $\delta$ sufficiently small, $\psi_{p_{ab}}=\hat{\psi}_{p_{ab}}$ in a neighborhood around $\Vert f_{ab} \Vert_d^d$ for all $(a,b) \in P$ such that $I_a \times I_b$ is relevant. 

Fix such a pair $(a,b)$. Let $S_\eta=\{ (x,y)  \in I_a \times I_b : f(x,y) \geq W_0(x,y) +\eta\}$. Let $S_\eta^c=(I_a \times I_b) \setminus S_\eta$. Observe
\begin{align*}
\Vert f_{ab}\Vert_d^d &= \frac{1}{\gamma_a \gamma_b}\left(\int_{S_\eta^c}f(x,y)^d dx \, dy+\int_{S_\eta} f(x,y)^d dx \, dy \right) \\
&\leq  \frac{1}{\gamma_a \gamma_b} \left(|S_{\eta}^c|(p_{ab}+ \eta)^d + |S_\eta| \right)\\
&\leq (p_{ab}+ \eta)^d + \frac{1}{\gamma_a \gamma_b}|S_\eta|.
\end{align*}
Let $\eta(\delta) = \sqrt{ \frac{\delta t(H, W_0)p_{ab} \gamma_a \gamma_b}{t(H,W_0, Y)}}$. By \Cref{new}, $$\eta(\delta)|S_{\eta(\delta)}| \leq\frac{\delta t(H, W_0)p_{ab} \gamma_a \gamma_b}{t(H,W_0, Y)} \implies |S_{\eta(\delta)}| \leq \eta(\delta).$$ It follows that
 $$ \Vert f_{ab} \Vert_d^d \leq \brac{ p_{ab} +\eta(\delta)}^d + \frac{\eta(\delta)}{\gamma_a \gamma_b}.$$

Note that as $\delta \to 0$, $\eta(\delta) \to 0$, and the above expression approaches $p_{ab}^d$. Thus we can pick $\delta$ sufficiently small so that $\Vert f_{ab} \Vert_d^d<q_{ab}^d-\ve$. Since $\psi_{p_{ab}} = \hat{\psi}_{p_{ab}}$ on $[0, q_{ab}^d]$, $\psi_{p_{ab}}=\hat{\psi}_{p_{ab}}$ in an $\ve$-neighborhood around $\Vert f_{ab} \Vert_d^d$.
Repeating the argument for each block concludes the proof.
\end{proof}

\noindent
We will combine \Cref{delta-small-star-holds} with the following lemma to prove
 \Cref{CMC-seq}.

\begin{lemma}\label{nice-seq}
Let $\ti{f}$ be a minimizer of the variational problem $\eqref{min_eq}$ with $t= (1+\delta) t(H,W_0)$ for $\delta > 0$.
Then there exists a sequence $\{f_n\}_{n \geq 1}$ such that $f_n \in \W_\Omega$, $f_n \geq W_0$ pointwise, $f_n=W_0$ on irrelevant blocks, $I_{W_0}(f_n) \to J_{W_0}(\ti{f})$ and $\delta_\square(  f_n, \ti{f}) \to 0$.
\end{lemma}

\begin{proof}
Let $\ti{f}$ be a minimizer of the variational problem $\eqref{min_eq}$ with $t= (1+\delta) t(H,W_0)$.
There exists a sequence $\{f_n\}_{n \geq 1}$ such that $f_n \in \W_\Omega$, $I_{W_0}(f_n) \to J_{W_0}(\ti{f})$ and $\delta_\square( \tilde{f}, f_n) \to 0$.
Let $R\subseteq[0,1]^2$ be the union of the relevant blocks (\Cref{relevant}). Define  $f_n' \in \W_\Omega$ such that
$$f_n'(x,y)=
\begin{cases}
\max\{f_n(x,y), W_0(x,y)\} & (x,y) \in R\\
 W_0(x,y) & (x,y) \not \in R
\end{cases}
$$
Note that this ensures that $f_n' \geq W_0$ pointwise, and note further that $f_n'=W_0$ whenever $f_n'\neq f_n$, which in turn implies that
$h_{W_0(x,y)}( f'_{n}(x,y)) \leq h_{W_0(x,y)}( f_{n}(x,y))$ for all $(x,y) \in [0,1]^2$ and hence
 $I_{W_0}(f_n') \leq I_{W_0}(f_n)$.

We claim that the lemma follows from showing that $d_\square(f'_n, f_n) \to 0$. Indeed, if $d_\square(f'_{n}, f_{n}) \to 0$, then $\delta_\square(f'_n, \ti{f}) \to 0$. Thus $\liminf I_{W_0}(f_n') \geq J_{W_0}(\ti{f})$. Since $I_{W_0}(f_n') \leq I_{W_0}(f_n)$ for each $n$, $\limsup I_{W_0}(f_n') \leq \limsup I_{W_0}(f_n)= J_{W_0}(\ti{f})$. Thus, $ I_{W_0}(f_n') \to J_{W_0}(\ti{f})$ as $n \to \infty$. 
Note that by construction $f'_{n} \geq W_0$ pointwise and $f'_{n}=W_0$ on irrelevant blocks. This completes the proof.

It remains to prove that $d_\square(f'_n, f_n) \to 0$. Suppose for contradiction that there exists $\ve>0$ and a subsequence such that $d_\square( f'_{n_i}, f_{n_i}) \geq \ve$ for all $i\geq 1$.
Let $$\eta= \min_{p \in Im(W_0) \setminus\{0,1\}} \min\{ h_p(p+\ve/4), h_p(p-\ve/4)\},$$
where we have chosen $\ve>0$ sufficiently small so that $p +\ve/4, p - \ve/4 \in (0,1)$ for all $p \in \text{Im}(W_0) \setminus\{0,1\}$.
We will show that $I_{W_0}(f_{n_i}') \leq I_{W_0}(f_{n_i}) -\ve\eta/4$, and then use this to derive a contradiction.

Indeed, let $S_{\ve/4}^+=\{(x,y) \in [0,1]^2: f'_{n_i}(x,y)-f_{n_i}(x,y) \geq \ve/4\}$ and
$S_{\ve/4}^-=\{(x,y) \in [0,1]^2: f_{n_i}(x,y)-f'_{n_i}(x,y) \geq \ve/4\}$. Let $S= S_{\ve/4}^+ \cup S_{\ve/4}^-$. \Cref{cut-sets} implies that $|S|\geq \ve/4$.

If $(x,y) \in S_{\ve/4}^+ \cap R$, then $\max\{ f_{n_i}(x,y), W_0(x,y) \} - f_{n_i}(x,y) \geq \ve/4$. It follows that $f_{n_i}'(x,y)= W_0(x,y)$. If $(x,y) \in S_{\ve/4}^+ \setminus R$, then $ W_0(x,y) - f_{n_i}(x,y) \geq \ve/4$. In both cases $f_{n_i}(x,y) \leq W_0(x,y) -\ve/4$, and so  $h_{W_0(x,y)}(f_{n_i}(x,y)) \geq \eta$. Therefore $$h_{W_0(x,y)}( f'_{n_i}(x,y))=0 \leq h_{W_0(x,y)}(f_{n_i}(x,y))-\eta.$$

If $(x,y) \in S_{\ve/4}^-$, then $(x,y) \not \in R$ because $f_n' \geq f_n$ on $R$. It follows that $f_{n_i}'(x,y) = W_0(x,y)$, and so $f_{n_i}(x,y) \geq W_0(x,y) +\ve/4$. Thus $h_{W_0(x,y)}( f_{n_i}(x,y)) \geq \eta$, which implies
$$h_{W_0(x,y)}( f'_{n_i}(x,y))=0 \leq h_{W_0(x,y)}(f_{n_i}(x,y))-\eta.$$

Recall that $h_{W_0(x,y)}( f'_{n_i}(x,y)) \leq h_{W_0(x,y)}( f_{n_i}(x,y))$ for all $(x,y) \in [0,1]^2$ and that $|S|\geq \ve/4$. Therefore
\begin{align*}
    I_{W_0}(f_{n_i}') \leq  \int_{[0,1]^2 \setminus S} h_{W_0(x,y)}( f_{n_i}(x,y)) _+ \int_{S} \left[h_{W_0(x,y)}( f_{n_i}(x,y))- \eta\right]
    \leq I_{W_0}(f_{n_i}) - \frac{\ve \eta}{4}.
\end{align*}

Next, consider the sequence $\{\ti{f}_{n_i}'\}_{i\geq 1}$. By the compactness of $\tW_\Omega$, there exists a convergent subsequence $\ti{f}'_{n_{i_k}} \to \ti{h}$ for some $\tilde{h} \in \tW_{\Omega}$.
Since $f'_{n_i} \geq f_{n_i}$ on all relevant blocks, $t(H, f'_{n_i})\geq t(H, f_{n_i})$, and so $t(H, h) \geq t(H,f) \geq (1+\delta)t(H,W_0)$.
It follows that $J_{W_0}(\ti{h}) \geq \min\{J_{W_0}(\ti{g}) : t(H,g) \geq (1+\delta) t(H, W_0)\}$, which in turn implies
\begin{align*}
    J_{W_0}(\ti{f}) &= \liminf{I_{W_0}(f_{n_{i_k}})} \geq \liminf{I_{W_0}(f'_{n_{i_k}})} +\frac{\ve\eta}{4} \geq J_{W_0}(\ti{h})+\frac{\ve\eta}{4}\\ &\geq \min\{J_{W_0}(\ti{g}) : t(H,g) \geq (1+\delta) t(H, W_0)\} +\frac{\ve\eta}{4}
    =J_{W_0}(\ti{f}) +\frac{\ve\eta}{4}.
\end{align*}
We thus have reached a contradiction.
\end{proof}

\begin{proof}[Proof of \Cref{CMC-seq}.]
By \Cref{delta-small-star-holds}, there exists $\delta_0, \ve$ such that if $f \in \W_\Omega$, $f\geq W_0$ pointwise, $f=W_0$ on irrelevant blocks (\Cref{relevant}), and $t(H, f) \leq (1+\delta_0) t(H, W_0)$, then $f$ satisfies the $\ve$-neighborhood minorant condition. Let $0 < \delta <\delta_0$, and let
$\ti{f}$ be a minimizer of the variational problem \eqref{min_eq} with $t=(1+\delta)t(H,W_0)$.

\Cref{homomorphism-cont,strictly-increasing} imply that $t(H,f)= (1+\delta) t(H, W_0)$.
By \Cref{nice-seq}, there exists a sequence $\{f_n\}_{n \geq 1}$ such that $f_n \in \W_\Omega$, $f_n \geq W_0$ pointwise, $f_n=W_0$ on irrelevant blocks, $I_{W_0}(f_n) \to J_{W_0}(\ti{f})$, and $\delta_\square(\tilde{ f}, f_n) \to 0$. Since $\delta_\square( \tilde{f}, f_n) \to 0$, $t(H,f_n) \to (1+\delta) t(H, W_0)$. Since $\delta< \delta_0$, there exists $n_0$ such that for all $n \geq n_0$, $t(H,f_n) \leq (1+ \delta_0) t(H,W_0)$. It follows by the assumption on $\delta_0$ that for all $n\geq n_0$, $f_n$ satisfies the $\ve$-neighborhood minorant condition.
\end{proof}

\noindent
The next proposition derives crucial properties for the minimizer of the relative entropy problem, which will be useful in the subsequent analysis. 

\begin{proposition}\label{proposition:equality-in-constraint} Let $m \in \mathbb{Z}^+$, $\gamma \in \Delta_m$, $W_0 \in \mathcal{B}^\gamma$, a finite regular graph $H$, $\tau = t(H, \cdot)$, and $t(H,W_0)\leq t\leq t^{\tau}_{\max}(\tW_{\Omega})$.
If $f$ is a minimizer of
$$\min\{I_{W_0}(g): g \in \mathcal{B}^\gamma, t(H,g) \geq
t\},
$$ then
$W_0\leq f$
pointwise, with equality on the irrelevant set. Furthermore, $t(H,f) =t$.

\end{proposition}
\begin{proof}
Let $W_0=(p_{ij})_{i,j \in [m]}$ and let $f= (\alpha_{ij})_{i,j \in [m]}$ be a minimizer. First, if
$\alpha_{ij}<p_{ij}$, we can decrease $I_{W_0}$ while maintaining the constraint $ t(H,g) \geq t$ by increasing
$\alpha_{ij}$.  Next observe that if
 $I_i\times I_j$ is
irrelevant, then $ t(H,f)$ does not depend on $f_{ij}$; since $h_p(\beta)$ has a unique minimum at $\beta=p$, this implies
$f_{ij}=p_{ij}$.

To prove the last statement,
suppose for contradiction that $t(H, f)> t$. There exists a relevant block $I_{a} \times I_{b}$ such that $\alpha_{ab} > p_{ab}$. Let $R$ be the union of the relevant blocks. Let $f_\beta$ be the following graphon
$$
f_\beta(x,y)=
\begin{cases}
\beta & (x,y) \in I_a \times I_b\\
f(x,y) & (x,y) \in R \setminus I_a \times I_b\\
W_0(x,y) & \text{otherwise}.
\end{cases}
$$
Since $h_p(\beta)$ is strictly increasing for $\beta\in [p,1]$,
$I_{W_0}(f_\beta)$ is strictly increasing for $\beta\geq p_{ab}$.  Combined with the continuity of $t(H, f_\beta)$ as a function of $\beta$, we conclude that there exists $\beta\in [p_{ab}, \alpha_{ab})$ such  that $t(H,f_{\beta}) >t$ and  $I_{W_0}(f) > I_{W_0}(f_\beta)$. This is a contradiction.
\end{proof}

\subsection{A symmetric regime for larger $\delta$}
We now establish the existence of a symmetric regime for larger $\delta$ (\Cref{sym-regime-near-one}). 

\begin{lemma} \label{big-delta-cmc}
Fix $m \in \mathbb{Z}^+$, $\gamma \in \Delta_m$, $W_0 \in \mathcal{B}^\gamma$ and $H$ a finite $d$-regular graph. There exists $\delta=\delta(W_0, H), \ve=\ve(W_0)>0$ for which $(1+\delta)t(H, W_0) < \max_{g\in \W_\Omega} t(H,g)$ such that the following is true: If $f \in \W_\Omega$, $f \geq W_0$ pointwise, $f= W_0$ on irrelevant blocks, and $t(H,f) \geq (1+\delta) t(H,W_0)$, then $f$ satisfies the $\ve$-neighborhood minorant condition.
\end{lemma}

\begin{proof}

Let $W_0=(p_{ij})_{i,j \in [m]}$.
By \Cref{all-about-psi}, for $p_{ij} \in (0,1)$ the function $\psi_{p_{ij}}:(0,1)\to \R$ is either convex 
or $\hat{\psi}_{p_{ij}}$ is constructed by replacing $\psi_{p_{ij}}$ with its lower common tangent on exactly one interval. 
If $\psi_{p_{ij}}$ is convex, then the  minorant condition is trivially satisfied for the $(i,j)$ block. Otherwise, let $\underline{q}_{ij}$ and $\overline{q}_{ij}$ be
such that $\underline{q}_{ij}^d$ and $\overline{q}_{ij}^d$ are respectively the left and right points of intersection between $\psi_{p_{ij}}$ and its lower common tangent.
Let $P=\{(i,j): \psi_{p_{ij}}''(x)=0 \text{ for some } x \in [p_{ij}^d,1] \}$, let $R$ denote the relevant blocks, and define
\begin{align*}
\ve_1 &= \min_{(i,j)\in P }\left(1 - \overline{q}_{ij}\right)^d\\
\ve_2 &= \min_{(i.j) \in P \setminus R}~~ \underline{q}_{ij}^d - p_{ij}^d\\
\ve &= \min\{\ve_1, \ve_2\}.
\end{align*}
Since $\psi'_{p_{ij}}(p_{ij}^d) = 0$ and  $\psi'_{p_{ij}}(1) = \infty$ for all $i,j$, we have $\underline{q}_{ij} > p_{ij}$ and $\overline{q}_{ij} < 1$ for $(i,j) \in P$, so that $\ve > 0$. Let $g_{ij}$ be the graphon in $\W_\Omega$ that takes value $\overline{q}_{ij} +\ve^{1/d}$ on the block $I_i \times I_j$ and takes value $1$ on the rest of $\Omega$. Let $\delta$ be such that $(1+\delta)t(H,W_0)=\max_{(i,j) \in P} t(H, g_{ij})$. 

Suppose  $f \in \W_\Omega$, $f \geq W_0$ pointwise, $f=W_0$ on irrelevant blocks, and $t(H,f ) \geq (1+\delta)t(H,W_0) \geq t(H,g_{ij})$. \Cref{lemma:new- holder-general} implies that for $(i,j) \in P$,
$t(H,g_{ij}) \leq t(H,f)\leq  t(H, f^*)$
where $f^*= ( \Vert f_{ij} \Vert_d )_{i,j \in [m]}$ is the $d$-averaged graphon. Let $(a,b) \in P$. When $I_a \times I_b$ is irrelevant, $\Vert f_{ab} \Vert_d^d=p_{ab}^d \leq \underline{q}_{ab}^d - \ve$. Since $\psi_{p_{ab}} = \hat{\psi}_{p_{ab}}$ on $[0, \underline{q}_{ab}^d]$, $\psi_{p_{ab}} = \hat{\psi}_{p_{ab}}$ in an $\ve$-neighborhood around $\Vert f_{ab}\Vert_d^d$. Next suppose $I_a \times I_b$ is relevant. Since $t(H,f^*)\geq t(H,g_{ab})$ and $g_{ab} \geq f^*_{ab}$ on $[0,1]^2 \setminus (I_a \times I_b)$, $f^*$ must be greater than $g_{ab}$ on $I_a \times I_b$.  It follows that $\Vert f_{ab} \Vert_d \geq \overline{q}_{ab}+\ve^{1/d}$, and so $\Vert f_{ab} \Vert_d^d \geq \overline{q}_{ab}^d +\ve$. Since $\psi_{p_{ab}} = \hat{\psi}_{p_{ab}}$ on $[ \overline{q}_{ab}^d,1]$, $\psi_{p_{ab}}=\hat{\psi}_{p_{ab}}$ in an $\ve$-neighborhood around $\Vert f_{ab} \Vert_d^d$.
\end{proof}

\begin{proof}[Proof of \Cref{sym-regime-near-one}.]
Let $W_0 \in \mathcal{B}^{\gamma,*}$.  By \Cref{big-delta-cmc}, there exist $\delta, \ve>0$ such that if $f \in \W_\Omega$, $f \geq W_0$ pointwise, $f=W_0$ on irrelevant blocks,
and $t(H,f) \geq (1+\delta) t(H,W_0)$, then $f$ satisfies the $\ve$-neighborhood minorant condition. Let $\eta$ be such that $(1+\delta) t(H,W_0)= (1-\eta)t_{\max}$. Since $(1+\delta)t(H, W_0) < \max_{g\in \W_\Omega} t(H,g)$, it holds that $\eta >0$. Let $t\in  ((1-\eta)t_{\max},t_{\max} ] $

 Let $g$ be a minimizer of the variational problem \eqref{min_eq} with this value of $t$. \Cref{nice-seq} implies that there exists a sequence $g_n \in \W_\Omega$ such that $I_{W_0}(g_n) \to J_{W_0}(\ti{g})$, each $g_n \geq W_0$ pointwise, and $\delta_\square(g_n, \ti{g}) \to 0$. It follows that $t(H,g_n) \to t(H,g) \geq t$. Since $t >(1+\delta) t(H,W_0)$, there exists some $n_0$ such that for all $n \geq n_0$, $t(H, g_n) \geq (1+\delta) t(H,W_0)$. It follows by \Cref{big-delta-cmc} that for all $n \geq n_0$, $g_n$ satisfies the $\ve$-neighborhood minorant condition. \Cref{sym-condition} implies that $\ti{g} \in \mathcal{B}^\gamma$, which establishes symmetry in this case. The rest of the proof is the same as that of \Cref{sym-regime}. 

\end{proof}

\section{Symmetry breaking in special cases}\label{sec:non-sym}
In previous sections we let $\gamma \in \Delta_m$ and defined $\mathcal{B}^\gamma$ as the set of block graphons in which the interval structure is given by the vector $\gamma$. In this section, we let $\gamma \in (0,1) \cap \mathbb{Q}$ and let $\mathcal{B}^{(\gamma, 1-\gamma)}$ denote the set of block graphons with two intervals, the first of length $\gamma$ and the second of length $1-\gamma$. 

Recall the definition
\begin{align*}
f_{p,q,r}^{\gamma}(x,y) &= \begin{cases}
p & \text{if } (x,y) \in [0,\gamma]^2\\
r &\text{if } (x,y) \in (\gamma, 1]^2\\
q & \text{otherwise.}
\end{cases}
\end{align*}
 We prove \Cref{nonsym opt}, which establishes the existence of a non-symmetric regime for graphons of the form $f_{0,p,p}^\gamma,f_{1,p,p}^\gamma,$ and  $f_{1,p,0}^\gamma$ when $p$ is sufficiently small.  

\begin{lemma}\label{nice-i-min} Let $\gamma \in (0,1)\cap \mathbb{Q}$ and $W_0\in\mathcal{B}^{(\gamma, 1-\gamma)}$ be a graphon with $Im(W_0) \in \{0,p,1\}$.  If $g \in \W_\Omega$, then $J_{W_0}(\ti{g})= I_{W_0}(g).$
 Moreover, if $\tau$ is a continuous graph parameter
 and $W_0$ is a graphon of the form $f_{z,p,p}^\gamma$ with $z \in \{0,1\}$, 
then
$$\min\{ J_{W_0}(\ti{f}): \ti{f} \in \ti{\mathcal{B}}^{(\gamma,1-\gamma)},\, \tau(\ti{f}) \geq t\}= \min\{ I_{W_0}(f): f \in \mathcal{B}^{(\gamma,1-\gamma)} \cup \mathcal{B}^{(1-\gamma, \gamma)}, \, \tau(f) \geq t\},$$
provided  $t \in \R$ is such that the above minima are finite.
If $W_0$ is a graphon of the form $f_{z_1,p,z_2}^\gamma$  with $z_1, z_2 \in \{0,1\}$, 
then
\begin{align}
\min\{ J_{W_0}(\ti{f}): \ti{f} \in \ti{\mathcal{B}}^{(\gamma,1-\gamma)},\, \tau(\ti{f}) \geq t\}= \min\{ I_{W_0}(f): f \in \mathcal{B}^{(\gamma,1-\gamma)}, \, \tau(f) \geq t\}, \nonumber
\end{align}
again provided  $t \in \R$ is such that these  minima are finite.
\end{lemma}
\noindent 
\Cref{b-gamma-compact} establishes that sets of the form $\{\ti{f} \in \ti{\mathcal{B}}^{(\gamma,1-\gamma)}:\, \tau(\ti{f}) \geq t\}$ are compact under $\delta_\square$, and therefore the above left-hand minima are well-defined. Since $f \in \mathcal{B}^{(\gamma,1-\gamma)}$ can be identified with $[0,1]^3$, the set $\mathcal{B}^{(\gamma,1-\gamma)}$ is compact with respect the Euclidean metric. Since $I_{W_0}$ is continuous, it follows that the righthand minima are well-defined.

\begin{lemma}\label{construct nonsym opt}
Let $\gamma \in (0,1) \cap \mathbb{Q}$ and $H$ be a $d$-regular graph with $v$ vertices.  Assume that
\begin{enumerate}
    \item \label{ns0pp} $0 < t< t(H, f^\gamma_{0,0,1})$ and denote $W_p=f_{0,p,p}^\gamma$, $\mathscr{C}(\gamma) = \mathcal{B}^{(\gamma,1-\gamma)} \cup \mathcal{B}^{(1-\gamma, \gamma)}$,  or
\item \label{ns1pp} $t(H, f^\gamma_{1,0,0}) < t< 1$ and denote $W_p=f_{1,p,p}^\gamma$, $\mathscr{C}(\gamma) = \mathcal{B}^{(\gamma,1-\gamma)} \cup \mathcal{B}^{(1-\gamma, \gamma)}$,  or
\item \label{ns1p0} $t(H,f^\gamma_{1,0,0}) < t< t(H, f^\gamma_{1,1,0})$ and denote $W_p=f_{1,p,0}^\gamma$, $\mathscr{C}(\gamma)=  \mathcal{B}^{(\gamma,1-\gamma)} $.
\end{enumerate}
Separately, under each of these assumptions, there exists $p_0 > 0$ such that if $p < p_0$,
\[\inf\{I_{W_p}(f): t(H,f) \geq t\} < \min\{ I_{W_p}(f): f \in \mathscr{C}(\gamma), \, t(H,f) \geq t\}.\]
\end{lemma}

\begin{proof}[Proof of \Cref{nonsym opt}]
We apply \Cref{construct nonsym opt} to conclude that there exists $p_0 > 0$ such that if $p < p_0$,
\[\inf\{I_{W_p}(f): t(H,f) \geq t\} < \min\{ I_{W_p}(f): f \in \mathscr{C}(\gamma), \, t(H,f) \geq t\}.\]
\Cref{homomorphism-cont,phi is h,nice-i-min} imply that
$$\min \{ J_{W_p}(\ti{g}) : t(H,g) \geq t\} < \min\{ J_{W_p}(\ti{f}): \ti{f} \in \ti{\mathcal{B}}^{(\gamma,1-\gamma)},\, t(H,f) \geq t\}.$$
Therefore if $\ti{g}$ is a minimizer of \eqref{min_eq}, then $\ti{g} \not \in \ti{\mathcal{B}}^{(\gamma,1-\gamma)}$, meaning $t$ is not in the symmetric regime.
\end{proof}

\subsection{Proof of \Cref{nice-i-min}}

\begin{proposition}\label{IW0-to-Ip}
Let $\gamma \in (0,1) \cap \mathbb{Q}$ and $W_0\in\mathcal{B}^{(\gamma, 1-\gamma)}$  be a graphon such that $Im(W_0) \subseteq \{0, p,1\}$. Let $I_p(f)= \int_{[0,1]^2} h_p(f(x,y)) dx dy$, and let $\Omega_q=\{(x,y) \in [0,1]^2 : W_0(x,y)=q\}$ for $q \in \{0,p,1\}$. For all $f \in \W_\Omega$, $$I_{W_0}(f)=I_{p}(f) - |\Omega_0| h_p(0) - |\Omega_1| h_p(1).$$
\end{proposition}

\begin{proof}
Suppose $f \in \W_\Omega$. Then for $q \in \{0,1\}$, $f=q$ almost everywhere on $\Omega_q$. It follows that
\begin{align*}
    I_{W_0}(f) &= \int_{\Omega_p} h_p(f(x,y)) dxdy + \int_{\Omega_0} h_0(f(x,y)) dx \,dy +\int_{\Omega_1} h_1(f(x,y)) dx \,dy\\
    &=\int_{\Omega_p} h_p(f(x,y)) dx\,dy\\
    &=\int_{[0,1]^2} h_p(f(x,y)) dx\,dy - |\Omega_0| h_p(0) - |\Omega_1| h_p(1)\\
    &= I_{p}(f) - |\Omega_0| h_p(0) - |\Omega_1| h_p(1).
\end{align*}
\end{proof}

\begin{lemma}\label{good int}Let $\gamma \in (0,1) \cap \mathbb{Q}$ and $W_0\in\mathcal{B}^{(\gamma, 1-\gamma)}$.
Let $\tW_\Omega$ be defined with respect to $W_0$.
\begin{itemize}
	\item[(i)] Suppose $W_0= f_{z,p,p}^\gamma$, and $z \in \{0,1\}$. If $\ti{f} \in \ti{\mathcal{B}}^{(\gamma,1-\gamma)} \cap \tW_\Omega$, there exists $g$ such that $\delta_\square(\ti{f},g)=0$ and $g \in (\mathcal{B}^{(\gamma,1-\gamma)} \cup \mathcal{B}^{(1-\gamma, \gamma)}) \cap \W_\Omega$.
	\item[(ii)] Suppose $W_0= f_{z_1,p,z_2}^\gamma$, and $z_1, z_2 \in \{0,1\}$.  If $\ti{f} \in \ti{\mathcal{B}}^{(\gamma,1-\gamma)} \cap \tW_\Omega$, 
	there exists $g$ such that $\delta_\square(\ti{f},g)=0$ and $g \in \mathcal{B}^{(\gamma,1-\gamma)} \cap \W_\Omega$.

\end{itemize}
\end{lemma}

\begin{proof}
Since $\ti{f} \in \ti{\mathcal{B}}^{(\gamma, 1-\gamma)}\cap \tW_\Omega$, there exists some $g \in \mathcal{B}^{(\gamma, 1-\gamma)}$  and $h \in \W_\Omega$ such that $\delta_\square(\ti{f},h) =0$ and $\delta_\square(\ti{f},g) =0$. Thus $\delta_\square(\ti{h}, \ti{g}) =0$.
Since $g$ is of the form $g=\sum_{ij}\alpha_{ij}1_{Y_i}\times 1_{Y_j}$ with $Y_1=[0,\gamma]$ and $Y_2=(\gamma,1]$, we
can use \Cref{step-functions} to conclude that $h$ must be of the same form with appropriate sets $Y_1', Y_2'$ of sizes $\gamma$ and $1-\gamma$.

First suppose that $W_0$ has the form $f_{z,p,p}^\gamma$ for $z \in \{0,1\}$.
We consider cases:
 \begin{enumerate}[(a)]
     \item If $|Y_1' \cap [0, \gamma]|>0$, then $\alpha_{11}=z$ and
     $g \in \mathcal{W}_\Omega$, which completes the proof.
     \item If $|Y_1' \cap [0, \gamma]|=0$, it must be the case that $\gamma\leq  1/2$  and $|Y_2' \cap [0,\gamma]|>0$,
      implying that $\alpha_{22}=z$.
      Note that we can always re-define $g$ on the measure zero set $(\{\gamma\} \times [0,1]) \cup ([0,1] \times \{\gamma\})$
      so that $g^{\phi} \in \mathcal{B}^{(1-\gamma, \gamma)}$ for $\phi(x) = 1- x$.  Since $\gamma \leq 1/2$, $g^\phi=z$ on $[0, \gamma] \times [0, \gamma]$ meaning $g^\phi \in \W_\Omega$. By construction, $g^\phi \in \mathcal{B}^{(1-\gamma, \gamma)}$ and $\delta_\square(g^\phi, \ti{f})=0$.
 \end{enumerate}

Next, suppose that $W_0$ has the form $f_{z_1,p,z_2}^\gamma$ for $z_1, z_2 \in \{0,1\}$. As $h\in \W_\Omega$, $h$ takes value $z_1$ on $[0, \gamma]^2$ and value $z_2$ on $(\gamma,1]^2$. 

 \begin{enumerate}[(a)]
     \item If $|Y'_1 \cap [0, \gamma]|>0$ and $|Y'_2 \cap (\gamma,1]|>0$, then
     $\alpha_{11}=z_1$, $\alpha_{22}=z_2$ and again $g\in\W_\Omega$.
     \item If $|Y'_1 \cap [0, \gamma]|=0$, then $\gamma \leq 1/2$, $|Y'_2 \cap [0, \gamma]|>0$ and $|Y_1' \cap ( \gamma,1]|>0$.  It follows that $\alpha_{11}=z_2$ and $\alpha_{22}=z_1$.
     \begin{itemize}[-]
         \item  Suppose $\gamma =1/2$. As before, by re-defining $g$ on the boundary if necessary, we note that for $\phi(x) = 1-x$, $g^\phi$ takes value $z_1$ on $[0,\gamma] \times [0,\gamma]$, and value $z_2$ on $(\gamma, 1] \times (\gamma, 1]$.
        Thus $g^\phi\in \W_\Omega$. By construction $g^\phi \in \mathcal{B}^{(\gamma, 1-\gamma)}$ and $\delta_\square(g^\phi, \ti{f})=0$.
         \item   If $\gamma<1/2$, then $|Y'_2|= 1-\gamma  = |( \gamma,1]|>1/2$. Thus $|Y'_2 \cap (\gamma,1]|>0$,
         and $\alpha_{22}=z_2$. Thus $z_2=z_1$ and $h$ must take value $z_1 = z_2$ almost everywhere, meaning $g$ does as well, so $g \in \W_\Omega \cap \mathcal{B}^{(\gamma, 1-\gamma)}$.
     \end{itemize}
     \item The case that $|Y'_2 \cap ( \gamma,1]|=0$ follows analogously to the above case.

  \noindent
  Note that if $\gamma \neq 1/2$, cases (b) and (c) only occur when $z_1=z_2$. Therefore when $z_1 \not = z_2$, $g \in \mathcal{B}^{(\gamma, 1-\gamma)} \cap \W_\Omega$.
 \end{enumerate} \end{proof}

\begin{proof}[Proof of \Cref{nice-i-min}.] First we show that if $g \in \W_\Omega$, then $J_{W_0}(\ti{g})= I_{W_0}(g).$ Let
$g \in \W_\Omega$. There exists a sequence of graphons $\{g_n\}_{n \geq 1}$ with each $g_n \in \W_\Omega$ such that $I_{W_0}(g_n ) \to J_{W_0}(\ti{g})$ and $\delta_\square(g_n, \ti{g}) \to 0$. It follows that there exists a sequence $\phi_n \in \mathcal{M}$ such that $d_\square(g_n^{\phi_n},g) \to 0$.  Note that $I_p(g_n^{\phi_n})= I_p(g_n)$. Let $c=|\Omega_0| h_p(0) +|\Omega_1| h_p(1)$. By \Cref{IW0-to-Ip}, $I_{W_0}(g)= I_p(g) -c$ and
$I_{W_0}(g_n)= I_p(g_n) -c$ since $g_n, g \in \W_\Omega$.
Leveraging the lower semi-continuity of $I_p$ with respect to $d_\square$ (\Cref{lemma:semi-continuity-I}), we obtain
\begin{align*}
    J_{W_0}(\ti{g}) &= \liminf_{n \to \infty} I_{W_0}(g_n)
     = \liminf_{n \to \infty} I_{p}(g_n)-c
    = \liminf_{n \to \infty} I_{p}(g_n^{\phi_n})-c
    \geq I_p(g) -c
    =I_{W_0}(g).
\end{align*}
Since the definition of $J_{W_0}$ implies that $J_{W_0}(\ti{g}) \leq I_{W_0}(g)$, it follows that $J_{W_0}(\ti{g}) =I_{W_0}(g)$.

Next suppose $W_0$ is of the form $f_{z,p,p}^\gamma$ or $f_{z_1,p,z_2}^\gamma$ where $z,z_1, z_2 \in \{0,1\}$ and $z_1= z_2$. Clearly $\min\{ I_{W_0}(f): f \in \mathcal{B}^{(\gamma,1-\gamma)} \cup \mathcal{B}^{(1-\gamma, \gamma)}, \,\tau(f) \geq t\}\geq \min \{ J_{W_0}(\ti{g}) : \ti{g} \in \ti{\mathcal{B}}^{(\gamma,1-\gamma)}, \tau(g) \geq t\}.$ Let $h$ be such that $J_{W_0}(\ti{h})= \min \{ J_{W_0}(\ti{g}) : \ti{g} \in \ti{\mathcal{B}}^{(\gamma,1-\gamma)}, \tau(\ti{g}) \geq t\}$. \Cref{good int} implies that we may assume $h \in (\mathcal{B}^{(\gamma,1-\gamma)} \cup \mathcal{B}^{(1-\gamma, \gamma)}) \cap \W_\Omega$.
Observe
\begin{align*}
    \min \{ J_{W_0}(\ti{f}) : \ti{f} \in \ti{\mathcal{B}}^{(\gamma,1-\gamma)}, \tau(f) \geq t\}&=J_{W_0}(\ti{h})
    =I_{W_0}(h)\\
    &\geq \min\{ I_{W_0}(f): f \in \mathcal{B}^{(\gamma,1-\gamma)} \cup \mathcal{B}^{(1-\gamma, \gamma)}, \, \tau(f) \geq t\}.
\end{align*}
This establishes the claim in these cases. The proof for $W_0 = f_{z_1, p,z_2}^{\gamma}$, $z_1, z_2 \in \{0,1\}$, $z_1 \neq z_2$ is analogous.
\end{proof}

\subsection{Proof of \Cref{construct nonsym opt}}

Our construction of a non-symmetric graphon with lower entropy than any symmetric graphon is different for each of the three cases. Each proof uses the following proposition.
\begin{proposition}\label{really-edge}
Let $\{W_p\}_{p \in (0,1)}$ be a family of graphons where $W_p$ takes value in $\{0,p,1\}$. Assume further that the sets where $W_p$ assumes the values $\{0,p,1\}$ is the same for all $p\in (0,1)$.
For $i \in \{0,p,1\}$, let $\Omega_i= \{(x,y) \in [0,1]: W_p(x,y)=i\}$. Let $E$ be the graph with two vertices and one edge.
If $f \in \W_\Omega$, then $$\lim_{p \to 0} \frac{I_{W_p}(f)}{\log{1/p}}=\frac{1}{2} \brac{ t(E,f) - |\Omega_1|}.$$
\end{proposition}

\begin{proof}
First observe that for a fixed $\alpha$, 
\begin{align}\label{old-fact}
\lim_{p \to 0} \frac{h_p(\alpha)}{\log \frac{1}{p}}= \lim_{p \to 0} \frac{\alpha \log\frac{\alpha}{p} + (1-\alpha) \log \frac{1-\alpha}{1-p} }{\log \frac{1}{p}}
=\alpha.
\end{align}
Since $\frac{ h_p(f)}{\log{1/p}}$ is bounded as $p\to 0$ (using Lemma \ref{lem:h_p}) and $f \in \W_\Omega$, the Dominated Convergence Theorem implies that
\begin{align*}
    \lim_{p \to 0} \frac{I_{W_p}(f)}{\log{1/p}}= \lim_{p \to 0}\frac{1}{2}\int_{\Omega_p} \frac{ h_p(f)}{\log{1/p}} = \frac{1}{2}\int_{\Omega_p}\lim_{p \to 0} \frac{ h_p(f)}{\log{1/p}} =\frac{1}{2} \int_{\Omega_p} f = \frac{1}{2} \brac{ t(E,f)- |\Omega_1|}.
\end{align*}
\end{proof}

\subsubsection{\ER graphs with a planted
independent set}\label{sec:6-non-symmetric-0ppp}
In this subsection, we prove the existence of a non-symmetric regime for $d$-regular subgraph counts in graphons of the form $f^\gamma_{0,p,p}$ when $p$ is sufficiently small.  We will do this by showing that the union of isolated vertices with a clique will
have lower relative entropy than the minimum in $ \mathcal{B}^{(\gamma,1-\gamma)}  \cup \mathcal{B}^{(1-\gamma, \gamma)}$.

\begin{proof}[Proof of \Cref{construct nonsym opt}, Statement \ref{ns0pp}.] Let $W_p= f^\gamma_{0,p,p}$.
First note that
\begin{align}\label{simp0}
&\min\{ I_{W_p}(f): f \in \mathcal{B}^{(\gamma,1-\gamma)}  \cup \mathcal{B}^{(1-\gamma, \gamma)}, \, t(H,f) \geq t\}\nonumber\\
&=\min\{ I_{W_p}(f_{0, \alpha,\beta}^z): \alpha, \beta \in  [p,1], z \in \{\gamma, 1-\gamma\}, \,t(H, f^z_{0,\alpha, \beta}) \geq t\}.
\end{align}
We can restrict to graphons of the form $f^z_{0,\alpha,\beta}$ since $I_{W_p}(f_{\eta, \alpha,\beta}^z)= \infty$ when $\eta \not =0$. We can further restrict to $\alpha, \beta \geq p$ since $h_p(\cdot)$ is decreasing on $[0, p]$ and $t(H, \cdot)$ is an increasing function.

Define the non-symmetric graphon $\chi_t$ as follows.
\begin{align*}
\chi_t(x,y) &= \begin{cases}
1 & (x,y) \in [1-t^{\frac{1}{v}}, 1]^2\\
0 & \text{otherwise.}
\end{cases}
\end{align*}
In other words, the graphon $\chi_t$ is the union of a clique and isolated vertices that has the required subgraph density (by the fact that $t(H,\chi_t)=t$).  Note that
the assumption that $t \leq t(H, f^\gamma_{0,0,1})$ implies that $t^{1/v} \leq 1-\gamma$ and so $\chi_t \in \W_\Omega$.
By \eqref{simp0}, it suffices to show
\begin{align}\label{compare limits}
\lim_{p \to 0} \frac{I_{W_p}(\chi_t)}{\log \frac{1}{p}} < \lim_{p \to 0} \frac{I_{W_p}(f_{0,\alpha, \beta}^{z})}{\log \frac{1}{p}}
\end{align}
for each triple $z \in \{\gamma, 1-\gamma\}$, $\alpha, \beta \in [p,1]$ such that $t(H, f^z_{0,\alpha, \beta}) \geq t$.

Let $E$ be the graph with two vertices and one edge. By \Cref{really-edge}, for $z\in \{ \gamma, 1-\gamma\}$  $$\lim_{p \to 0} \frac{I_{W_p}(f_{0,\alpha, \beta}^{z})}{\log \frac{1}{p}}= \frac{1}{2} t(E, f_{0,\alpha, \beta}^z)\quad \text{ and } \quad \lim_{p \to 0} \frac{I_{W_p}(\chi_t)}{\log \frac{1}{p}} = \frac{1}{2} t(E, \chi_t)
= \frac{1}{2} t^{\frac{2}{v}}.$$  Therefore to establish \eqref{compare limits}, it suffices to show that
\begin{align*}
t ^{\frac{2}{v}} \leq t(H, f_{0,\alpha, \beta}^{z})^{\frac{2}{v}} < t (E, f_{0,\alpha, \beta}^{z}).
\end{align*}
By the generalized H\"older inequality (\Cref{theorem:holder}) and the facts that $E(H)=dv/2$ and $g^d\leq g$ for any graphon $g$,
\begin{align}\label{genholder}
t(H, g) &= \int_{[0,1]^v} \prod_{(i,j) \in E(H)} g(x_i, x_j) dx_1 \dots dx_v \leq \left(\int_{[0,1]^2} \left(g(x,y)\right)^d \, dx \, dy \right)^{\frac{e(H)}{d}}\nonumber\\
&\leq \left(\int_{[0,1]^2} g(x,y) \, dx \, dy \right)^{\frac{v}{2}}
= t(E, g)^{\frac{v}{2}}
\end{align}
The second inequality is strict if $g$ takes values in $(0,1)$ on a set with positive measure, as is the case for all graphons in the set of potential minimizers described on the right hand side of  \eqref{simp0}, with the exception of graphons of the form  $f^z_{0,1,1}$. We claim that for such graphons, the first inequality is strict. When we apply \Cref{theorem:holder}, we take $f_{1}(x_i, x_j)= f^z_{0,1,1}(x_i,x_j)$. To have equality, $f_1$ must have a product representation $f_1(x_i,x_j)=f_{1i}(x_i)f_{1j}(x_j)$. Suppose such a representation exists. For $x_i,x_j \in [0,z)^2$, $f_1(x_i,x_j) =0$, and so $f_{1i}(x_i) =0$ or $f_{1j}(x_j)=0$. However, then $f_1(x_i, \cdot) = 0$ or $f_1(\cdot, x_j) = 0$ on a set of positive measure, which is not consistent with the structure of the graphon $f^z_{0,1,1}$.
Finally, taking $g= f_{0,\alpha, \beta}^z$ and rearranging establishes \eqref{compare limits}.
\end{proof}

\subsubsection{\ER graphs with a planted clique}
In this subsection, we prove the existence of a non-symmetric regime for $d$-regular subgraph counts in graphons of the form $f^\gamma_{1,p,p}$ when $p$ is sufficiently small.
Again it will be  the union of a clique and isolated vertices which has lower relative entropy than the minimizer in $ \mathcal{B}^{(\gamma,1-\gamma)}  \cup \mathcal{B}^{(1-\gamma, \gamma)}$.

\begin{proof}[Proof of \Cref{construct nonsym opt}, Statement \ref{ns1pp}.] Let $W_p= f^\gamma_{1,p,p}$.
First note that
\begin{align}\label{simp1}
&\min\{ I_{W_p}(f): f \in \mathcal{B}^{(\gamma,1-\gamma)}  \cup \mathcal{B}^{(1-\gamma, \gamma)}, \, t(H,f) \geq t\}\nonumber\\
&=\min\{ I_{W_p}(f_{1, \alpha,\beta}^z): z \in \{\gamma, 1-\gamma\}, \alpha, \beta \in [p,1], \,t(H, f^z_{1,\alpha, \beta})\geq t\},
\end{align}
by similar reasoning to the proof of \Cref{construct nonsym opt}, Statement \ref{ns0pp}.
Define the non-symmetric graphon $\chi_t$ as follows.
\begin{align*}
\chi_t(x,y) &= \begin{cases}
1 & (x,y) \in [0, t^{\frac{1}{v}}]^2\\
0 & \text{otherwise.}
\end{cases}
\end{align*}
In other words, the graphon $\chi_t$ is the union of a clique and isolated vertices that has the required subgraph density.
Note that since $t\geq t(H,f^{\gamma}_{1,0,0})$, $t^{1/v} \geq \gamma$ and so $\chi_t \in \W_\Omega$. By \eqref{simp1}, it suffices to show that
\begin{align*}
\lim_{p \to 0} \frac{I_{W_p}(\chi_t)}{\log \frac{1}{p}} < \lim_{p \to 0} \frac{I_{W_p}(f_{1,\alpha, \beta}^{\gamma})}{\log \frac{1}{p}}
\end{align*}
for each triple $z \in \{\gamma, 1-\gamma\}$, $\alpha, \beta \in [p,1]$ such that $t(H, f^z_{1,\alpha, \beta})\geq t$.

Let $E$ be the graph with two vertices and one edge. By \Cref{really-edge}, for $z\in \{ \gamma, 1-\gamma\}$
$$
\lim_{p \to 0} \frac{I_{W_p}(f_{1,\alpha, \beta}^{z})}{\log \frac{1}{p}}
= \frac{ t(E, f_{1,\alpha, \beta}^z)- \gamma^2}{2} \quad \text{ and } \quad \lim_{p \to 0} \frac{I_{W_p}(\chi_t)}{\log \frac{1}{p}}
=  \frac{t(E, \chi_t)- \gamma^2}{2}
= \frac{t^{\frac{2}{v}}-\gamma^2}{2}.
$$
Therefore, it suffices to show that
\begin{align*}
\frac{1}{2} \left(t^{\frac{2}{v}} - \gamma^2 \right) < \frac{1}{2}\left(t(E, f_{1,\alpha, \beta}^{z}) - \gamma^2\right)
\impliedby  t ^{\frac{2}{v}} \leq t(H, f_{1,\alpha, \beta}^{z})^{\frac{2}{v}} < t (E, f_{1,\alpha, \beta}^{z}).
\end{align*}
This holds by \eqref{genholder}, as long as one of $\alpha$ or $\beta$ is not equal to $1$. Finally, note that $f^{z}_{1,1,1}$ is clearly not an optimizer since $t(H, f^{z}_{1,1,1}) > t$ and so we can always produce $g \in \mathcal{B}^{(z,1-z)} \setminus \{f^{z}_{1,1,1}\}$ such that $t(H,g) \geq t$ and $I_{W_p}(g) < I_{W_p}(f^{z}_{1,1,1})$.
\end{proof}

\subsubsection{\ER graphs with a planted clique and independent set}
In this subsection, we prove the existence of a non-symmetric regime for $d$-regular subgraph counts in graphons of the form $f^\gamma_{1,p,0}$ when $p$ is sufficiently small.  This time, it will be  the union of a clique, a bipartite complete graph and isolated vertices which has lower relative entropy than the minimizer in $ \mathcal{B}^{(\gamma,1-\gamma)}$.

\begin{proof}[Proof of \Cref{construct nonsym opt}, Statement \ref{ns1p0}.] Let $W_p=f^\gamma_{1,p,0}$ and $ t(H, f^\gamma_{1,0,0})<t< t(H, f^\gamma_{1,1,0})$.
First note that
\begin{align}\label{simp2}
&\min\{ I_{W_p}(f): f \in \mathcal{B}^{(\gamma,1-\gamma)} , \, t(H,f) \geq t\}\nonumber\\
&=\min\{ I_{W_p}(f_{1, \alpha,0}^\gamma):  \alpha \in [0,1], \,t(H, f^z_{1,\alpha, 0})\geq t\}\nonumber\\
&=\min\{ I_{W_p}(f_{1, \alpha,0}^\gamma):  \alpha \in [0,1], \,t(H, f^z_{1,\alpha, 0})= t\}.
\end{align}
The first equality follows because $I_{W_p}(f_{\eta, \alpha,\beta}^z)= \infty$ when $\eta \not =1$ or $\beta \not =0$. The second equality follows by \Cref{proposition:equality-in-constraint}.

We construct a non-symmetric graphon $\chi_{\alpha}$ such that $t(H,\chi_\alpha)=t(H, f_{1,\alpha,0}^{\gamma})$. Let
\begin{align*}
\chi_{\alpha}(x,y) &= \begin{cases}
1 & (x,y) \in [0, \gamma+(1-\gamma) \alpha^d]^2 \setminus
(\gamma,1]^2\\
0 & \text{otherwise.}
\end{cases}
\end{align*}
Let $s_k$ be the number of labeled independent sets of size $k$ in $H$, and let $v=|V(H)|$ be the number of vertices of $H$. In any homomorphism of $H$ in $f_{1, \alpha,0}^{\gamma}$, the vertices of $H$ mapped to the interval $(\gamma,1]$ must form an independent set. Counting homomorphisms by the number of vertices that map to $(\gamma, 1]$, we obtain
$$t(H, f_{1,\alpha,0}^{\gamma})=\sum_{k=0}^v s_k \gamma^{v-k} (1-\gamma)^k \alpha^{dk} =\sum_{k=0}^v s_k \gamma^{v-k}  ((1-\gamma)\alpha^d)^k 1^{dk} =t(H,\chi_\alpha).$$
\noindent
To establish symmetry breaking, it suffices to show that $\chi_{\alpha}$ has lower entropy than the class of symmetric graphons $f^{\gamma}_{1,\alpha,0}$. Thus by \eqref{simp2}, it is enough to show that
\begin{align}\label{goal 3}
\lim_{p \to 0} \frac{I_{W_p}(\chi_\alpha)}{\log \frac{1}{p}} < \lim_{p \to 0} \frac{I_{W_p}(f_{1, \alpha, 0}^{\gamma})}{\log \frac{1}{p}}
\end{align}
for all $\alpha$ such that $t(H, f_{1,\alpha,0}^\gamma)=t$.
Since $ t(H, f^\gamma_{1,0,0})<t< t(H, f^\gamma_{1,1,0})$, $t(H, f_{1,\alpha,0}^z)\not= t$ when $\alpha\in \{0,1\}$. Thus, it suffices to establish \eqref{goal 3} when $\alpha \in (0,1)$.

\noindent
Observe using \eqref{old-fact} that
\begin{align*}
\lim_{p \to 0}\frac{I_{W_p}(\chi_\alpha)}{\log \frac{1}{p}} =\lim_{p \to 0} \frac{(1-\gamma) \alpha^d \gamma h_p(1) + (1 - \gamma-(1-\gamma) \alpha^d)\gamma h_p(0)}{\log \frac{1}{p}}= (1-\gamma) \gamma \alpha^d
\end{align*}
and
\begin{align*}
\lim_{p \to 0} \frac{I_{W_p}(f_{1,\alpha, 0}^{\gamma})}{\log \frac{1}{p}}
= \lim_{p \to 0} \frac{(1-\gamma)\gamma h_p(\alpha)}{\log \frac{1}{p}}
=(1-\gamma)\gamma \alpha.
\end{align*}
Noting that {$0 < \alpha < 1$} establishes \eqref{goal 3}, and completes the proof.
\end{proof}

\section{Bipartite \ER graphs}
\label{sec:bipartite}
In this section, we prove \Cref{theorem:variational-problem-bipartite,theorem:eigenvalue}, which precisely identify the symmetric and non-symmetric regimes for $d$-regular subgraph counts  and the operator norm in bipartite \ER graphs respectively. Throughout this section, we fix $p \in (0,1)$ and $\gamma \in (0,1) \cap \mathbb{Q}$.
We use the notation $f_p^\gamma$ to denote the bipartite graphon with density $p$ and blocks of size $\gamma$ and $1-\gamma$, as illustrated in \Cref{fig:bipartite-graphon}. 
\subsection{Density of $d$-regular subgraphs}
We will apply \Cref{lemma:symmetric,lemma:non-symmetric} to identify the symmetric and non-symmetric regimes respectively.
\begin{lemma}\label{lemma:symmetric} Let $p \in (0,1)$, $\gamma \in (0,1) \cap \mathbb{Q}$ and $W_0= f_p^\gamma$.
Let $H$ be a $d$-regular graph with $d \geq 1$. Let $0 < p \leq r \leq 1$ be such that $(r^d, h_p(r))$ is on the convex minorant of $\psi_p$. If $f \in \W_\Omega$ and $t(H,f) \geq t(H, f_r^\gamma)$, then $I_{W_0}(f) \geq I_{W_0}(f_r^\gamma)$ with equality if and only if $f=f_r^\gamma$ almost everywhere.
\end{lemma}

\begin{lemma}\label{lemma:non-symmetric} Let $p \in (0,1)$, $\gamma \in (0,1) \cap \mathbb{Q}$ and $W_0= f_p^\gamma$.
Let $H$ be a $d$-regular graph with $d \geq 1$. Let $0 < p < r <1$ be such that $(r^d, h_p(r))$ is not on the convex minorant of $\psi_p$. Then there exists $g \in \mathcal{W}$ such that
$t(H,g) >
 t(H, f_r^\gamma)$ and $I_{W_0}(g) < I_{W_0}(f_r^\gamma)$.
\end{lemma}

\noindent We now prove Theorem \ref{theorem:variational-problem-bipartite}, which completely characterizes the symmetric and non-symmetric regimes for $d$-regular homomorphism densities in bipartite \ER graphons.

\begin{proof}[Proof of Theorem \ref{theorem:variational-problem-bipartite}]
Suppose that the point $(r^d, h_p(r))$ lies on the convex minorant of $\psi_p$. We will show that $t_r^\gamma= t(H, f_r^\gamma)$ is in the symmetric regime for $t(H, \cdot)$.
Let $\ti{g} \in \tW_\Omega$ be such that $J_{W_0}(\ti{g})= \min \{J_{W_0}(\ti{f}): t(H,f) \geq t_r^\gamma\}$. We may assume that $g \in \W_\Omega.$ By \Cref{nice-i-min}, $J_{W_0}(\ti{g})= I_{W_0}(g)$. Since $t(H,g) \geq t_r^\gamma$ and $(r^d, h_p(r))$ lies on the convex minorant of $\psi_p$, \Cref{lemma:symmetric} implies that $I_{W_0}(g) \geq I_{W_0}(f_r^\gamma)$. Since $I_{W_0}(g)=J_{W_0}(\ti{g})= \min \{J_{W_0}(\ti{g}): t(H,g) \geq t_r^\gamma\} \leq I_{W_0}(f_r^\gamma)$, it follows that $I_{W_0}(f_r^\gamma)= I_{W_0}(g)$. \Cref{lemma:symmetric} implies that $g = f_r^\gamma$, meaning that $\ti{f_r^\gamma}$ is the unique symmetric solution.

Next, suppose that the point $(r^d, h_p(r))$ does not lie on the convex minorant of $\psi_p$. We will show that $t_r^\gamma= t(H, f_r^\gamma)$ is not in the symmetric regime for $t(H, \cdot)$.
 \Cref{lemma:non-symmetric} implies that there exists $g \in \W_\Omega$ such that $t(H,g) > t(H,f_r^\gamma)$ and $I_{W_0}(g) < I_{W_0}(f_r^\gamma)$.  By \Cref{nice-i-min}, $J_{W_0}(\ti{g})= I_{W_0}(g)$. We apply \Cref{nice-i-min} and obtain
\begin{align}\label{min comp}
    \min\{J_{W_0}(\ti{g}) :t(H,g) \geq t_r^\gamma, \ti{g} \in \ti{\mathcal{B}}^{(\gamma, 1-\gamma)}\}&=   \min\{I_{W_0}(g) :t(H,g) \geq t_r^\gamma, g \in \mathcal{B}^{(\gamma, 1-\gamma)} \cup \mathcal{B}^{( 1-\gamma, \gamma)} \}\nonumber \\
    &=   \min\{I_{W_0}(f_q^\gamma) : q \in [0,1], t(H,f_q^\gamma) \geq t_r^\gamma\}
    =I_{W_0}(f_r^\gamma)\nonumber\\
   & > I_{W_0}(g)\nonumber
    \geq \min \{ J_{W_0}(\ti{g}): t(H,g) \geq t_r^\gamma\}.
\end{align}
The second equality follows by noting that if $\gamma\not= 1/2$ and $g \in \mathcal{B}^{( 1-\gamma, \gamma)}$, then $I_{W_0}(g)=\infty$ or $g$ is the zero graphon. The third equality follows by noting that $I_{W_0}(f_q^\gamma)$ and $t(H, f_q^\gamma)$ are increasing functions of $q$.
\end{proof}

\subsubsection{Proof for the symmetric regime}
The following lemma describes a norm condition on $f$ that implies that the graphon $f_r^\gamma$ has lower entropy.
\begin{lemma}\label{lemma:symmetric-general}
Suppose that $d \geq 1$ and $p \leq r \leq1$ 
are such that the point $(r^d, h_p(r))$ lies on the convex minorant of $\psi_p$ and
\[
\Vert f \Vert_d^d \geq 2 \gamma (1-\gamma) r^d .
\]
Then $I_{W_0}(f) \geq I_{W_0}(f_r^\gamma)$, with equality occurring if and only if $f = f_r^\gamma$ almost everywhere.
\end{lemma}
\begin{proof}The statement is trivial if $f \not \in
\W_\Omega$. For $f\in
\W_\Omega$,
\begin{align}
I_{W_0}(f) &= \int_{0}^{\gamma} \int_{\gamma}^1 h_p(f(x,y)) dx dy
= \int_{0}^{\gamma} \int_{\gamma}^1 \psi_p\left(f^d(x,y)\right) dx dy \nonumber\\
&\geq \int_{0}^{\gamma} \int_{\gamma}^1 \hat{\psi}_p \left(f^d(x,y)\right) dx dy \nonumber\\
&\geq \gamma(1-\gamma) \hat{\psi}_p\left( \frac{1}{\gamma(1-\gamma)} \int_0^{\gamma} \int_{\gamma}^1 f^d(x,y) dx dy  \right) \label{eq:jensen}\\
& = \gamma(1-\gamma) \hat{\psi}_p\left( \frac{1}{\gamma(1-\gamma)} \frac{\Vert f \Vert_d^d }{2} \right)\nonumber\\
&\geq \gamma (1-\gamma) \hat{\psi}_p(r^d) \label{eq:increasing}\\
&= \gamma(1-\gamma) \psi_p(r^d)
= \gamma(1-\gamma) h_p(r)
= I_{W_0}(f_r^\gamma). \nonumber
\end{align}
Note that \eqref{eq:jensen} is an application of Jensen's inequality, and \eqref{eq:increasing} is due to $\hat{\psi}_p$ being an increasing function on $[r^d, 1]$. If $f \not = f_r^\gamma$, then the step using Jensen's inequality is a strict inequality. Therefore, $I_{W_0}(f) \geq I_{W_0}(f_r^\gamma)$, with equality occurring if and only if $f = f_r^\gamma$ almost everywhere.
\end{proof}

\noindent
The following lemma establishes a norm condition on graphons that satisfy the subgraph density requirement.
\begin{lemma}\label{lemma:norm} Let $p,r \in (0,1)$, $\gamma \in (0,1) \cap \mathbb{Q}$ and $W_0= f_p^\gamma$.
Let $H$ be a $d$-regular graph with $d \geq 1$. Let $f \in \W_\Omega$ be such that $t(H,f) \geq t(H, f_r^\gamma)$. Then $\Vert f \Vert_d^d \geq 2 \gamma (1-\gamma) r^d$.
\end{lemma}
\begin{proof}
We may assume $H$ is bipartite. Since $H$ is $d$-regular, $H$ must have $m$ vertices in each partition class and $dm$ edges for some $m \in \mathbb{Z}^+$. Let $c$ be the number of connected components of $H$. Note that $t(H, f_r^\gamma)= 2^c (\gamma(1-\gamma))^ m r^{dm}$.
Let $f$ be any graphon such that $I_{W_0}(f) < \infty$ and $t(H,f) \geq t(H, f_r^\gamma)$. Then $t(H, f_r^\gamma) \leq t(H, f)$ implies
\begin{align}
t(H, f_r^\gamma) &= 2^c \gamma^m (1-\gamma)^m r^{dm} \leq \int_{[0,1]^{2m}} \prod_{(i,j) \in E(H)} f(x_i, x_j) dx_1, \dots, dx_{2m} \nonumber\\
&= 2^c \int_{0}^{\gamma} \int_{\gamma}^1 \cdots \int_{0}^{\gamma} \int_{\gamma}^1 \prod_{(i,j) \in E(H)} f(x_i, x_j) dx_1, \dots, dx_{2m} \label{eq:alternation}\\
&\leq 2^c \gamma^m (1-\gamma)^m \prod_{(i,j) \in E(H)} \left(\frac{1}{\gamma (1-\gamma)}\int_{x_i = 0}^{\gamma} \int_{x_j = \gamma}^1 f(x_i, x_j)^d dx_j\, dx_i \right)^{\frac{1}{d}} \label{eq:holder}\\
&= 2^c \gamma^m (1-\gamma)^m \left(\frac{1}{\gamma (1-\gamma)}\int_{x = 0}^{\gamma} \int_{y = \gamma}^1 f(x, y)^d dy \,dx \right)^{\frac{md}{d}} \nonumber \\
&= 2^c \left(\int_{x = 0}^{\gamma} \int_{y = \gamma}^1 f(x, y)^d dy\, dx \right)^m = 2^c\left(\frac{1}{2} \Vert f \Vert_d^d \right)^m. \nonumber
\end{align}
In \eqref{eq:alternation} we rewrite the density by ordering the vertices so that they alternate between the sides of the bipartition. The factor $2^c$ accounts for the fact that within each component, a partition class of vertices can map to either $[0, \gamma]$ or $(\gamma, 1]$,  and the other partition class will map to the other interval.  The generalized H\"older inequality from  \Cref{theorem:holder} implies \eqref{eq:holder}. In the application of  \Cref{theorem:holder}, we set $p_i = d$ for every $i \in [2m]$. The  Radon--Nikodym derivatives of the measures are given by
\begin{align*}
\frac{d\mu_{2k+1}}{dx} &= \begin{cases}
\frac{1}{\gamma} & 0 \leq x \leq \gamma\\
0 & \gamma < x \leq 1
\end{cases} \,\,\,\,\, \, \,\,\, \text{ and } \,\,\,\,\, \, \,\,\,
\frac{d\mu_{2k}}{dx} = \begin{cases}
0 & 0 \leq x \leq \gamma\\
\frac{1}{1-\gamma} & \gamma < x \leq 1
\end{cases}
\end{align*}
\noindent for $k \in \{0, 1, \dots, m\}$, i.e. $\mu_{2k+1}$ is uniform on $[0, \gamma]$ and $\mu_{2k}$ is uniform on $(\gamma,1]$. The sets $A_1, \dots, A_{e(H)}$ correspond to the set $E(H)$. We conclude that $\Vert f \Vert_d^d \geq 2 \gamma(1-\gamma) r^d$.
\end{proof}

\begin{proof}[Proof of \Cref{lemma:symmetric}.]
Let $f \in \W_\Omega$ be such that $t(H,f) \geq t(H, f_r^\gamma)$. \Cref{lemma:norm} implies that $\Vert f \Vert_d^d \geq 2 \gamma (1-\gamma) r^d$. It follows by  \Cref{lemma:symmetric-general} that $I_{W_0}(f) \geq I_{W_0}(f_r^\gamma)$ with equality if and only if $f=f_r^\gamma$.
\end{proof}

\subsubsection{Proof for the non-symmetric regime}

\begin{proof}[Proof of \Cref{lemma:non-symmetric}.]
Since $(r^d, h_p(r))$ is not on the convex minorant of $\psi_p(x)= h_p(r^{1/d})$,
we may use \Cref{all-about-psi} to conclude there exist $r_1, r_2, r$ such that $p< r_1 < r< r_2 \leq 1$ and $(r^d, h_p(r))$ lies strictly above the line segment joining $(r_1^d, h_p(r_1))$ and $(r_2^d, h_p(r_2))$.
Let $s \in (0,1)$ be such that $$r^d = s r_1^d + (1-s) r_2^d,$$ and thus
\begin{equation}\label{non-minorant}
s h_p(r_1) +(1-s) h_p(r_2) < h_p(r).
\end{equation}

We use the values $r_1, r_2,$ and $s$ to define a family of graphons $(g^\ve)_{\ve>0}$. We will prove that for $\epsilon >0$ sufficiently small (i) $t(H, g^{\epsilon}) > t(H, f)$ and (ii) $I_{W_0}(g^\ve) < I_{W_0}(f_r^\gamma)$.
Define
\begin{equation}\label{i_defns}
\begin{split}
       \alpha_1& = \gamma s \ve^2\\
    \alpha_2& = (1- \gamma) s \ve^2  \\
    \alpha_3&=(1-\gamma) \brac{(1-s) \ve^2 +\ve^3}\\
    \alpha_4 &= \gamma \brac{ (1-s) \ve^2 +\ve^3}
\end{split}
\quad \quad \quad \quad \quad \quad
 \begin{split}
    I_1&= [0, \alpha_1]\\
    I_2&= (\gamma, \gamma+ \alpha_2]\\
    I_3&=(1- \alpha_3,1]\\
    I_4&=(\gamma-\alpha_4 , \gamma].
    \end{split}
\end{equation}
Let $$I_{c14}= [0, \gamma] \setminus (I_1\cup I_4)\quad \text{ and } \quad I_{c23}= ( \gamma, 1] \setminus (I_2\cup I_3).$$

\noindent
Define
\begin{align*}
g^\ve(x,y) &= \begin{cases}
0 &  (x,y) \in ([0, \gamma)\times [0, \gamma)) \cup ((\gamma,1] \times (\gamma,1])\\
r_1 &  (x,y) \in (I_1 \times I_{c23}) \cup (I_{c23} \times I_1) \cup (I_2 \times I_{c14}) \cup (I_{c14} \times I_2) \\
r_2 &  (x,y) \in  (I_3 \times I_{c14}) \cup (I_{c14} \times I_3) \cup (I_4 \times I_{c23}) \cup (I_{c23} \times I_4)  \\
r & \text{otherwise}.
\end{cases}
\end{align*}
Figure \ref{fig:bipartite-construction} illustrates the construction of the graphon $g^{\epsilon}$.

\tikzset{blockr/.style= {rectangle, draw=black!50, fill=black!20, thick}}
\tikzset{blockr1/.style= {rectangle, draw=black!50, fill=black!5, thick}}
\tikzset{blockr2/.style= {rectangle, draw=black!50, fill=black!35, thick}}
\begin{figure}[h]
\centering
\begin{tikzpicture}
\draw [] (0,0) rectangle node {$0$} (2,-2);
\draw [] (2,-2) rectangle node {$0$} (5,-5);
\draw [blockr] (2,0) rectangle node {$r$} (2.5,-1/3);
\draw [blockr1] (2.5,0) rectangle node {$r_1$} (4.5,-1/3);
\draw [blockr] (4.5,0) rectangle node {$r$} (5,-1/3);
\draw [blockr1] (2,-1/3) rectangle node {$r_1$} (2.5,-5/3);
\draw [blockr] (2.5,-1/3) rectangle node {$r$} (4.5,-5/3);
\draw [blockr2] (4.5,-1/3) rectangle node {$r_2$} (5,-5/3);
\draw [blockr] (2,-5/3) rectangle node {$r$} (2.5,-2);
\draw [blockr2] (2.5,-5/3) rectangle node {$r_2$} (4.5,-2);
\draw [blockr] (4.5,-5/3) rectangle node {$r$} (5,-2);
\draw [blockr] (0,-2) rectangle node {$r$} (1/3, - 2.5);
\draw [blockr1] (1/3,-2) rectangle node {$r_1$} (5/3, - 2.5);
\draw [blockr] (5/3,-2) rectangle node {$r$} (2, - 2.5);
\draw [blockr1] (0,-2.5) rectangle node {$r_1$} (1/3, - 4.5);
\draw [blockr] (1/3,-2.5) rectangle node {$r$} (5/3, - 4.5);
\draw [blockr2] (5/3,-2.5) rectangle node {$r_2$} (2, - 4.5);
\draw [blockr] (0,-4.5) rectangle node {$r$} (1/3, - 5);
\draw [blockr2] (1/3,-4.5) rectangle node {$r_2$} (5/3, - 5);
\draw [blockr] (5/3,-4.5) rectangle node {$r$} (2, - 5);
\draw [|-|] (-0.2,0) -- (-0.2,-2) node[midway,left] {$\gamma$};
\draw [|-|] (-0.2,-2) -- (-0.2,-5) node[midway,left] {$1-\gamma$};
\draw [|-|] (2,0.2) -- (2.5,0.2) node[midway,above] {$\alpha_2$};
\draw [|-|] (4.5,0.2) -- (5,0.2) node[midway,above] {$\alpha_3$};
\draw [|-|] (5.2,0) -- (5.2,-1/3) node[midway,right] {$\alpha_1$};
\draw [|-|] (5.2,-5/3) -- (5.2,-2) node[midway,right] {$\alpha_4$};
\end{tikzpicture}
\caption{Construction of $g^{\epsilon}$.}
\label{fig:bipartite-construction}
\end{figure}
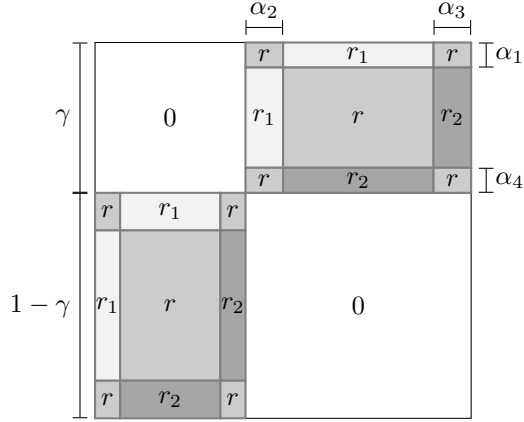

Next we claim that $t(H,g^\ve) > t(H,f_r^\gamma)$ for sufficiently small $\ve$.  Let $m$ be such that $H$ has $2m$ vertices  and $dm$ edges. Let $c$ be the number of connected components of $H$.
Note that the only embeddings of $H$ into $g^\ve$ which contribute a value other than $r^{e(H)}$ to the integral in $t(H, g^\ve)$ are such that at least one vertex of $H$ is mapped to $\bigcup_{j=1}^4 I_j$. Since each $\alpha_i$ is of order $\epsilon^2$, in order to compute $t(H,g^\ve) - t(H,f_r^\gamma)$ up to error $O\brac{\ve^4}$ it suffices to consider embeddings in which only one vertex is mapped to $\bigcup_{j=1}^4 I_j$. Observe

\begin{align*}
    t(H,g^\ve) - t(H,f_r^\gamma)&=2^c  \int_0^\gamma \int_\gamma^1 \cdots  \int_0^\gamma \int_\gamma^1 \brac{\prod_{(i,j) \in E(H)} g^\ve(x_i,y_j) -r^{e(H)}} dx_1 \dots dx_{2m}\\
    &= 2^c m \left[\alpha_1 \gamma^{m-1} (1-\gamma)^m (r_1^d -r^d) r^{e(H)-d} + \alpha_2(1-\gamma)^{m-1} \gamma^{m} (r_1^d - r^d) r^{e(H)-d} \right.\\
    &\quad \quad \left.+ \alpha_3 (1-\gamma)^{m-1} \gamma^m (r_2^d - r^d) r^{e(H)-d} +\alpha_4 \gamma^{m-1} (1-\gamma)^m (r_2^d -r^d) r^{e(H)-d}\right] +O\brac{\ve^4}\\
    &= 2^{c+1}m \gamma^m (1-\gamma)^m r^{e(H)-d} \brac{s\ve^2 (r_1^d -r^d) + \brac{(1-s)\ve^2+\ve^3} \brac{r_2^d-r^d} } + O \brac{\ve^4}\\
    &= 2^{c+1}m \gamma^m (1-\gamma)^m r^{e(H)-d}  \brac{r_2^d-r^d}\ve^3  +O\brac{\ve^4}.
\end{align*}
Since $r_2 > r$, the above computation implies that $ t(H,g^\ve) - t(H,f_r^\gamma) >0$ for $\ve$ sufficiently small.

Next we show that $I_{W_0}(g^\ve) < I_{W_0}(f_r^\gamma)$ for sufficiently small $\ve$. Observe
\begin{align*}
    I_{W_0}(g^\ve) - I_{W_0}(f_r^\gamma)&= \brac{\alpha_1 \brac{1-\gamma-\alpha_2 -\alpha_3} +\alpha_2 \brac{ \gamma -\alpha_1 -\alpha_4}} \brac{ h_p(r_1) -h_p(r)}\\
    &\quad + \brac{ \alpha_3 \brac{\gamma-\alpha_1 -\alpha_4} + \alpha_4 \brac{1-\gamma-\alpha_2-\alpha_3}}\brac{ h_p(r_2) -h_p(r)}\\
    &=2 \gamma (1-\gamma)\brac{1-\ve^2-\ve^3 }\left[ s\ve^2 \brac{ h_p(r_1) -h_p(r)} +\brac{(1-s)\ve^2+\ve^3} \brac{ h_p(r_2) -h_p(r)}\right]\\
    &=2 \gamma (1-\gamma)\brac{1-\ve^2-\ve^3 }\ve^2\left[ s h_p(r_1)+(1-s)h_p(r_2) -h_p(r)+\ve \brac{ h_p(r_2) -h_p(r)}\right].
    \end{align*}
Using the condition \eqref{non-minorant}, we conclude that there exists $\ve$ sufficiently small such that $I_{W_0}(g^\ve) - I_{W_0}(f_r^\gamma)<0$, as desired.
\end{proof}

\subsection{Largest eigenvalue}
In this subsection we prove \Cref{theorem:eigenvalue}, which characterizes the symmetric and non-symmetric regimes for the largest eigenvalue of the adjacency matrix of a bipartite graph. Recall that $\Vert \cdot \Vert_{\text{op}}$ is a continuous extension of the normalized graph spectral norm (Lemma \ref{lemma:continuous-extension}). Note that $\Vert f_r^{\gamma} \Vert_{\text{op}} = r \sqrt{\gamma(1-\gamma)}$.
We will use the following two lemmas to prove \Cref{theorem:eigenvalue}.

\begin{lemma}\label{lemma:norm-inequalities} Let $\gamma\in (0,1) \cap \mathbb{Q}$ and let $W_0 = f_p^{\gamma}$. For every $f$ such that $f \in \W_\Omega$, we have $\Vert f \Vert_1 \leq \Vert f \Vert_{\emph{op}} \leq \frac{1}{\sqrt{2}} \Vert f \Vert_2$.
\end{lemma}

\begin{lemma}\label{lemma:improving-eigenvalue}
Let $0 < p\leq r < 1$ be such that $(r^2, h_p(r))$ does not lie on the convex minorant of $x \mapsto h_p(\sqrt{x})$. Then there exists some $g \in \mathcal{W}_\Omega$ with $\Vert g \Vert_{\emph{op}} >  r \sqrt{\gamma (1-\gamma)}$ and $I_{W_0}(g) < I_{W_0}(f_r^\gamma)$.
\end{lemma}

\begin{proof}[Proof of Theorem \ref{theorem:eigenvalue}]
Let $\psi_p(x)=h_p(\sqrt{x})$.
Suppose that the point $(r^2, h_p(r))$ lies on the convex minorant of $\psi_p$. We will show that $t_r^\gamma= \Vert f_r^\gamma\Vert_\op$ is in the symmetric regime for $t(H, \cdot)$.
Let $\ti{g} \in \tW_\Omega$ be such that $J_{W_0}(\ti{g})= \min \{J_{W_0}(\ti{g}): \Vert g \Vert_{\op} \geq t_r^\gamma\}$. We may assume that $g \in \W_\Omega.$ By \Cref{nice-i-min}, $J_{W_0}(\ti{g})= I_{W_0}(g)$. Since $\Vert g \Vert_{\op} \geq t_r^\gamma=r \sqrt{ \gamma (1-\gamma)}$, \Cref{lemma:norm-inequalities} implies that $\Vert g \Vert_{2} \geq r \sqrt{ 2\gamma (1-\gamma)}.$ Next, by
\Cref{lemma:symmetric-general}, we have $I_{W_0}(g) \geq I_{W_0}(f_r^\gamma)$ with equality if and only if $g = f_r^\gamma$. Since $I_{W_0}(g)=J_{W_0}(\ti{g})= \min \{J_{W_0}(\ti{g}): t(H,g) \geq t_r^\gamma\} \leq I_{W_0}(f_r^\gamma)$, it follows that $I_{W_0}(f_r^\gamma)= I_{W_0}(g)$.  By \Cref{lemma:symmetric-general}, we conclude that $g = f_r^\gamma$, meaning that $\ti{f_r^\gamma}$ is the unique symmetric solution.

Next, suppose that the point $(r^2, h_p(r))$ does not lie on the convex minorant of $\psi_p$. We will show that $t_r^\gamma= \Vert f_r^\gamma\Vert_\op$ is not in the symmetric regime for $t(H, \cdot)$. \Cref{lemma:improving-eigenvalue} implies that there exists $g \in \W_\Omega$ such that $\Vert g \Vert_{\op} >  r \sqrt{\gamma (1-\gamma)}$ and $I_{W_0}(g) < I_{W_0}(f_r^\gamma)$.  By \Cref{nice-i-min}, $J_{W_0}(\ti{g})= I_{W_0}(g)$. We apply \Cref{nice-i-min} and obtain
\begin{align}\label{min comp}
    \min\{J_{W_0}(\ti{g}) :\Vert g\Vert_\op \geq t_r^\gamma, \ti{g} \in \ti{\mathcal{B}}^\gamma\}&=   \min\{I_{W_0}(g) :\Vert g\Vert_\op \geq t_r^\gamma, g \in \mathcal{B}^{(\gamma, 1-\gamma)} \cup \mathcal{B}^{( 1-\gamma, \gamma)} \}\nonumber \\
    &=   \min\{I_{W_0}(f_q^\gamma) : q \in [0,1], \Vert f_q^\gamma\Vert_\op \geq t_r^\gamma\}=I_{W_0}(f_r^\gamma)\nonumber\\
    &> I_{W_0}(g)\geq \min \{ J_{W_0}(\ti{g}): \Vert g\Vert_\op \geq t_r^\gamma\}.
\end{align}
The second equality follows by noting that if $\gamma\not= 1/2$ and $g \in \mathcal{B}^{( 1-\gamma, \gamma)}$, then $I_{W_0}(g)=\infty$ or $g$ is the zero graphon. The third equality follows by noting that $I_{W_0}(f_q^\gamma)$ and $\Vert f_q^\gamma\Vert_\op$ are increasing functions of $q$.

\noindent
It follows by \eqref{min comp}  that any minimizer $\ti{g} \not \in \ti{\mathcal{B}}^{(\gamma, 1-\gamma)}$, and so the problem is not in the symmetric regime.
\end{proof}

\begin{proof}[Proof of \Cref{lemma:norm-inequalities}.]
As stated in \cite{Lubetzky2015}, the left inequality follows from the observation that
$$\Vert f\Vert_1 = \Vert T_f \boldsymbol{1} \Vert_1\leq \Vert T_f \boldsymbol{1} \Vert_2 \leq \Vert f \Vert_{\op}.$$
To derive the upper bound, we use the \CS inequality. Observe that for any $u: [0,1] \to \mathbb{R}$,
\begin{align*}
\Vert T_f u\Vert_2^2 &= \int_0^1 \brac{ \int_0^1 f(x,y) u(y) dy}^2 dx\\
&= \int_0^\gamma \brac{ \int_\gamma^1 f(x,y) u(y) dy}^2 dx+\int_\gamma^1 \brac{ \int_0^\gamma f(x,y) u(y) dy}^2 dx\\
&\leq \int_\gamma^1 u(y)^2 dy \brac{ \int_0^\gamma \int_\gamma^1 f(x,y)^2 dy dx}+\int_0^\gamma u(y)^2 dy \brac{ \int_\gamma^1 \int_0^\gamma f(x,y)^2 dy dx} \\
&= \frac{1}{2} \Vert u \Vert_2^2 \Vert f \Vert_2^2.
\end{align*}
It follows that $\Vert T_f u\Vert_2 \leq \frac{1}{\sqrt{2}}\Vert f \Vert_2 \Vert u \Vert_2$ for all $u$, and thus $\Vert f \Vert_{\op} \leq \frac{1}{\sqrt{2}} \Vert f \Vert_2$.
\end{proof}

\begin{proof}[Proof of \Cref{lemma:improving-eigenvalue}]
Let $g^{\epsilon}$ be as defined in the proof of Lemma \ref{lemma:non-symmetric}. We have already shown that
$I_{W_0}(g^{\epsilon}) < I_{W_0}(f_r^\gamma)$ for small enough $\epsilon > 0$. It remains to show that $\Vert g^{\epsilon} \Vert_{\text{op}} > r\sqrt{ \gamma (1-\gamma)}$.

 For this claim, it suffices to exhibit a function $u \in L^2([0,1])$ such that $(T_{g^{\epsilon}} u)(x) > r \sqrt{\gamma (1-\gamma)} u(x)$ for all $x \in [0,1]$. Recall the definitions given in \eqref{i_defns}. Let
$$
u(x)=
\begin{cases}
\frac{\sqrt{1-\gamma}}{\gamma}(\gamma- \alpha_1 - \alpha_4) r_1 & x \in I_1\\
r \sqrt{1-\gamma}  & x \in I_{c14}\\
\frac{\sqrt{1-\gamma}}{\gamma}(\gamma- \alpha_1 - \alpha_4) r_2 & x \in I_4\\
\frac{\sqrt{\gamma}}{1-\gamma}(1 - \gamma - \alpha_2 - \alpha_3) r_1 & x \in I_2\\
r \sqrt{\gamma}  & x \in I_{c23}\\
\frac{\sqrt{\gamma}}{1-\gamma}(1 - \gamma- \alpha_2 - \alpha_3) r_2 & x \in I_3.\end{cases}
$$
Recall that
\[1 - \epsilon^2 - \epsilon^3 = \frac{1}{\gamma}(\gamma - \alpha_1 - \alpha_4) = \frac{1}{1-\gamma}(1-\gamma - \alpha_2 - \alpha_3). \]

\noindent
We consider six cases, and assume $\ve$ is sufficiently small in each. For $x \in I_1$,
\begin{align*}
T_{g^{\epsilon}}u(x) &= \int_{0}^1 g^{\epsilon}(x,y) u(y) dy\\
&= \alpha_2 r \cdot \frac{\sqrt{\gamma}}{1-\gamma} (1-\gamma - \alpha_2 - \alpha_3) r_1 + (1-\gamma - \alpha_2 - \alpha_3) r_1 \cdot \sqrt{\gamma} r + \alpha_3 r \cdot \frac{\sqrt{\gamma}}{1-\gamma} (1-\gamma - \alpha_2 -\alpha_3) r_2\\
&> (1-\gamma - \alpha_2 - \alpha_3) r_1 \sqrt{\gamma} r
=\frac{1-\gamma}{\gamma} (\gamma - \alpha_1 - \alpha_4)\sqrt{\gamma} r_1 r\\
&= r \sqrt{\gamma(1-\gamma)} \frac{\sqrt{1-\gamma}}{\gamma} (\gamma - \alpha_1 - \alpha_4)r_1= r \sqrt{\gamma(1-\gamma)} u(x).
\end{align*}

For $x \in I_4$,
\begin{align*}
T_{g^{\epsilon}}u(x) &> (1-\gamma - \alpha_2 - \alpha_3) r_2 \sqrt{\gamma} r
=\frac{1-\gamma}{\gamma}(\gamma - \alpha_1 -\alpha_4) r_2 \sqrt{\gamma} r\\
&= r \sqrt{\gamma(1-\gamma)} \frac{\sqrt{1-\gamma}}{\gamma} (\gamma - \alpha_1 - \alpha_4) r_2
= r \sqrt{\gamma(1-\gamma)} u(x).
\end{align*}

For $x \in I_2$,
\begin{align*}
T_{g^{\epsilon}}u(x) &> (\gamma - \alpha_1 - \alpha_4) r_1\sqrt{1-\gamma} r
= \frac{\gamma}{1-\gamma} (1-\gamma - \alpha_2 - \alpha_3) r_1 \sqrt{1-\gamma}  r\\
&=r \sqrt{\gamma(1-\gamma)}\frac{\sqrt{\gamma}}{1-\gamma} (1-\gamma - \alpha_2 -\alpha_3)r_1
= r\sqrt{\gamma(1-\gamma)} u(x).
\end{align*}

For $x \in I_3$,
\begin{align*}
T_{g^{\epsilon}}u(x) &> (\gamma - \alpha_1 - \alpha_4) r_2 \sqrt{1-\gamma} r
= \frac{\gamma}{1-\gamma} (1-\gamma - \alpha_2 -\alpha_3) r_2\sqrt{1-\gamma}r \\
&= r \sqrt{\gamma(1-\gamma)} \frac{\sqrt{\gamma}}{1-\gamma}(1-\gamma - \alpha_2 -\alpha_3) r_2
= r \sqrt{\gamma(1-\gamma)} u(x).
\end{align*}

Next consider $x \in I_{c14}$. Using the fact that $r^2 = sr_1^2 + (1-s)r_2^2$, along with $\alpha_2 = (1-\gamma) s \epsilon^2$ and \\$\alpha_3 = (1-\gamma)\left((1-s)\epsilon^2 + \epsilon^3\right)$, we obtain
\begin{align*}
T_{g^{\epsilon}}u(x) &= \left(1 - \gamma - \alpha_2 - \alpha_3\right) \left(\alpha_2 \frac{\sqrt{\gamma}}{1-\gamma} r_1^2 + \sqrt{\gamma} r^2 + \alpha_3 \frac{\sqrt{\gamma}}{1-\gamma} r_2^2\right)\\
&= (1-\gamma)(1-\epsilon^2 - \epsilon^3) \left( \frac{\sqrt{\gamma}}{1-\gamma} \left(\alpha_2 r_1^2 + \alpha_3 r_2^2 \right) + \sqrt{\gamma} r^2\right)\\
&= (1-\gamma)(1-\epsilon^2 - \epsilon^3) \left( \frac{\sqrt{\gamma}}{1-\gamma} \left((1-\gamma) s \epsilon^2 r_1^2 + (1-\gamma) \left((1-s)\epsilon^2 + \epsilon^3\right) r_2^2 \right) + \sqrt{\gamma} r^2\right)\\
&= \sqrt{\gamma} (1-\gamma)(1-\epsilon^2 - \epsilon^3) \left( s \epsilon^2 r_1^2 + \left((1-s)\epsilon^2 + \epsilon^3\right) r_2^2 + r^2\right)\\
&= \sqrt{\gamma} (1-\gamma)(1-\epsilon^2 - \epsilon^3) \left(\epsilon^2 r^2 + \epsilon^3 r_2^2 + r^2\right)
= \sqrt{\gamma} (1-\gamma)\left(r^2 + \left(r_2^2 - r^2\right)\epsilon^3 + O(\epsilon^4) \right)\\
&> \sqrt{\gamma}(1-\gamma) r^2
= r \sqrt{\gamma(1-\gamma)} \sqrt{1-\gamma} r
= r \sqrt{\gamma(1-\gamma)} u(x),
\end{align*}
where the inequality holds for sufficiently small $\epsilon > 0$ since $r_2 > r$.

Similarly, for $x \in I_{c23}$,
\begin{align*}
T_{g^{\epsilon}}u(x) &= \left(\gamma - \alpha_1 - \alpha_4\right) \left(\alpha_1 \frac{\sqrt{1-\gamma}}{\gamma} r_1^2 + \sqrt{1-\gamma}r^2 + \alpha_4 \frac{\sqrt{1-\gamma}}{\gamma} r_2^2 \right)\\
&= \gamma (1-\epsilon^2 - \epsilon^3) \left(\frac{\sqrt{1-\gamma}}{\gamma} \left( \alpha_1 r_1^2 + \alpha_4 r_2^2\right) + \sqrt{1-\gamma} r^2 \right)\\
&= \gamma (1-\epsilon^2 - \epsilon^3) \left(\frac{\sqrt{1-\gamma}}{\gamma} \left( \gamma s \epsilon^2 r_1^2 + \gamma \left((1-s)\epsilon^2 + \epsilon^3\right) r_2^2\right) + \sqrt{1-\gamma} r^2 \right)\\
&= \gamma \sqrt{1-\gamma} (1-\epsilon^2 - \epsilon^3) \left(s \epsilon^2 r_1^2 + \left((1-s)\epsilon^2 + \epsilon^3\right) r_2^2 + r^2 \right)\\
&= \gamma \sqrt{1-\gamma} (1-\epsilon^2 - \epsilon^3) \left( \epsilon^2 r^2 + \epsilon^3 r_2^2 + r^2\right)
= \gamma \sqrt{1-\gamma} \left(r^2 + (r_2^2 - r^2) \epsilon^3 + O(\epsilon^4)\right)\\
&> \gamma \sqrt{1-\gamma} r^2
= r \sqrt{\gamma(1-\gamma)} \sqrt{\gamma} r
= r\sqrt{\gamma (1-\gamma)} u(x),
\end{align*}
where again the inequality holds for sufficiently small $\epsilon > 0$.

We have shown that $T_{g^{\epsilon}}u(x) > r \sqrt{\gamma(1-\gamma)} u(x)$ for all $x \in [0,1]$, for $\epsilon > 0$ sufficiently small. Therefore, there exists $\epsilon > 0$ such that $\Vert g^{\epsilon} \Vert_{\text{op}} > r \sqrt{\gamma(1-\gamma)}$.
\end{proof}

\section{Open questions}
\label{sec:open}
We collect here some questions arising naturally from our investigations.

\begin{enumerate}
    \item Our results establish a ``reentrant phase transition'' in upper tail large deviations for homomorphism densities in specific block model random graphs. Note that our results on the symmetric regime are quite general, and applicable for arbitrary block graphons. In contrast, our proof for the existence of a symmetry breaking regime is case-specific, and does not generalize directly. It is natural to believe that this reentrant phase transition phenomenon should hold for a much wider family of block graphons, and it would be  interesting to investigate this further.
    \item A natural follow up question concerns the precise boundary between the symmetric and non-symmetric regimes. So far, this boundary has been identified for very homogeneous graphs---the Erd\H{o}s-R\'{e}nyi random graph in \cite{Lubetzky2015} and the Erd\H{o}s-R\'{e}nyi bipartite graph in this article.  We expect the general case to be significantly more challenging and is beyond the scope of this paper.
    \item Theorems \ref{sym-regime} and \ref{sym-regime-near-one} together identify a symmetric regime for the homomorphism density of a regular subgraph. In this regime, is the solution to the variational problem unique?
    \item Another natural direction of inquiry concerns the behavior of the minimizer(s) in the symmetry breaking regime. In fact, we do not even know whether the upper tail variational problem \eqref{min_eq} has a unique minimizer in the symmetry breaking phase. Any tangible progress on this uniqueness question would be a promising start in this direction. Moreover, it would be of interest to identify the structure of the minimizer(s) in the non-symmetric regime. These questions remain open even for \ER graphs, and were already raised in \cite{Chatterjee2011} and \cite{Lubetzky2015}.
    \item Finally, we note that our analysis of the upper tail variational problem \eqref{min_eq} is restricted to regular subgraphs. Non-trivial extensions to non-regular graphs will likely require new ideas, and will provide new insights on the upper tail problem.
\end{enumerate}

\section{Appendix}
\subsection{Weak topology LDP upper bound}
To prove the upper bound LDP in the weak topology, \Cref{weak ldp},
we will use a general LDP upper bound given in \cite[Section 4.3]{Chatterjee2015} which we restate as
 \Cref{theorem:chaterjee} below.

We need some notation. Let $\mathscr{H}$ be a real topological vector space whose topology satisfies the Hausdorff property. Let $\mathscr{H}^*$ denote the dual space of continuous linear functionals on $\mathscr{H}$. Let $\mathscr{B}$ denote the Borel sigma-algebra of $\mathscr{H}$ and let $\{\mu_n\}_{n \geq 1}$ be a sequence of probability measures on $(\mathscr{H}, \mathscr{B})$.
Define the logarithmic moment generating function $\Lambda_n: \mathscr{H}^* \to (-\infty, \infty]$ of $\mu_n$ as
$$\Lambda_n(\lambda) = \log \int_{\mathscr{H}} \exp{ \lambda(x)} d \mu_n (x).$$
Given a ``rate''  $\{ \ve_n\}_{n\geq 1}$, i.e.,  a sequence of positive real numbers $\ve_n$ tending to $0$, we define
 $\bar{\Lambda} : \mathscr{H}^* \to [-\infty, \infty]$ and its Fenchel-Legendre transform $\bar{\Lambda}^* : \mathscr{H} \to [-\infty, \infty]$ as
 $$\bar{\Lambda}(\lambda)= \limsup_{n \to \infty} \ve_n \Lambda_n ( \lambda / \ve_n)$$
  $$\bar{\Lambda}^*(x)= \sup_{\lambda \in \mathscr{H}^*} ( \lambda(x)- \bar{\Lambda}(\lambda)).$$

  \begin{lemma}[Theorem 4.1 of \cite{Chatterjee2015}]\label{theorem:chaterjee}
  For any compact set $\Gamma \subseteq  \mathscr{H}$,
   $$\limsup_{n \to \infty} \ve_n \log \mu_n (\Gamma) \leq - \inf_{ x \in \Gamma} \bar{\Lambda}^*(x).$$
  \end{lemma}

\begin{proof}[Proof of \Cref{weak ldp}]
We closely follow the proof of Theorem 5.1 in \cite{Chatterjee2015}. Let $\mathscr{H}$ be the vector space $L^2([0,1]^2)$ with the weak topology. For each $a,f \in \mathscr{H}$, define $\lambda_a$ in the dual space $\mathscr{H}^*$ as
$$ \lambda_a(f)= \int_{[0,1]^2} a(x,y) f(x,y) dx dy.$$ 
Setting $\mu_n = \mathbb{P}_{kn, W_0}$, define $\Lambda_n: \mathscr{H}^* \to \R$  as $$\Lambda_n(\lambda) = \log \int_{\mathscr{H}} \exp{ \lambda(f)} d \pr_{kn, W_0}(f) = \log \left(\mathbb{E}_{f \sim \mathbb{P}_{kn, W_0}}\left[ \exp{\lambda(f)} \right]\right).$$
Set $\epsilon_n = \frac{2}{(kn)^2}$, and let $$ \bar{\Lambda}(\lambda) =  \limsup_{n \to \infty} \frac{2 \Lambda_n ((kn)^2 \lambda/2)}{(kn)^2}.$$

Let $f^G$ be the empirical graphon on $kn$ vertices drawn from $W_0$. For $i,j \in [kn]$, let $X_{ij}$ be the indicator for the event that $\{i,j\}$ is an edge in $G$. Since $G$ is a simple undirected graph $X_{ij}=X_{ji}$ and $X_{ii}=0$. For ease of notation let $W_0^{ij} = W_0(i/(kn), j/(kn))$. Note $X_{ij} \sim Bern(W_0^{ij})$. Let
$I_{kn,1}=[0,1/kn]$, $I_{kn,i}=((i-1)/kn,i/kn]$ for $i=2,\dots, kn$ and let
 $B_{i,j, n}$ be the square $I_{kn,i}\times I_{kn,j}$.
 Let $S$ be the set of symmetric $L^2$ functions. For $a \in S$, let $\hat{a}_{n}$ denote the level $kn$ approximant, i.e.
$$\hat{a}_n(x,y)= (kn)^2 \int_{B}a(w,z) dw \, dz$$
where $B= B_{i,j,n}$ is such that $(x,y) \in B_{i,j,n}.$

Observe that for an empirical graphon $f^G$,
$$ \lambda_a(f^G) = \sum_{1 \leq i, j \leq kn, i \neq j}  X_{ij} \int_{B(i,j,n)} a(x,y) dx\, dy= \sum_{1 \leq i < j \leq kn} X_{ij} \int_{B(i,j,n) \cup B(j,i,n)}a(x,y) dx\, dy.$$
Recall $X_{ij}  \sim Bern(W_0^{ij})$, and so for any $\theta$
$$\mathbb{E}\left[ \exp{ \theta X_{ij}}\right]= W_0^{ij}\exp{\theta} + 1- W_0^{ij}.$$
 Since the events $\{X_{ij}\}_{i < j}$ are independent, it follows that for any $a \in S$
\begin{align}
    \Lambda_n((kn)^2 \lambda_a/2)&=\log \left(\mathbb{E}_{f \sim \mathbb{P}_{kn, W_0}}\left[\exp{ (kn)^2 \lambda_a(f)/2} \right]\right) \nonumber\\
    &=\log \left(\mathbb{E}_{f \sim \mathbb{P}_{kn, W_0}}\left[\exp{ \frac{(kn)^2}{2} \sum_{1 \leq i < j \leq kn} X_{ij} \int_{B(i,j,n) \cup B(j,i,n)}a(x,y) dx dy} \right]\right) \nonumber\\
    &= \log \prod_{1 \leq i < j \leq kn} \brac{W_0^{ij}\exp{ \frac{(kn)^2}{2} \int_{B(i,j,n) \cup B(j,i,n)} a(x,y) dx dy} +1 -  W_0^{ij}}\nonumber \\
    &= \sum_{1 \leq i < j \leq kn} \log  \brac{W_0^{ij}\exp{ \frac{(kn)^2}{2} \int_{B(i,j,n) \cup B(j,i,n)} a(x,y) dx dy} +1 -  W_0^{ij}}\nonumber \\
    &=(kn)^2 \sum_{1\leq i < j \leq kn} \int_{B(i,j,n) }\log\brac{W_0^{ij} \exp{\hat{a}_n(x,y)}+1- W_0^{ij}}dx dy\nonumber \\
    &= \frac{(kn)^2}{2} \int_{[0,1]^2 \setminus B_n} \log\brac{ W_0(x,y) \exp{\hat{a}_n (x,y)}+1- W_0(x,y)} dx dy \label{Lambda simple}
\end{align}
where $B_n \triangleq \bigcup_{i=1}^n B(i,i,n)$.

Next we consider $u_p(x)\triangleq \log \brac{ p e^x+1-p},$ in order to reason about the limit of the above integral as $n\to \infty$. For $p \in [0,1]$, $u_p'(x)= pe^x/(pe^x+1-p)$, and so $|u_p'(x)| \leq 1$ everywhere. Thus $|u_p(x)-u_p(y)| \leq |x-y|$ for all $x,y \in \R$. It follows that
\begin{align*}
|\log &\brac{ W_0(x,y) \exp{ \hat{a}_n(x,y)}+1-W_0(x,y)} -\log \brac{ W_0(x,y) \exp{a(x,y)}+1-W_0(x,y)}| \\
&\leq |\hat{a}_n(x,y) - a(x,y)|.
\end{align*}
By Proposition 2.6 of \cite{Chatterjee2015}, $\hat{a}_n \to a$ in $L^2$, and therefore using the above inequality and the \CS inequality, we obtain
\begin{align*}
&\int_{[0,1]^2} \log  \brac{ W_0(x,y) \exp{ \hat{a}_n(x,y)}+1-W_0(x,y)} - \left( \log \brac{ W_0(x,y) \exp{a(x,y)}+1-W_0(x,y)}\right) dx dy\\
&\leq \int_{[0,1]^2} |\hat{a}_n(x,y) - a(x,y)| dx dy\\
&\leq  \left(\int_{[0,1]^2}  \left(\hat{a}_n(x,y) - a(x,y)\right)^2 dx dy\right) ^\frac{1}{2} \to 0.
\end{align*}
Therefore,
\begin{align}
    &\lim_{n \to \infty} \int_{[0,1]^2} \log  \brac{ W_0(x,y) \exp{ \hat{a}_n(x,y)}+1-W_0(x,y)} dx \,dy \nonumber\\
    &=  \int_{[0,1]^2} \log \brac{ W_0(x,y) \exp{a(x,y)}+1-W_0(x,y)} dx \,dy. \label{ahat limit}
\end{align}

Next we consider the limit of the above integral over the set $B_n$.
Since $|u_p(x)-u_p(y)| \leq |x-y|$, taking $y=0$, we obtain $|u_p(x)| \leq |x|$. The \CS inequality then implies that
\begin{align*}
    &\bigg| \int_{B_n} \log \brac{ W_0(x,y) \exp{ \hat{a}_n(x,y)}+1-W_0(x,y)} dx dy \bigg|  \leq \bigg| \int_{B_n} |\hat{a}_n(x,y)| dx dy \bigg| \\
    &\quad \quad \leq \brac{ |B_n| \int \hat{a}_n(x,y)^2 dx dy }^{1/2} \leq \frac{\Vert\hat{a}_n\Vert_2 }{\sqrt{kn}}.
\end{align*}
Since $\hat{a}_n\to a$ in $L^2$,
\begin{align}
  \lim_{n \to \infty}\int_{B_n} \log \brac{ W_0(x,y) \exp{ \hat{a}_n(x,y)}+1-W_0(x,y)} dx \,dy=0.  \label{bn limit}
\end{align}
Finally we use \eqref{Lambda simple}, \eqref{bn limit}, and \eqref{ahat limit} to compute for $a \in S$
\begin{align*}
    \overline{\Lambda}(\lambda_a)&= \limsup_{n \to \infty} \frac{2\Lambda_n\brac{ (kn)^2 \lambda_a/2}}{(kn)^2}\\
    &=\limsup_{n \to \infty}\int_{[0,1]^2 \setminus B_n} \log\brac{ W_0(x,y) \exp{\hat{a}_n (x,y)}+1- W_0(x,y)} dx \,dy\\
    &=\limsup_{n \to \infty}\int_{[0,1]^2} \log\brac{ W_0(x,y) \exp{\hat{a}_n (x,y)}+1- W_0(x,y)} dx \,dy\\
    &=\int_{[0,1]^2} \log\brac{ W_0(x,y) \exp{a (x,y)}+1- W_0(x,y)} dx\, dy.
\end{align*}

For $f \in \mathscr{H}$, let   $$\bar{\Lambda}^*(x)\triangleq \sup_{\lambda \in \mathscr{H}^*} ( \lambda(f)- \bar{\Lambda}(\lambda)).$$
By \Cref{proposition:rate-function-equivalent},
 $$\bar{\Lambda}^*(f) \geq \sup_{a \in S} ( \lambda_a(f)- \bar{\Lambda}(\lambda_a))=2 I_{W_0}(f).$$
Combined with the compactness of the weak topology \cite[Proposition 2.8]{Chatterjee2015}, \cite[Theorem 4.1]{Chatterjee2015}, stated here as \Cref{theorem:chaterjee}, implies that
$$\limsup_{n \to \infty } \frac{2}{(kn)^2}  \log \pr_{kn, W_0}\brac{ F} \leq - \inf_{f \in F} 2I_{W_0}(f).$$
\end{proof}

\subsection{Other useful results}\label{other-useful-results}
The following theorem appears in \cite{Finner1992}. We include the notation given by \cite{Lubetzky2015}.

\begin{theorem}\label{theorem:holder}
Let $\mu_1, \dots, \mu_n$ be probability measures on $\Omega_1, \dots, \Omega_n$, respectively, and let $\mu = \prod_{i=1}^n \mu_i$ be the product measure on $\Omega = \prod_{i=1}^n \Omega_i$. Let $A_1, \dots, A_m$ be nonempty subsets of $[n] = \{1, \dots, n\}$ and write
$\Omega_A = \prod_{l \in A} \Omega_l$ and $\mu_A = \prod_{l \in A} \mu_l$. Let $f_i \in L^{p_i}(\Omega_{A_i}, \mu_{A_i})$ with $p_i \geq 1$ for each $i \in [m]$ and suppose in addition that $\sum_{i : l \in A_i} \frac{1}{p_i} \leq 1$ for each $l \in [n]$. Then
\[\int \prod_{i=1}^m |f_i| d \mu \leq \prod_{i=1}^m \left( \int |f_i|^{p_i} d \mu_{A_i} \right)^{\frac{1}{p_i}}. \]
Assume without loss of generality that the sets $\{k : i \in A_k\}$ are distinct for each $i \in [n]$, and $\int | f_j|^{p_j} d \mu_{A_j} >0$ for all $j \in [m]$. Then 
equality holds if and only if there exist functions $f_{ji}$ on $\Omega_i$ and constants $\alpha_{ji} >0$ such that
\begin{enumerate}
    \item for all $j \in [m]$, $|f_j| = \prod_{i \in A_j} |f_{ji}|$ almost everywhere with respect to $\mu_{A_j}$
    \item for all $i \in [n]$, $r,s \in \{k : i \in A_k\}$, and $r \not = s$,  $\alpha_{ri} |f_{ri}|^{p_r}= \alpha_{si} |f_{si}|^{p_s}$ almost everywhere with respect to $\mu_i$.
\end{enumerate}
If $p_i = d$ for every $i \in [m]$, the inequality reduces to 
\[\int \prod_{i=1}^m |f_i| d \mu \leq \prod_{i=1}^m \left( \int |f_i|^{d} d \mu_{A_i} \right)^{\frac{1}{d}}. \]
\end{theorem}

\begin{proof}[Proof of \Cref{in thm 2.3}]
	The proof is very similar to that of Theorem 2.3 \cite{Chatterjee2011}, and we only sketch it here.
	Since both $f_n$ and $g_n$ are block graphons, the distance $d_\square(f_n, g_n)$ can be written as a maximum
	over $4^{kn}$ pairs of sets $S,T\subset [0,1]$, that are unions of a subset of the intervals used in the definition of $g_n$.  But given
	$S$ and $T$, the expectation of $\int_{S\times T}f_n$ is  $\int_{S\times T}g_n$.  Azuma's inequality then shows that the probability
	that the difference is larger than $\ve$ is bounded by $e^{-c\ve^2 (kn)^2}$ for some universal constant $c>0$.  The union bound now
	implies the proposition.
\end{proof}

\begin{proposition}[Lemma A.1 of \cite{Lubetzky2015}]\label{all-about-psi}
 Let $d\geq 1$ and $p \in (0,1)$. Consider $\psi_p(x) = h_p\left(x^{1/d}\right)$  with domain
 $[0,1]$.
 \begin{enumerate}
     \item The function  $\psi_p(x)$ is decreasing on $[0,p^d]$ and increasing on $[p^d, 1]$.
     \item Let
     \[p_0= \frac{d-1}{d-1 + e^{\frac{d}{d-1}}}.\]
    \begin{enumerate}
        \item If $p>p_0$, then $\psi_p(x)$ is convex and $\psi_p''>0$ on $[0,1]$.
        \item If $p=p_0$, then $\psi_p(x)$ is convex and $\psi_p''=0$ at exactly one point in $(p^d,1)$.
     \item For $p <p_0$, the function  $\psi_p(x)$ has exactly two inflection points $r_1^d$ and $r_2^d$ with $p< r_1 < r_2\leq 1$.  The function $\psi_p$ is convex on $[0,r_1^d] \cup [r_2^d,1]$  and concave on $[r_1^d,r_2^d]$. Moreover $\psi_p''$ is strictly positive on $[0,r_1^d) \cup (r_2^d,1]$. 
    \end{enumerate}
 \end{enumerate}
\end{proposition}

\subsection{Behavior at $t = t_{\max}$.}

\subsubsection{Proof of \Cref{thm:tmax_behavior}}
In this subsection we prove Theorem~\ref{thm:tmax_behavior}.
\begin{proof}[Proof of Theorem~\ref{thm:tmax_behavior}]
First, observe $t(H, \ti{f}_{\max}) = t_{\max}$. Consider the optimization problem
\begin{align}
\min\{J_{W_0}(\ti{f}): t(H, \ti{f} ) \geq t_{\max}\}.  
\label{eq:variational-problem-max}
\end{align}
Let $\ti{f}_*$ be a minimizer of \eqref{eq:variational-problem-max}. \Cref{nice-seq} states that there exists a sequence $f_n \in \mathcal{W}_{\Omega}$ such that $t(H, f_n) \to t_{\max}$, $f_n \geq W_0$ pointwise, $f_n = W_0$ on irrelevant blocks, $I_{W_0}(f_n) \to J_{W_0}(\ti{f}_*)$, and $\delta_{\square}(f_n, \ti{f}_*) \to 0$. For $n \geq n_0$, the $f_n$ satisfy the 
$\varepsilon'$-neighborhood minorant condition for some $\varepsilon' > 0$
(this follows from \Cref{big-delta-cmc}). It suffices to prove the claim for $\varepsilon$ satisfying $0 < \varepsilon \leq \varepsilon'$; we therefore assume this inequality.
\Cref{sym-condition} implies that $\ti{f}_* \in \ti{\mathcal{B}}^{\gamma}$, and $J_{W_0}(\ti{f}_*) = I_{W_0}(g) < \infty$ for some $g \in \mathcal{B}^{\gamma}$.
Therefore
\begin{align}
I_{W_0}(g) = \min\{J_{W_0}(\ti{f}): t(H, \ti{f} ) \geq t_{\max}\} \leq I_{W_0}(f_{\max})<\infty, \nonumber
\end{align}
and thus $g \in \mathcal{B}^{\gamma} \cap \mathcal{W}_{\Omega}$. Further, $\delta_{\square}(f_*, g) =0$ implies $t(H, g) = t_{\max}$. This implies $g =1$ on the relevant blocks. Moreover, $I_{W_0}(g) \leq I_{W_0}(f_{\max})$
which 
is possible iff $g = W_0$ on the irrelevant blocks.
We therefore conclude that $g = f_{\max}$. We then have $\ti{f}_* = \ti{f}_{\max}$ and $J_{W_0}(\ti{f}_*) = J_{W_0}(\ti{f}_{\max}) = I_{W_0}(f_{\max})$. We have thus established that $\ti{f}_{\max}$ is the unique minimizer to the upper tail variational problem for $t= t_{\max}$.

Next, we have, 
\begin{align}
\mathbb{P}_{kn, W_0} \left( \delta_{\square}(f^{G_{kn}} , \ti{f}_{\max} ) \geq \epsilon  \Big| t(H, G_{kn}) \geq t_{\max}   \right) = \frac{\mathbb{P}_{kn, W_0} \left(\delta_{\square}( f^{G_{kn}}, \ti{f}_{\max} ) \geq \epsilon, t(H, G_{kn}) \geq t_{\max}\right) }{\mathbb{P}_{kn, W_0} \left( t(H, G_{kn}) \geq t_{\max} \right) } . \label{eq:cond1}
\end{align}
Note that the set $\{ \ti{f} : \delta_{\square}(\ti{f}, \ti{f}_{\max} ) \geq \epsilon, t(H, \ti{f}) \geq t_{\max} \}$ is closed, and thus Theorem \ref{ldp} implies
\begin{align*}
\limsup_{n \to \infty}& \frac{1}{(kn)^2} \log \mathbb{P}_{kn, W_0} \left(\delta_{\square}(f^{G_{kn}} , \ti{f}_{\max} ) \geq \epsilon, t(H, G_{kn})
 \geq t_{\max}\right)
 \\
 &\leq - \inf\{J_{W_0}(\tilde{f}) : \delta_{\square}(\ti{f}, \ti{f}_{\max})  \geq \epsilon, t(H, \ti{f}) \geq t_{\max} \}
 =:-C'.
\end{align*}
We conclude that for every $\eta > 0$, there exists $N(\eta)$ such that if $n \geq N(\eta)$, then
\begin{align}
&\frac{1}{(kn)^2} \log \mathbb{P}_{kn, W_0} \left(\delta_{\square}(f^{G_{kn}} , \ti{f}_{\max} ) \geq \epsilon, t(H, G_{kn}) \geq t_{\max}\right) \leq - C' + \eta \nonumber,
\end{align}
or equivalently
\begin{align}
 \mathbb{P}_{kn, W_0} \left(\delta_{\square}(f^{G_{kn}} , \ti{f}_{\max} ) \geq \epsilon, t(H, G_{kn}) \geq t_{\max}\right) \leq \exp{-(kn)^2(C' - \eta) }. \label{eq:num_bound}
\end{align}

Next, we turn to the denominator. Let $(I_i\times I_j)_{i,j\in [k]}$ be the blocks of $W_0$, and let $R \subset [0,1]^2$ denote the union of the relevant blocks.  Recall the definition of $\vartheta(\cdot)$ from \eqref{eq:k-gamma}, let $A = (A_{ij})_{i,j\in [kn]}$ denote the adjacency matrix of $G_{kn}$ and let $S$ be the set of relevant edges:
\[S = \left\{(i,j) \in [kn]^2  : I_{k\left(\frac{i}{kn}\right)} \times  I_{k\left(\frac{j}{kn}\right)} \subset R\right\}.\]
Observe that
\begin{align*}
\mathbb{P}_{kn, W_0} (t(H, G_{kn}) = t_{\max}) &= \mathbb{P}\left(\cap_{(i,j) \in S} \{ A_{ij} = 1\}  \right) \\
&= \prod_{a: I_a \times I_a \subset R} p_{aa}^{\binom{n}{2}}  \prod_{a < b : I_a \times I_b \subset R} p_{ab}^{n^2}. \nonumber \\
&=  \exp{\binom{n}{2} \sum_{a:  I_a \times I_a \subset R} \log(p_{aa}) + n^2  \sum_{a < b : I_a \times I_b \subset R} \log(p_{ab}) }\\
&=  \exp{\frac{1}{2} n^2 \sum_{a,b :  I_a \times I_b \subset R} \log(p_{ab}) - \frac{n}{2} \sum_{a:  I_a \times I_a \subset R} \log(p_{aa})  }\\
&=  \exp{-\frac{1}{2} (kn)^2 \sum_{a,b :  I_a \times I_b \subset R}\frac{1}{k^2} \log\left(\frac{1}{p_{ab}} \right) - \frac{n}{2} \sum_{a:  I_a \times I_a \subset R} \log(p_{aa})  }\\
&=  \exp{- (kn)^2 I_{W_0}(f_{\max}) \left(1+ o(1)\right) }.
\end{align*}
Recall that $I_{W_0}(f_{\max}) = J_{W_0}(\ti{f}_{\text{max}})$, so that
\begin{align}
\mathbb{P}_{kn, W_0} (t(H, G_{kn}) \geq t_{\max})  = \exp{- (kn)^2 J_{W_0}(\ti{f}_{\max}) (1+o(1)) }. \label{eq:lower_bound}
\end{align}

\noindent
Applying \eqref{eq:num_bound} and \eqref{eq:lower_bound} to \eqref{eq:cond1}, we obtain for $n \geq N(\eta)$
\begin{align*}
\mathbb{P}_{kn, W_0} \left( \delta_{\square}(f^{G_{kn}} , \ti{f}_{\max} ) \geq \epsilon  \Big| t(H, G_{kn}) \geq t_{\max}   \right) &\leq \exp{-(kn)^2(C' - \eta) + (kn)^2 J_{W_0}(\ti{f}_{\max}) (1+o(1))}\\
&= \exp{-(kn)^2\left(C' - \eta - J_{W_0}(\ti{f}_{\max}) (1+o(1)) \right)}.
\end{align*}
Recalling that $\ti{f}_{\max}$ is the unique minimizer of \eqref{eq:variational-problem-max}, the proof is complete by observing
that
\begin{align*}
C'&= \inf\{J_{W_0}(\tilde{f}) : \delta_{\square}(\ti{f}, \ti{f}_{\max})  \geq \delta, t(H, \ti{f}) \geq t_{\max} \}
\\
&> \inf\{J_{W_0}(\tilde{f}) :  t(H, \ti{f}) \geq t_{\max} \}
=
J_{W_0}(\ti{f}_{\max}).
\end{align*}

\end{proof}

\subsubsection{Elaboration on \Cref{remark:tmax}}
Let $W_0$ be a uniform $k$-block graphon, and $\tau = t(H, \cdot)$, where $H$ is a finite $d$-regular graph. By \Cref{lemma:new- holder-general}, $t(H, f) \leq t(H, f^*) \leq t_{\max}$, with equality if and only if $f^* = \mathbf{1}_{W_0>0}$, as all non-trivial blocks of $W_0$ are relevant. Thus $f^* = \mathbf{1}_{W_0>0}$ is the unique solution to $t(H,f) = t_{\max}$ in this setting.

Note that if $W_0 = f_p^{\gamma}$, and $\tau(\ti{f}) = \| f \|_{\text{op}}$,  $t_{\max}= \sqrt{\gamma(1-\gamma)}$. Further, using \Cref{lemma:norm-inequalities}, we conclude that $\| \ti{f}^{G}\|_{\text{op}} \geq t_{\max} $ implies that $\|f^{G}\|_2^2 \geq 2 \gamma (1 - \gamma)$. This is possible if and only if $\ti{f}^{G} = \ti{f}_1^{\gamma}$.
This establishes the desired claim.

\bibliography{updated-references}

\begin{thebibliography}{10}

\bibitem{augeri2018nonlinear}
Fanny Augeri.
\newblock Nonlinear large deviation bounds with applications to wigner matrices
  and sparse erd{\H{o}}s--r{\'e}nyi graphs.
\newblock {\em The Annals of Probability}, 48(5):2404--2448, 2020.

\bibitem{austin2019structure}
Tim Austin.
\newblock The structure of low-complexity {Gibbs} measures on product spaces.
\newblock {\em The Annals of Probability}, 47(6):4002--4023, 2019.

\bibitem{basak2019upper}
Anirban Basak and Riddhipratim Basu.
\newblock {Upper tail large deviations of the cycle counts in \ER graphs in the
  full localized regime}.
\newblock {\em arXiv preprint arXiv:1912.11410}, 2019.

\bibitem{bhattacharya2018upper}
Bhaswar~B Bhattacharya and Shirshendu Ganguly.
\newblock Upper tails for edge eigenvalues of random graphs.
\newblock {\em SIAM Journal on Discrete Mathematics}, 34(2):1069--1083, 2020.

\bibitem{bhattacharya2017upper}
Bhaswar~B Bhattacharya, Shirshendu Ganguly, Eyal Lubetzky, and Yufei Zhao.
\newblock Upper tails and independence polynomials in random graphs.
\newblock {\em Advances in Mathematics}, 319:313--347, 2017.

\bibitem{bhattacharya2016upper}
Bhaswar~B Bhattacharya, Shirshendu Ganguly, Xuancheng Shao, and Yufei Zhao.
\newblock Upper tails for arithmetic progressions in a random set.
\newblock {\em arXiv preprint arXiv:1605.02994}, 2016.

\bibitem{bhattacharya2019upper}
Sohom Bhattacharya and Amir Dembo.
\newblock Upper tail for homomorphism counts in constrained sparse random
  graphs.
\newblock {\em Random Structures \& Algorithms}, 59(3):315--338, 2021.

\bibitem{Borgs2008}
Christian Borgs, Jennifer~T Chayes, L{\'a}szl{\'o} Lov{\'a}sz, Vera~T S{\'o}s,
  and Katalin Vesztergombi.
\newblock Convergent sequences of dense graphs {I}: Subgraph frequencies,
  metric properties and testing.
\newblock {\em Advances in Mathematics}, 219(6):1801--1851, 2008.

\bibitem{Borgs2012}
Christian Borgs, Jennifer~T Chayes, L{\'a}szl{\'o} Lov{\'a}sz, Vera~T S{\'o}s,
  and Katalin Vesztergombi.
\newblock Convergent sequences of dense graphs {II}: Multiway cuts and
  statistical physics.
\newblock {\em Annals of Mathematics}, 179:151--219, 2012.

\bibitem{chatterjee2012missing}
Sourav Chatterjee.
\newblock The missing log in large deviations for triangle counts.
\newblock {\em Random Structures \& Algorithms}, 40(4):437--451, 2012.

\bibitem{Chatterjee2015}
Sourav Chatterjee.
\newblock Large deviations for random graphs.
\newblock {\em {\'E}cole d'{\'e}t{\'e} de Probabilit{\'e}s de Saint-Flour}, 45,
  2015.

\bibitem{chatterjee2016nonlinear}
Sourav Chatterjee and Amir Dembo.
\newblock Nonlinear large deviations.
\newblock {\em Advances in Mathematics}, 299:396--450, 2016.

\bibitem{chatterjee2010applications}
Sourav Chatterjee and Partha~S Dey.
\newblock {Applications of Stein's method for concentration inequalities}.
\newblock {\em The Annals of Probability}, 38(6):2443--2485, 2010.

\bibitem{Chatterjee2011}
Sourav Chatterjee and SR~Srinivasa Varadhan.
\newblock {The large deviation principle for the \ER random graph}.
\newblock {\em European Journal of Combinatorics}, 32(7):1000--1017, 2011.

\bibitem{cook2018large}
Nicholas Cook and Amir Dembo.
\newblock Large deviations of subgraph counts for sparse erd{\H{o}}s--r{\'e}nyi
  graphs.
\newblock {\em Advances in Mathematics}, 373:107289, 2020.

\bibitem{demarco2012upper}
B~DeMarco and J~Kahn.
\newblock Upper tails for triangles.
\newblock {\em Random Structures \& Algorithms}, 40(4):452--459, 2012.

\bibitem{demarco2012tight}
Robert DeMarco and Jeff Kahn.
\newblock Tight upper tail bounds for cliques.
\newblock {\em Random Structures \& Algorithms}, 41(4):469--487, 2012.

\bibitem{dembo2018large}
Amir Dembo and Eyal Lubetzky.
\newblock A large deviation principle for the {\ER} uniform random graph.
\newblock {\em Electronic Communications in Probability}, 23, 2018.

\bibitem{amirdembooferzeitouni2010}
Amir Dembo and Ofer Zeitouni.
\newblock {\em Large Deviations Techniques and Applications}, volume~38 of {\em
  Stochastic Modelling and Applied Probability}.
\newblock Springer-Verlag Berlin Heidelberg, second edition, 2010.

\bibitem{Dhara2019}
Souvik Dhara and Subhabrata Sen.
\newblock Large deviation for uniform graphs with given degrees.
\newblock {\em The Annals of Applied Probability}, 32(3):2327--2353, 2022.

\bibitem{durrett2019}
Rick Durrett.
\newblock {\em Probability: {Theory} and examples}, volume~49.
\newblock Cambridge University Press, 2019.

\bibitem{eldan2018gaussian}
Ronen Eldan.
\newblock Gaussian-width gradient complexity, reverse log-{Sobolev}
  inequalities and nonlinear large deviations.
\newblock {\em Geometric and Functional Analysis}, 28(6):1548--1596, 2018.

\bibitem{Finner1992}
Helmut Finner.
\newblock A generalization of {H\"older's} inequality and some probability
  inequalities.
\newblock {\em The Annals of Probability}, pages 1893--1901, 1992.

\bibitem{Grebik2021}
Jan Greb{\'\i}k and Oleg Pikhurko.
\newblock Large deviation principles for block and step graphon random graph
  models.
\newblock {\em arXiv preprint arXiv:2101.07025}, 2021.

\bibitem{harel2019upper}
Matan Harel, Frank Mousset, and Wojciech Samotij.
\newblock Upper tails via high moments and entropic stability.
\newblock {\em Duke Mathematical Journal}, 1(1):1--104, 2022.

\bibitem{jansongraphonbook}
S.~Janson.
\newblock {\em Graphons, cut norm and distance, couplings and rearrangements},
  volume~4 of {\em New York Journal of Mathematics Monographs}.
\newblock University at Albany, Albany, NY, 2013.

\bibitem{janson2004upper}
Svante Janson, Krzysztof Oleszkiewicz, and Andrzej Ruci{\'n}ski.
\newblock Upper tails for subgraph counts in random graphs.
\newblock {\em Israel Journal of Mathematics}, 142(1):61--92, 2004.

\bibitem{janson2002infamous}
Svante Janson and Andrzej Ruci{\'n}ski.
\newblock The infamous upper tail.
\newblock {\em Random Structures \& Algorithms}, 20(3):317--342, 2002.

\bibitem{janson2004deletion}
Svante Janson and Andrzej Ruci{\'n}ski.
\newblock The deletion method for upper tail estimates.
\newblock {\em Combinatorica}, 24(4):615--640, 2004.

\bibitem{kim2004divide}
Jeong~Han Kim and Van~H Vu.
\newblock Divide and conquer martingales and the number of triangles in a
  random graph.
\newblock {\em Random Structures \& Algorithms}, 24(2):166--174, 2004.

\bibitem{liu2019upper}
Yang~P Liu and Yufei Zhao.
\newblock On the upper tail problem for random hypergraphs.
\newblock {\em Random Structures \& Algorithms}, 58(2):179--220, 2021.

\bibitem{Lovasz2012}
L{\'a}szl{\'o} Lov{\'a}sz.
\newblock {\em Large networks and graph limits}, volume~60.
\newblock American Mathematical Society, 2012.

\bibitem{Lovasz2007}
L{\'a}szl{\'o} Lov{\'a}sz and Bal{\'a}zs Szegedy.
\newblock Szemer{\'e}di's lemma for the analyst.
\newblock {\em GAFA Geometric And Functional Analysis}, 17(1):252--270, 2007.

\bibitem{Lubetzky2015}
Eyal Lubetzky and Yufei Zhao.
\newblock On replica symmetry of large deviations in random graphs.
\newblock {\em Random Structures \& Algorithms}, 47(1):109--146, 2015.

\bibitem{Markering2020}
Maarten Markering.
\newblock The large deviation principle for inhomogeneous
  {Erd{\H{o}}s--R{\'e}nyi} random graphs.
\newblock {\em Journal of Theoretical Probability}, pages 1--17, 2022.

\bibitem{vsileikis2019counterexample}
Matas {\v{S}}ileikis and Lutz Warnke.
\newblock A counterexample to the {DeMarco-Kahn} upper tail conjecture.
\newblock {\em Random Structures \& Algorithms}, 55(4):775--794, 2019.

\bibitem{vu2001large}
Van~H Vu.
\newblock A large deviation result on the number of small subgraphs of a random
  graph.
\newblock {\em Combinatorics, Probability and Computing}, 10(1):79--94, 2001.

\end{thebibliography}
\bibliographystyle{plain}
\end{document}